\documentclass[a4paper,12pt]{amsart}

\usepackage{amsmath,amsthm,amssymb,stmaryrd, color} 
\usepackage{a4wide}
\usepackage{hyperref}

\theoremstyle{plain} 
\newtheorem{theorem}{Theorem}[section]
\newtheorem{lemma}[theorem]{Lemma}
\newtheorem{proposition}[theorem]{Proposition}
\newtheorem{corollary}[theorem]{Corollary} 

\theoremstyle{definition} 
\newtheorem{definition}[theorem]{Definition}
\newtheorem{remark}[theorem]{Remark}
\newtheorem{example}[theorem]{Example}
\numberwithin{equation}{section}

\usepackage{enumitem}
\setlist[itemize]{itemsep=0ex,label=--}
\setlist[enumerate]{itemsep=0ex,topsep=0ex}

\newcommand{\GG}{\mathbb{G}}
\newcommand{\HH}{\mathbb{H}}
\newcommand{\RR}{\mathbb{R}}
\newcommand{\NN}{\mathbb{N}} 
\newcommand{\ZZ}{\mathbb{Z}} 
\newcommand{\CC}{\mathbb{C}}
\newcommand{\PP}{\mathbb{P}} 
\newcommand{\XX}{\mathbb{X}} 

\newcommand{\Aa}{\mathcal{A}}
\newcommand{\Ff}{\mathcal{F}}
\newcommand{\Hh}{\mathcal{H}}
\newcommand{\Dd}{\mathcal{D}}
\newcommand{\Mm}{\mathcal{M}}
\newcommand{\Ll}{\mathcal{L}}

\DeclareSymbolFont{bbold}{U}{bbold}{m}{n}
\DeclareSymbolFontAlphabet{\mathbbold}{bbold}
\newcommand{\GGamma}{\mathbbold{\Gamma}}
 
\newcommand{\LLambda}{\mathbbold{\Lambda}}

\newcommand{\Cst}{$C^*$-\relax}
\newcommand{\qdim}{\operatorname{dim_{\text{$q$}}}}
\renewcommand{\dim}{\operatorname{dim}}
\newcommand{\Tr}{\operatorname{Tr}}
\newcommand{\qTr}{\operatorname{qTr}}
\newcommand{\Supp}{\operatorname{Supp}}
\newcommand{\Span}{\operatorname{Span}}
\newcommand{\Corep}{\operatorname{Corep}}
\newcommand{\Rep}{\operatorname{Rep}}
\newcommand{\Hom}{\operatorname{Hom}}
\newcommand{\End}{\operatorname{End}}
\newcommand{\id}{\mathord{\operatorname{id}}} 
\newcommand{\Dom}{\operatorname{Dom}}
\newcommand{\sgn}{\operatorname{sgn}}

\newcommand{\ad}{\mathrm{ad}}
\newcommand{\Ad}{\operatorname{Ad}}
\newcommand{\ev}{{\operatorname{ev}}}
\newcommand{\ts}{\textstyle}
\newcommand{\bs}{\backslash}
\newcommand{\one}{1\!\!1}

\newcommand{\Pol}{\mathcal{O}}
\newcommand{\Polc}{\mathcal{O}_c}
\newcommand{\Pold}{c_c}
\newcommand{\Tdisc}{T}
\newcommand{\dc}[1]{\llbracket #1\rrbracket}
\newcommand{\cc}{c}
\newcommand{\qc}{\mathrm{q}c}
\newcommand{\qmult}{\operatorname{qmult}}
\newcommand{\mult}{\operatorname{mult}}
\newcommand{\diag}{\operatorname{diag}}
\newcommand{\CokerX}[2]{N(#1\curvearrowright #2)}
\newcommand{\Coker}[2]{\CokerX{#1}{#1/#2}}
\newcommand{\Cokeralg}[2]{\Aa(#1\curvearrowright #1/#2)}

\title[Hecke algebras of quantum groups]{Hecke algebras and the Schlichting completion for discrete quantum groups}
\author{Adam Skalski}
\address{Adam Skalski, Institute of Mathematics of the Polish Academy of Sciences, ul.\ \'Snia\-dec\-kich 8, 00-656 Warszawa, Poland}
\email{a.skalski@impan.pl}
\author{Roland Vergnioux}
\address{Roland Vergnioux, Normandie Univ, UNICAEN, CNRS, LMNO, 14000 Caen, France} 
\email{roland.vergnioux@unicaen.fr}
\author{Christian Voigt}
\address{Christian Voigt, School of Mathematics \& Statistics,
	University of Glasgow,
	University Place,
	Glasgow G12 8QQ,
	United Kingdom} 
\email{christian.voigt@glasgow.ac.uk }

\makeatletter
\@namedef{subjclassname@2020}{%
  \textup{2020} Mathematics Subject Classification}
\makeatother

\begin{document}

\begin{abstract}
  We introduce Hecke algebras associated to discrete quantum groups with commensurated quantum subgroups. We study their modular properties and the associated
  Hecke operators. In order to investigate their analytic properties we adapt the construction of the Schlichting completion to the quantum setting, thus
  obtaining locally compact quantum groups with compact open quantum subgroups. We study in detail a class of examples arising from quantum HNN extensions.
\end{abstract}


	
\maketitle

\section{Introduction}

Hecke algebras, originally studied in the analysis of Hecke operators for
elliptic modular forms, play a prominent r\^ole in representation theory and
harmonic analysis. In applications to number theory one is typically interested 
in Hecke operators associated with arithmetic groups. Abstractly, the relevant operators can be
described starting from a discrete group $\Gamma$ together with a commensurated
subgroup, i.e.\ a subgroup $\Lambda\subset\Gamma$ such that
$\Lambda\cap \Lambda^g$ has finite index in $\Lambda$ for all $g\in\Gamma$,
where $\Lambda^g = g\Lambda g^{-1}$.  At this level of generality, Hecke
operators can be viewed as $\Gamma$-equivariant bounded operators on
$\ell^2(\Gamma/\Lambda)$ and can be described using the Hecke algebra
$\Hh(\Gamma,\Lambda)$, which is nothing but the space of functions on double
cosets $c_c(\Lambda\bs\Gamma/\Lambda)$ equipped with a suitable convolution
product.

In their seminal paper \cite{BostConnes}, Bost and Connes exhibited an
intriguing connection between Hecke algebras, number theory and noncommutative
geometry. The Hecke algebra underlying the Bost-Connes system
is part of a quantum statistical mechanical system whose equilibrium states are
intimately related to class field theory, and the time evolution of the system can
be explicitly described by means of the modular function
\begin{displaymath}
  \nabla : g\mapsto [\Lambda:\Lambda\cap\Lambda^g] / [\Lambda^g:\Lambda\cap\Lambda^g],
\end{displaymath}
comparing the number of left and right cosets in a double coset.

Hecke operators and Hecke algebras can also be defined for locally compact
groups $G$ together with compact open subgroups $H\subset G$. Moreover, both
situations are related by the Schlichting completion construction which
associates to each discrete Hecke pair $(\Gamma,\Lambda)$ in a canonical way a pair $(G,H)$
consisting of a totally disconnected locally compact group $G$ and a compact open subgroup $H\subset G$ 
such that $\Hh(\Gamma,\Lambda)\cong\Hh(G,H)$ \cite{Schlichting}. 
Analytical properties of the algebra of Hecke operators are often easier to analyze at 
the level of the Schlichting completion, see \cite{Tzanev}, \cite{Anantharaman_Approx} and
\cite{KLQ_Schlichting}.

\bigskip

The aim of this article is to extend some of this theory to the case of locally
compact quantum groups, both in the discrete and compact open settings. This 
exhibits new combinatorial behavior which is invisible in the classical case, related 
to the ``relative dimension'' constants naturally associated with quantum subgroups of 
discrete quantum groups. A major motivation for the passage to the quantum
framework, apart from producing interesting examples of von Neumann algebras, is
to create new locally compact quantum groups out of known discrete quantum
groups. For this purpose we develop a generalization of the Schlichting
completion procedure, which provides a slightly new perspective 
even for classical groups, and leads to a deeper understanding of the quantum
quotient spaces $\GGamma/\LLambda$ for discrete quantum groups. The Schlichting
completion yields algebraic quantum groups in the sense of Van Daele
\cite{VanDaele_Algebraic}, and we obtain concrete examples by considering pairs
$(\GGamma,\LLambda)$ of discrete quantum groups arising from HNN extensions.

\bigskip

Let us describe the main results obtained in the article. After collecting some
preliminaries in Section \ref{sec_quantum_groups}, we begin our analysis in the
setting of a subgroup $\LLambda$ in a discrete quantum group $\GGamma$ in
Section \ref{sec_discrete_hecke}. We give a detailed description of the
noncommutative quotient space $\GGamma/\LLambda$ and of the associated module
category, see Proposition~\ref{prp_quotient_description} and
Theorem~\ref{thm_strong_inv}. This allows us to obtain an explicit formula for the
quantum analogue $\mu$ of the counting measure on $\GGamma/\LLambda$, in terms
of the equivalence relation induced by $\LLambda$ on irreducible
corepresentations of $\GGamma$, see Definition~\ref{def_scalar}. We also give in
Proposition~\ref{prp_kappa_cat} a categorical interpretation of the constants
$\kappa$ that appear in this formula. These constants are trivial in the
classical case.

With this in place it is easy to write down the definition of the
convolution product of the Hecke algebra $\Hh(\GGamma,\LLambda)$, see
Definitions~\ref{def_convol} and~\ref{df_hecke_algebra}. We prove in
Theorem~\ref{thm_adjoint} that this algebra is canonically isomorphic to the
algebra of Hecke operators, i.e.\ $\GGamma$-equivariant linear maps on
$c_c(\GGamma/\LLambda)$. Moreover, in Theorem~\ref{thm_bounded} we give a
combinatorial characterization of the boundedness of the Hecke operators on
$\ell^2(\GGamma/\LLambda)$, in terms of the constants $\kappa$.

Next we investigate the modular properties of the canonical state on
$\Hh(\GGamma,\LLambda)$ and give an explicit formula for the corresponding
modular operator in Proposition~\ref{prop_compute_nabla} and
Theorem~\ref{thm_modularautomorphism}. This formula involves the number of left
and right cosets in double cosets, as in the classical case, but in general also the
modular structure of the discrete quantum group $\GGamma$, as well as the constants $\kappa$. 
In Paragraph~\ref{sec_HNN} we consider an explicit class of examples arising from HNN extensions.

In Section~\ref{sec_compact_open_hecke} we switch to the setting of a locally
compact quantum group $\GG$ with a compact open quantum subgroup $\HH$, working in the
framework of algebraic quantum groups. In this case it is easier to define the
Hecke algebra $\Hh(\GG,\HH)$ since compactly supported functions on
$\HH\bs\GG/\HH$ are also compactly supported on $\GG$, and one can directly use
the convolution product of the quantum algebra of functions $\Polc(\GG)$. Again,
we establish a canonical isomorphism of $\Hh(\GG,\HH)$ with the algebra of
$\GG$-equivariant maps on $\Pold(\GG/\HH)$ in Proposition~\ref{prp_endotd}. As
an example, we discuss the case of the quantum doubles $\GG = \HH\bowtie\hat\HH$ of 
a compact quantum group $\HH$: in this case $\Hh(\GG,\HH)$ identifies with the
algebra of characters of $\HH$.

In Paragraph \ref{sec_schlichting} we associate a pair $(\GG,\HH)$ to each
discrete Hecke pair $(\GGamma,\LLambda)$ by means of a quantum analogue of the
Schlichting completion. More precisely, we construct $\Polc(\GG)$ directly as a
subalgebra of $\ell^\infty(\GGamma)$ using the Hecke convolution product between
$c_c(\GGamma/\LLambda)$ and $c_c(\LLambda\bs\GGamma)$, see
Definition~\ref{def_schlichting}. This seems to be a new point of view even in
the classical case. Moreover we establish in Proposition~\ref{prp_heckeident} a
canonical identification between $\Hh(\GG,\HH)$ and $\Hh(\GGamma,\LLambda)$. It
follows in particular that Hecke operators are bounded on $\ell^2(\GG/\HH)$ as
well as on $\ell^2(\GGamma/\LLambda)$, and this yields an analytic proof of
the combinatorial property of the constants $\kappa$ mentioned above, see
Definition~\ref{def_RT} and Theorem~\ref{thm_bounded}.

Finally, in Paragraph \ref{sec_reduction} we study the notion of reduced pair
both in the discrete setting and in the compact open one, with the property 
of being reduced corresponding to faithfulness of the $\GGamma$-action on $\GGamma/\LLambda$, 
resp. of the $\GG$-action on $\GG/\HH$. We construct a reduced pair associated to an
arbitrary Hecke pair in Propositions~\ref{reductiondiscrete}
and~\ref{prop_reduction_lc}. Moreover we prove that the Schlichting completion
$\GG$ is non-discrete whenever the Hecke pair $(\GGamma,\LLambda)$ is reduced
and $\LLambda$ is infinite, see Lemma~\ref{lem_schlichting_discrete}. It follows
that the Schlichting completions of the Hecke pairs constructed in Section~\ref{sec_HNN} 
via HNN extensions yield non-discrete locally compact quantum groups, whose modular 
automorphisms can be computed by the explicit formulas of Paragraph~\ref{sec_hecke_discrete}.

We would like to thank the anonymous referee for their careful reading of our original manuscript and a 
number of valuable suggestions and comments. 

\section{(Quantum group) preliminaries}
\label{sec_quantum_groups}

In this short section we introduce general conventions, fix our notation and
offer a brief review of some definitions and facts from the theory of quantum
groups. For more details we refer the reader to the following sources:
\cite{KlimykSchmuedgen_Book}, \cite{Woronowicz_CQG}, \cite{VanDaele_Algebraic},
\cite{KustermansVaes}, \cite{NeshveyevTuset_Book}, \cite{VoigtYuncken_Book}.

Tensor products of algebras, minimal/spatial tensor products of $C^*$-algebras
and Hilbert space tensor product of spaces/operators will be usually denoted by
$\otimes$; if we want to stress that we are dealing with the algebraic tensor
product we will use the symbol $\odot$. If $\varphi$ is a linear form on an
algebra $A$ and $a\in A$ we denote $a\varphi$, $\varphi a$ the forms given by
$a\varphi(b)=\varphi(ba)$ and $\varphi a(b) = \varphi(ab)$, for all $b \in A$.

\subsection{Algebraic quantum groups}

By definition, an algebraic quantum group $\GG$ is given by a multiplier Hopf
$*$-algebra $\Polc(\GG)$ together with positive invariant functionals
\cite{VanDaele_Algebraic}. Recall from \cite{KustermansVanDaele_CstarAlgebraic}
that one can associate to $\GG$ in a canonical way a locally compact quantum
group, i.e.\ a (reduced) Hopf $C^*$-algebra $C_0(\GG)$ satisfying the axioms of
Kustermans and Vaes \cite{KustermansVaes} and containing $\Polc(\GG)$ as a dense
$*$-subalgebra. We denote by $\varphi$, $\psi$ the left and right Haar weights
of $\GG$, which are defined on $\Polc(\GG)$. We denote the dual multiplier Hopf
algebra by $\Dd(\GG)$.

Not all locally compact quantum groups arise in this way. In particular classical locally compact groups which fit into the algebraic quantum group framework
are precisely the ones admitting a compact open subgroup \cite[Section 3]{LandstadVanDaele_CompactOpen}, and this is precisely the class of groups which
naturally appears in the study of Hecke algebras.

More specifically, if $G$ is a locally compact group with a compact open
subgroup $H\subset G$, then we get an algebraic quantum group by considering
\begin{displaymath}
  \Polc(G) 
  = \Span G \cdot \Pol(H) \subset C_0(G),
\end{displaymath}
where $\Pol(H)$ is the usual space of representative functions on $H$, which
embeds canonically in $C_0(G)$ via extension by $0$, and $(g \cdot f)(x)=f(xg)$
for $g\in G$ and $f\in C_0(G)$ is the action by right translation.
The resulting multiplier Hopf $*$-algebra is independent of the choice of $H$,
and in fact uniquely determined \cite{LandstadVanDaele_CompactOpen}.

A morphism between algebraic quantum groups
from $\GG_1$ to $\GG_2$ is given by a nondegenerate $*$-homomorphism
$\pi: \Polc(\GG_2)\rightarrow\Mm(\Polc(\GG_1))$, compatible with the
comultiplications. Observe that the algebraic multiplier algebra
$\Mm(\Polc(\GG_1))$ typically contains operators which are unbounded at the
Hilbert space level. However, if $W_{\GG_2}$ denotes the multiplicative unitary associated with $\GG_2$, then $(\pi \otimes \id)(W_{\GG_2})$ is a unitary
element of $\Mm(\Polc(\GG_1) \odot \Dd(\GG_2))$, and hence it induces a bounded
(unitary) operator on $L^2(\GG_1) \otimes L^2(\GG_2)$.
From the properties of the multiplicative unitary at the algebraic level it
follows that this yields in fact a bicharacter at the $C^*$-level, and from
\cite{MeyerRoyWoronowicz_Homomorphisms} we conclude that $\pi$ extends to a
nondegenerate $*$-homomorphism $C^u_0(\GG_2) \rightarrow M(C_0^u(\GG_1))$
between the universal completions constructed in
\cite{Kustermans_UnivAlgebraic}.  That is, $\pi$ determines a morphism between
the associated locally compact quantum groups.  For simplicity we will
abbreviate $C_b(\GG_1)=M(C_0^u(\GG_1))$.

If $\GG$ is an algebraic quantum group, then an algebraic compact open quantum
subgroup $\HH\subset\GG$ is given by a non-zero central projection
$p_\HH\in \Polc(\GG)$ such that
$\Delta(p_\HH){(1\otimes p_\HH)} = {p_\HH\otimes p_\HH}$, compare \cite[Theorem
4.3, Proposition 4.4, Corollary 3.8]{KalantarKasprzakSkalski_Open}. Then
$C(\HH) = p_\HH C_0(\GG)$ is a Woronowicz-\Cst algebra and the canonical
morphism from $\HH$ to $\GG$ corresponds to the Hopf $*$-homomorphism
$\pi_\HH : C_0(\GG)\to C(\HH)$,
$f\mapsto p_\HH f$. We note that it seems unclear whether central projections of
$C_0(\GG)$ satisfying the above condition (which a priori describe all open
compact quantum subgroups of $\GG$) automatically lie in $\Polc(\GG)$.

\subsection{Discrete quantum groups}

An algebraic quantum group $\GGamma$ is called discrete if the corresponding
algebra $\Polc(\GGamma)$ is a direct sum of matrix algebras. The existence of
invariant functionals is then automatic and we have
$C_0(\GGamma) = C_0^u(\GGamma)$. In this case we use a different notation to
emphasize the analogy with the classical situation: we put
$\Polc(\GGamma) = c_c(\GGamma)$, $C_0(\GGamma) = c_0(\GGamma)$,
$C_b(\GGamma) = \ell^\infty(\GGamma)$, $\Mm(\Polc(\GGamma)) = c(\GGamma)$,
$\Dd(\GGamma) = \CC[\GGamma]$. The left, resp.\ right Haar weights of $\GGamma$
are denoted $h_L$, $h_R$ and their modular groups $\sigma^L$, $\sigma^R$; the
scaling automorphism group will be denoted $\tau$. We will also use the antipode
$S$, the unitary antipode $R$ and the co-unit $\epsilon$.

We denote by $\Corep(\GGamma)$ the category of finite dimensional normal unital
$*$-representations $\alpha : \ell^\infty(\GGamma) \to B(H_\alpha)$, equipped
with the tensor structure $\tilde{\otimes}$ coming from the coproduct:
$\alpha\tilde\otimes\beta := (\alpha\otimes\beta)\Delta$; for simplicity we will
simply write $\alpha \otimes \beta$ in what follows. We denote $\dim(\alpha)$,
resp.\ $\qdim(\alpha)$ the classical, resp.\ quantum dimension of a
corepresentation $\alpha$. The trivial corepresentation is denoted
$1 \in I(\GGamma)$. We denote $\bar\alpha$ the conjugate corepresentation of
$\alpha$, which is unique up to equivalence; by definition we have non-zero
morphisms $t_\alpha \in \Hom(1,\bar\alpha\otimes\alpha)$ which are uniquely
determined up to a scalar when $\alpha$ is irreducible and satisfy the conjugate
equation.

We choose a set $I(\GGamma)$ of representatives of irreducible objects up to
equivalence. We can then identify
$\ell^\infty(\GGamma) = \ell^\infty\text{-}\bigoplus_{\alpha\in I(\GGamma)} B(H_\alpha)$
and we have $c_c(\GGamma) = \bigoplus B(H_\alpha)$ and
$c(\GGamma) = \prod B(H_\alpha)$.  We denote by $p_\alpha \in c_c(\GGamma)$,
$\alpha \in I(\GGamma)$ the minimal central projections and write
$a_\alpha = p_\alpha a$ for $a \in c(\GGamma)$. The coproduct is determined by
the formula $(p_\beta\otimes p_\gamma)\Delta(a) v = v a_\alpha$ for any
$v \in \Hom(\alpha,\beta\otimes\gamma)$ and $a \in c_c(\GGamma)$. Let us record the
following fact which is certainly well-known to experts.

\begin{lemma}\label{lem_subobjects}
	Let $\alpha, \beta \in I(\GGamma)$.
  The tensor product corepresentation $\alpha\otimes\beta$ contains at most
  $\dim(\beta)^2$ irreducible subobjects (counted with multiplicities).
\end{lemma}

\begin{proof}
  Decompose $\alpha\otimes\beta = \bigoplus_{i=1}^p \gamma_i$ with
  $\gamma_i \in I(\GGamma)$. By Frobenius reciprocity we have
  $\alpha \subset \gamma_i\otimes\bar\beta$, hence
  $\dim(\gamma_i) \geq \dim(\alpha)/\dim(\beta)$. Then we can write
  \begin{align*}
    \dim(\alpha)\dim(\beta) = {\ts\sum} \dim(\gamma_i) 
    \geq p ~ \frac{\dim(\alpha)}{\dim(\beta)},
  \end{align*}
  which yields the result.
\end{proof}

A quantum subgroup $\LLambda$ of $\GGamma$ is given by a surjective
$*$-homomorphism $\pi : \ell^\infty(\GGamma) \to \ell^\infty(\LLambda)$ such
that $(\pi\otimes\pi)\Delta = \Delta\pi$. Due to the very special structure of
$\ell^\infty(\GGamma)$, one can identify
$\ell^\infty(\LLambda)$ with $p_\LLambda\ell^\infty(\GGamma)$ for a uniquely
determined central projection $p_\LLambda\in\ell^\infty(\GGamma)$, in such a way
that $\pi(a) = p_\LLambda a$. The coproduct of $\ell^\infty(\LLambda)$ is then
$\Delta_\LLambda(a) = (p_\LLambda\otimes p_\LLambda)\Delta(a) =
{(p_\LLambda\otimes \id)}\Delta(a) = (\id\otimes p_\LLambda)\Delta(a)$ for
$a \in \ell^\infty(\LLambda)$. Note that we have in the same way
$c_c(\LLambda) = p_\LLambda c_c(\GGamma)$,
$c(\LLambda) = p_\LLambda c(\GGamma)$. Using precomposition with $\pi$ the
category $\Corep(\LLambda)$ is fully and faithfully embedded in
$\Corep(\GGamma)$. We take for $I(\LLambda)$ the subset of $I(\GGamma)$ such
that $\alpha \in I(\LLambda)$ iff $p_\LLambda p_\alpha\neq 0$. We have
$1\in I(\LLambda)$ and if $\alpha$, $\beta \in I(\LLambda)$,
$\gamma \in I(\GGamma)$ and $\gamma \subset \alpha\otimes\beta$ then
$\gamma \in I(\LLambda)$ and $\bar\alpha$, $\bar\beta\in I(\LLambda)$.

\section{Discrete quantum Hecke pairs}
\label{sec_discrete_hecke}

In this section we define the Hecke algebra associated to a discrete quantum
group $\GGamma$ and a commensurated quantum subgroup $\LLambda$. We start by
studying the structure of the quotient spaces $\GGamma/\LLambda$,
$\LLambda\bs\GGamma$, first at the level of the set of irreducible
corepresentations $I(\GGamma)$, and then at the finer level of the quantum
algebras of functions $\ell^\infty(\GGamma/\LLambda)$,
$\ell^\infty(\LLambda\bs\GGamma)$. We obtain in particular in
Theorem~\ref{thm_strong_inv} a description of $\ell^\infty(\GGamma/\LLambda)$
using the classical quotient space $I(\GGamma)/\LLambda$.

We can then define the Hecke algebra $\Hh(\GGamma,\LLambda)$ and its convolution
product, and prove that it is represented by Hecke operators on
$c_c(\GGamma/\LLambda)$. We give a combinatorial characterization of the
$\ell^2$-boundedness of Hecke operators and describe the modular properties of
the canonical state of $\Hh(\GGamma,\LLambda)$. Finally we investigate examples
arising from quantum HNN extensions.

\subsection{Quotient spaces}\label{sec_quotient_spaces}

\subsubsection{Classical cosets}

Suppose that $\GGamma$ is a discrete quantum group with a quantum subgroup
$\LLambda$.  Associated to $\LLambda$ is the equivalence relation $\sim$ on
$I(\GGamma)$ such that $\alpha \sim \beta$ iff
$\alpha \subset \beta\otimes\gamma$ for some $\gamma \in I(\LLambda)$ iff
$\Delta(p_\alpha)(p_\beta\otimes p_\LLambda)\neq 0$, see
\cite[Lemma~2.3]{Vergnioux_Amalg}. There is corresponding left version:
$\alpha\backsim\beta$ iff $\alpha \subset \gamma\otimes\beta$ for some
$\gamma \in I(\LLambda)$ iff
$\Delta(p_\alpha)(p_\LLambda\otimes p_\beta)\neq 0$. We denote by
$I(\GGamma)/\LLambda$, $\LLambda\bs I(\GGamma)$ the corresponding quotient
spaces and by $[\alpha]$ the class of $\alpha \in I(\GGamma)$ in
$I(\GGamma)/\LLambda$ or in $\LLambda\bs I(\GGamma)$, with the context
determining the choice of left or right cosets. For
$\sigma\in I(\GGamma)/\LLambda$ or $\LLambda\bs I(\GGamma)$ we write
$p_\sigma = \sum_{\alpha\in\sigma} p_\alpha \in \ell^\infty(\GGamma)$. Note that
in both cases $p_{[1]} = p_\LLambda$, and the projections $p_\sigma$ are
pairwise orthogonal, central and sum up to $1$. Note that they need not be
minimal central projections in $\ell^\infty(\GGamma/\LLambda)$, but are finite
sums of such, see \cite[Section 5]{DCKSS_DiscreteActions} and also the
discussion below.

Recall that the quantum quotient spaces are given by the algebras
\begin{align*}
  \ell^\infty(\GGamma/\LLambda) = \{ a\in\ell^\infty(\GGamma) \mid (1\otimes
  p_\LLambda)\Delta(a) = a\otimes p_\LLambda\}, \\
  \ell^\infty(\LLambda\bs\GGamma) = \{ a\in\ell^\infty(\GGamma) \mid
  (p_\LLambda\otimes 1)\Delta(a) = p_\LLambda\otimes a\}. 
\end{align*}
One can use the same conditions to define $c(\GGamma/\LLambda)$,
$c(\LLambda\bs\GGamma)$ but one has to be a bit more careful for the spaces of
finitely supported functions. One can check that
$p_\sigma \in c(\GGamma/\LLambda)$ for any $\sigma \in I(\GGamma)/\LLambda$ and
one defines
$c_c(\GGamma/\LLambda) = \Span \{ p_\sigma c(\GGamma/\LLambda), \sigma\in
I(\GGamma)/\LLambda\}$ (and similarly for the left versions). One can show that
$p_\sigma c(\GGamma/\LLambda) = p_\sigma \ell^\infty(\GGamma/\LLambda)$ for any
$\sigma \in I(\GGamma)/\LLambda$ \cite[Lemma~3.3]{VergniouxVoigt}. It is easy to
check that the coproduct $\Delta$ restricts to von Neumann, resp.\ algebraic left
coactions
$\Delta : \ell^\infty(\GGamma/\LLambda) \to \ell^\infty(\GGamma) \bar\otimes
\ell^\infty(\GGamma/\LLambda)$, resp.\
$\Delta : c_c(\GGamma/\LLambda) \to \Mm(c_c(\GGamma) \otimes
c_c(\GGamma/\LLambda))$, and similarly to right coactions
$\Delta : \ell^\infty(\LLambda\bs \GGamma) \to \ell^\infty(\LLambda\bs\GGamma)
\bar\otimes \ell^\infty(\GGamma)$ resp.
$\Delta : c_c(\LLambda\bs\GGamma) \to \Mm(c_c(\LLambda\bs\GGamma) \otimes
c_c(\GGamma))$.  We have e.g.
$(c_c(\GGamma)\otimes 1)\Delta(c_c(\GGamma/\LLambda)) = c_c(\GGamma)\otimes
c_c(\GGamma/\LLambda)$.

The next Lemma works in general for open quantum subgroups of locally compact
quantum groups, see \cite[Lemma 3.1, Corollary
3.9]{KalantarKasprzakSkalski_Open}. The fact that the unitary antipode exchanges
the left and right von Neumann algebraic homogeneous spaces for any closed
quantum subgroup can be found for example in
\cite{KasprzakSoltan_Projection}. We include a simple proof for the discrete
case.

\begin{lemma} \label{lemma_quotient_antipode} Let $\LLambda$ be a quantum
  subgroup of $\GGamma$. We have
  $R(c(\GGamma/\LLambda)) = c(\LLambda\bs \GGamma)$ and
  $S(c(\GGamma/\LLambda)) = c(\LLambda\bs\GGamma)$. The groups $\tau$,
  $\sigma^L$, $\sigma^R$ stabilize $c(\GGamma/\LLambda)$ and
  $c(\LLambda\bs\GGamma)$.
\end{lemma}

\begin{proof}
  For any $\alpha \in I(\GGamma)$ and $t \in \mathbb{R}$ we have
  $S(p_\alpha) = R(p_\alpha) = p_{\bar\alpha}$ and
  $\tau_t(p_\alpha) = \sigma_t^R(p_\alpha) = \sigma_t^L(p_\alpha) = p_\alpha$
  hence
  $S(p_\LLambda) = R(p_\LLambda) = \tau_t(p_\LLambda) = \sigma_t^R(p_\LLambda) =
  \sigma_t^R(p_\LLambda)$.  For $a \in c(\GGamma/\LLambda)$ we have
  $(p_\LLambda\otimes 1)\Delta(S(a)) = \sigma(S\otimes S)[\Delta(a)(1\otimes
  S^{-1}(p_\LLambda))] = \sigma(S\otimes S)[\Delta(a)(1\otimes p_\LLambda)] =
  \sigma(S\otimes S)(a\otimes p_\LLambda) = p_\LLambda \otimes S(a)$ (with
  $\sigma$ denoting the tensor flip), hence $S(a) \in
  c(\LLambda\bs\GGamma)$. The same holds for $R(a)$ since we also have 
  $\Delta(R(a)) = \sigma(R\otimes R)\Delta(a)$. On the other hand for any
  $t \in \mathbb{R}$ we have
  $(1\otimes p_\LLambda)\Delta(\tau_t(a)) = (\tau_t\otimes\tau_t)[(1\otimes
  p_\LLambda)\Delta(a)] = \tau_t(a)\otimes p_\LLambda$ hence
  $\tau_t(a) \in c(\GGamma/\LLambda)$, and similarly on the left. Finally, note
  that in the discrete case we have $\tau_t = \sigma^L_t = \sigma^R_{-t}$.
\end{proof}

\bigskip

One can proceed similarly for double cosets. More precisely we denote
$c(\LLambda\bs\GGamma/\LLambda) = c(\GGamma/\LLambda) \cap
c(\LLambda\bs\GGamma)$. We consider the equivalence relation $\approx$ on
$I(\GGamma)$ generated by $\sim$ and $\backsim$, and we have in fact
$\alpha\approx\beta$ iff there exist $\delta$, $\gamma \in I(\LLambda)$ such
that $\beta \subset \delta\otimes\alpha\otimes\gamma$ --- indeed the set of such
$\beta$'s is closed under $\sim$ and $\backsim$, and decomposing
$\alpha\otimes\gamma$ into $\alpha'$s one sees that
$\beta\backsim\alpha'\sim\alpha$ for one of these $\alpha'$s. We denote by
$\LLambda\bs I(\GGamma)/\LLambda$ the corresponding quotient space and write
$\dc \alpha$ for the corresponding class of $\alpha \in I(\GGamma)$. For
$\sigma \in \LLambda\bs I(\GGamma)/\LLambda$ we denote
$p_\sigma = \sum_{\alpha\in\sigma} p_\alpha \in \ell^\infty(\GGamma)$, obtaining
again a family of central, pairwise orthogonal projections summing up to $1$. We
clearly have $p_\sigma \in c(\LLambda\bs\GGamma/\LLambda)$ and we denote
$c_c(\LLambda\bs\GGamma/\LLambda) = \Span \{p_\sigma
c(\LLambda\bs\GGamma/\LLambda), \sigma \in \LLambda\bs I(\GGamma)/\LLambda\}
\subset c(\LLambda\bs\GGamma/\LLambda)$.

For $\tau \in \LLambda\bs I(\GGamma)/\LLambda$ we have
$p_\tau = \sum \{p_\sigma, \sigma \in I(\GGamma)/\LLambda, \sigma\subset\tau\} =
\sum \{p_\sigma, \sigma \in \LLambda\bs I(\GGamma), \sigma\subset\tau\}$.  We
denote
$R(\tau) = \# \{\sigma \in I(\GGamma)/\LLambda, \sigma\subset\tau\} \in \NN \cup
\{+\infty\}$ and
$L(\tau) = \#\{\sigma \in \LLambda\bs I(\GGamma), \sigma\subset\tau\}$. We have
$c_c(\GGamma/\LLambda)\cap c(\LLambda\bs\GGamma) \subset
c_c(\LLambda\bs\GGamma/\LLambda)$, and similarly
$c(\GGamma/\LLambda)\cap c_c(\LLambda\bs\GGamma) \subset
c_c(\LLambda\bs\GGamma/\LLambda)$.  Already in the classical case these
inclusions can be strict.

\begin{proposition}
  Let $\LLambda$ be a quantum subgroup of $\GGamma$. The subset
  $I(\GGamma') = \{\alpha \in I(\GGamma) \mid L(\dc\alpha), R(\dc\alpha) <
  +\infty\}$ defines an intermediate quantum subgroup
  $\LLambda \subset \GGamma' \subset \GGamma$ such that
  $c_c(\LLambda\bs\GGamma'/\LLambda) = c_c(\GGamma/\LLambda)\cap
  c_c(\LLambda\bs\GGamma)$.
\end{proposition}

\begin{proof}
  It suffices to show that $I(\LLambda) \subset I(\GGamma')$,
  $I(\GGamma')\otimes I(\GGamma') \subset \mathbb{Z} I(\GGamma')$ and
  $\overline{I(\GGamma')} \subset I(\GGamma')$, see
  e.g. \cite[Section~2]{Vergnioux_Amalg}. Since $\dc 1 = I(\LLambda) = [1]$ we
  have $R(\dc 1) = L(\dc 1) = 1$ and $\dc 1 = I(\LLambda) \subset I(\GGamma')$.
  Since $\alpha \sim \beta$ $\iff$ $\bar\alpha \backsim\bar\beta$, we have
  $L(\dc\alpha) = R(\dc{\bar\alpha})$ hence $\bar\alpha \in I(\GGamma')$ if
  $\alpha \in I(\GGamma')$. Let $\alpha$, $\beta \in I(\GGamma')$ and let
  $\delta \in I(\GGamma)$, $\delta \subset \alpha\otimes\beta$. Decompose into
  finite unions of right cosets $\dc\alpha = \bigsqcup [\alpha_i]$,
  $\dc\beta = \bigsqcup [\beta_j]$, and consider the finite set of all the
  elements $\delta_k \in I(\GGamma)$ such that
  $\delta_k\subset \alpha_i\otimes\beta_j$ for some $i$, $j$. Take $\lambda$,
  $\mu\in I(\LLambda)$ and $\gamma \subset \lambda\otimes\delta\otimes\mu$. Then
  we have $\gamma \subset \lambda\otimes\alpha\otimes\beta\otimes\mu$.
  Decomposing $\lambda \otimes\alpha$ into $\alpha'$s we have by irreducibility
  $\gamma \subset \alpha'\otimes\beta\otimes\mu$ for some
  $\alpha'\in I(\GGamma)$ and since $\alpha' \in \dc\alpha$ we have
  $\alpha' \subset \alpha_i\otimes\lambda'$ for some $i$ and some
  $\lambda'\in I(\LLambda)$. Then
  $\gamma \subset \alpha_i\otimes\lambda'\otimes\beta\otimes \mu$. Proceeding
  similarly with $\lambda'\otimes\beta$ we find $j$ and
  $\lambda'' \in I(\LLambda)$ such that
  $\gamma \subset \alpha_i\otimes\beta_j\otimes \lambda''\otimes\mu$. It follows
  that $\gamma \in [\delta_k]$ for some $k$. As a result $\dc\delta$ is covered
  by a finite number of right cosets. Applying this to $\bar\delta$ we see that
  $\delta \in I(\GGamma')$.

  By definition, for $\tau \in \LLambda\bs\GGamma'/\LLambda$ we can write
  $p_\tau$ as a finite sum $p_\tau = \sum p_{\sigma_i}$ with
  $\sigma_i \in I(\GGamma)/\LLambda$. It follows that
  $c_c(\LLambda\bs\GGamma'/\LLambda) \subset c_c(\GGamma/\LLambda)$. Similarly
  $c_c(\LLambda\bs\GGamma'/\LLambda) \subset c_c(\LLambda\bs\GGamma)$.
  Conversely, let $a \in c_c(\GGamma/\LLambda)\cap c_c(\LLambda\bs\GGamma)$ and
  consider $\alpha \in I(\GGamma)$ such that $p_\alpha a \neq 0$. Then we claim
  that $p_\gamma a \neq 0$ for all $\gamma \in \dc\alpha$. Indeed we have
  $p_{[\alpha]} a \neq 0$ and since $a \in c(\GGamma/\LLambda)$ Lemma~3.3 in
  \cite{VergniouxVoigt} shows that $p_{\beta} a \neq 0$ for all
  $\beta \subset \alpha\otimes\lambda$, $\lambda \in I(\LLambda)$. Using now the
  fact that $a \in c(\LLambda\bs\GGamma)$ and again
  \cite[Lemma~3.3]{VergniouxVoigt} we have $p_\gamma a \neq 0$ for all
  $\gamma \subset \mu\otimes\beta$, $\mu \in I(\LLambda)$.  In particular we
  have $p_\sigma a \neq 0$ for all $\sigma \in I(\GGamma)/\LLambda$,
  $\sigma \subset \dc\alpha$. But since $a \in c_c(\GGamma/\LLambda)$, there is
  at most a finite number of $\sigma$'s such that $p_\sigma a\neq 0$. Hence
  $R(\dc\alpha) < +\infty$. Similarly $L(\dc\alpha)<+\infty$. As a result
  $\alpha \in I(\GGamma')$ and we can conclude that $a \in c(\GGamma')$.
\end{proof}

\begin{definition} \label{def_commensurator} Let $\LLambda$ be a quantum
  subgroup of $\GGamma$. The quantum subgroup
  $\LLambda \subset \GGamma' \subset \GGamma$ of the previous proposition is
  called the {\em commensurator} of $\LLambda$ in $\GGamma$. We say that
  $(\GGamma,\LLambda)$ is a {\em Hecke pair}\/ (or that $\LLambda$ is almost
  normal in $\GGamma$) if $R(\tau)$, $L(\tau) < +\infty$ for any
  $\tau\in \LLambda\bs I(\GGamma)/\LLambda$, or equivalently, if
  $c_c(\GGamma/\LLambda) \cap c_c(\LLambda\bs\GGamma) =
  c_c(\LLambda\bs\GGamma/\LLambda)$, i.e.\ $\GGamma' = \GGamma$.
\end{definition}

\subsubsection{Quantum cosets}\label{sec_mod_cat}

We give now a more precise description of the quantum quotient space algebra
$\ell^\infty(\GGamma/\LLambda)$ and of the corresponding
$\Corep(\GGamma)$-module-category.

For every von Neumann subalgebra $M \subset \ell^\infty(\GGamma)$ we have a
restriction functor from $\Corep(\GGamma) = \Rep (\ell^\infty(\GGamma))$ to the
category $\Rep(M)$ of finite dimensional normal $*$-repre\-sentations of $M$ which
factors the respective forgetful functors to the category of finite
dimensional Hilbert spaces. If $M$ is left invariant, i.e.
$\Delta(M) \subset \ell^\infty(\GGamma)\bar\otimes M$, then $\Rep(M)$ is
naturally equipped with the structure of a left $\Corep(\GGamma)$-module-\Cst
category, by considering $\alpha\otimes\pi := (\alpha\otimes\pi)\Delta$ and the
Hilbert space tensor product of morphisms.

When $M = \ell^\infty(\GGamma/\LLambda)$ this module category is equivalent to
the one naturally associated with the $\hat\GGamma$-\Cst algebra
$C^*_r(\LLambda)$, compare \cite[Theorem 6.4]{DeCommerYamashita}. We
write $\Rep(M) = \Corep(\GGamma/\LLambda)$ in this case and choose a set
$I(\GGamma/\LLambda)$ of representatives of irreducible objects in this category up to
equivalence, not to be confused with $I(\GGamma)/\LLambda$.  We denote
$\Hom_{\GGamma/\LLambda}$ the corresponding morphism spaces, and note that
we have a restriction functor from $\Corep(\GGamma)$ to
$\Corep(\GGamma/\LLambda)$. Moreover, to describe $\Corep(\GGamma/\LLambda)$ is
suffices to describe its full subcategory with objects from $\Corep(\GGamma)$,
i.e.\ the spaces $\Hom_{\GGamma/\LLambda}(\alpha,\beta)$ for $\alpha$,
$\beta\in I(\GGamma)$ --- one can then recover $\Corep(\GGamma/\LLambda)$ via
idempotent completion.

The next proposition is a variant and improvement of
\cite[Lemma~3.3]{VergniouxVoigt} and \cite[Theorem 5.2, Theorem
5.6]{DCKSS_DiscreteActions}. It gives a concrete description of
$\ell^\infty(\GGamma/\LLambda)$ and $\Corep(\GGamma/\LLambda)$. Recall 
that $\Corep(\LLambda)$ is faithfully and fully embedded in $\Corep(\GGamma)$.

\begin{definition} \label{def_proj}
  For $v\in\Corep(\GGamma)= \Rep(\ell^\infty(\GGamma))$ we denote by
  $v_\LLambda$ the largest subobject of $v$ belonging to $\Corep(\LLambda)$,
  given by the projection $v(p_\LLambda) \in B(H_v)$.
\end{definition}

Let us apply Definition \ref{def_proj} to the space $B(H_\alpha,H_\beta)$, viewed as a corepresentation of $\GGamma$
via the identification with $H_{\bar\alpha}\otimes H_\beta$ given by
$S \mapsto (\id\otimes S)t_\alpha$. By Frobenius reciprocity we then have
\begin{align*}
  B(H_\alpha,H_\beta)_\LLambda &= \Span\{ v(1\otimes\eta) \mid \lambda \in \Corep(\LLambda),
                                 v\in\Hom(\alpha\otimes\lambda,\beta), \eta \in H_\lambda\} 
\end{align*}
In the next proposition we use the commutant of
$B(H_\alpha, H_\alpha)_\LLambda=B(H_\alpha)_\LLambda$ inside $B(H_\alpha)$:
\begin{align*}
  B(H_\alpha)_\LLambda'&=\{b\in B(H_\alpha) \mid bf=fb \textup{ for all } f \in B(H_\alpha, H_\alpha)_\LLambda\} \\
                       &=\{b\in B(H_\alpha) \mid (b\otimes\id)w=wb \textup{ for all }\lambda \in I(\LLambda), w \in \Hom(\alpha, \alpha \otimes\lambda)\}.
\end{align*}

Recall that for $a\in \ell^\infty(\GGamma)$ and $\alpha\in I(\GGamma)$ we denote by $a_\alpha \in B(H_\alpha)$ the component $p_\alpha a$ of $a$. 

\begin{proposition} \label{prp_quotient_description} Let $\LLambda$ be a quantum
  subgroup of $\GGamma$ and $\alpha$, $\beta\in I(\GGamma)$. The map
  $(a \mapsto a_\alpha)$ is an injective $*$-homomorphism from
  $p_{[\alpha]} \ell^\infty(\GGamma/\LLambda)$ to $B(H_\alpha)$, with image
  $p_{\alpha}\ell^\infty(\GGamma/\LLambda) = B(H_\alpha)_\LLambda'$. More generally we
  have $\Hom_{\GGamma/\LLambda}(\alpha,\beta) = B(H_\alpha,
  H_\beta)_\LLambda$.
\end{proposition}

\begin{proof}
  Let $a \in \ell^\infty(\GGamma/\LLambda)$. We have then
  $(p_\alpha\otimes p_\LLambda)\Delta(a) = a_\alpha\otimes p_\LLambda$. . Since
  $(p_\alpha\otimes p_\LLambda)\Delta$ is non-zero, hence also injective on any
  matrix block $p_\beta \ell^\infty(\GGamma)$ with
  $\beta\subset \alpha\otimes \lambda$, $\lambda \in I(\LLambda)$, we see that
  $a_\alpha=0$ $\Rightarrow$ $p_{[\alpha]}a=0$. Denote $b = a_\alpha$, and let
  $v \in \Hom(\alpha,\alpha\otimes\lambda)$ with $\lambda \in I(\LLambda)$. We
  have $b\otimes p_\lambda = (p_\alpha\otimes p_\lambda)\Delta(a)$ hence
  $(b\otimes \id)v = (p_\alpha\otimes p_\lambda)\Delta(a)v = v p_\alpha a = vb$.

  Conversely, start from an element $b \in B(H_\alpha)'_\LLambda$. Then, for any
  $\lambda$, $\mu \in I(\LLambda)$ and any
  $v \in \Hom(\alpha,\alpha\otimes\lambda\otimes\mu)$ we have
  $vb = (b\otimes\id\otimes\id)v$ --- it suffices to decompose
  $\lambda\otimes\mu$ into irreducibles, which are still in $I(\LLambda)$. Even
  more, for any $v \in \Hom(\alpha\otimes\lambda,\alpha\otimes\mu)$ we have
  $v (b\otimes\id) = (b\otimes\id)v$ --- apply the previous property to
  $w = (v\otimes \id) (\id\otimes t_\lambda)\in
  \Hom(\alpha,\alpha\otimes\mu\otimes\bar\lambda)$ and the conjugate equation.

  Now, for any $\beta \in [\alpha]$, choose $\lambda_\beta\in I(\LLambda)$ and
  $v_\beta \in \Hom(\beta,\alpha\otimes\lambda_\beta)$ isometric --- with
  $\lambda_\alpha = 1$ and $v_\alpha = \id$. Define $a_\beta\in B(H_\beta)$ by
  putting $ a_\beta = v_\beta^*(b\otimes\id)v_\beta$.  Take $\beta$,
  $\beta' \in [\alpha]$, $\mu\in I(\LLambda)$ and
  $w \in \Hom(\beta,\beta'\otimes\mu)$. Then we have, applying the identity
  $(b\otimes\id)u = u(b\otimes\id)$ to
  $u = (v_{\beta'}\otimes\id)w v_\beta^* \in
  \Hom(\alpha\otimes\lambda_\beta,\alpha\otimes\lambda_{\beta' }\otimes\mu)$:
  \begin{align*}
    w a_\beta &= (v_{\beta'}^*\otimes\id)(v_{\beta'}\otimes\id)w v_\beta^* (b\otimes\id)v_\beta \\
              &= (v_{\beta'}^*\otimes\id) (b\otimes\id\otimes\id) (v_{\beta'}\otimes\id)w v_\beta^* v_\beta
                = (a_{\beta'}\otimes\id) w.
  \end{align*} 
  This shows that
  $(p_{\beta'}\otimes p_\mu)\Delta(a_\beta) = a_{\beta'}\otimes p_\mu$.  Putting
  all $a_\beta$ together we obtain $a \in p_{[\alpha]}\ell^\infty(\GGamma)$ such
  that $(\id\otimes p_\LLambda)\Delta(a) = a\otimes p_\LLambda$, i.e.
  $a \in p_{[\alpha]}\ell^\infty(\GGamma/\LLambda)$. Moreover by our choice of
  $v_\alpha$ we have $a_\alpha = b$.
\end{proof}

\begin{remark}
  In particular $p_{[\alpha]}\ell^\infty(\GGamma/\LLambda)$ is a finite
  dimensional \Cst algebra for all $[\alpha] \in I(\GGamma)/\LLambda$, hence
  $\ell^\infty(\GGamma/\LLambda)$ is a von Neumann direct product of matrix
  algebras. The inclusion
  $p_\alpha \ell^\infty(\GGamma/\LLambda) \subset B(H_\alpha)$ can be strict,
  even for all elements $\alpha$ in a given class
  $\sigma \in I(\GGamma)/\LLambda$ --- e.g. we always have
  $p_{[1]}\ell^\infty(\GGamma/\LLambda) = \CC$, and for the dual of $SO(3)$ seen
  as a subgroup of the dual of $SU(2)$ we have
  $p_\sigma\ell^\infty(\GGamma/\LLambda) = \CC$ for both classes
  $\sigma \in I(\GGamma)/\LLambda$. It can also happen, e.g. for
  $\GGamma_1\subset \GGamma_1\times\GGamma_2$, that
  $p_\alpha\ell^\infty(\GGamma/\LLambda)$ is a proper, non-trivial sub-\Cst
  algebra of $B(H_\alpha)$.
  
  Further, recall that \cite[Theorem 5.2, Theorem 5.6]{DCKSS_DiscreteActions}
  show that each $[\alpha]$ corresponds to a certain equivalence class of
  minimal central projections in $\ell^\infty(\GGamma/\LLambda)$, determined by
  the left adjoint action of $\widehat{\GGamma}$, and $p_{[\alpha]}$ is equal to
  the sum of the projections in the aforementioned equivalence class. Thus the
  fact that each $p_\alpha\ell^\infty(\GGamma/\LLambda)$ is simple is equivalent
  to the equivalence relation above being trivial. This need not be the case, as
  already classical Clifford theory shows.
\end{remark}

One can reformulate the Hecke condition of Definition~\ref{def_commensurator}
using the left action of $\LLambda$ on $\ell^\infty(\GGamma/\LLambda)$ given by
$\alpha(a) = (p_\LLambda\otimes 1)\Delta(a) \in \ell^\infty(\LLambda)\bar\otimes
\ell^\infty(\GGamma/\LLambda)$ for $a \in
\ell^\infty(\GGamma/\LLambda)$. Following \cite[Section
4]{DKSS_ClosedSubgroups}, define an equivalence relation on
$I(\GGamma/\LLambda)$ by putting $i\equiv_\LLambda j$ if
$\Delta(q_i)(p_\LLambda\otimes q_j) \neq 0$, where $q_i$, $q_j$ are the minimal
central projections in $\ell^\infty(\GGamma/\LLambda)$ corresponding to $i$ and
$j$ respectively. Equivalently, $i\equiv_\LLambda j$ iff
$j\subset \lambda\otimes i$ for some $\lambda\in I(\LLambda)$ with respect to
the module-category structure mentioned previously. Then the Hecke condition is
satisfied iff the action of $\LLambda$ on $\GGamma/\LLambda$ ``has finite
orbits'', as stated precisely in the next proposition.

\begin{proposition} \label{prop_hecke_orbits} The pair $(\GGamma,\LLambda)$ is a
  Hecke pair iff the equivalence relation $\equiv_\LLambda$ on
  $I(\GGamma/\LLambda)$ has finite classes.
\end{proposition}

\begin{proof}
  Denote by $q_i\in\ell^\infty(\GGamma/\LLambda)$ the minimal central projection
  associated with $i\in I(\GGamma/\LLambda)$. For any
  $[\alpha] \in I(\GGamma)/\LLambda$ we have a finite set
  $I([\alpha]) \subset I(\GGamma/\LLambda)$ such that
  $p_{[\alpha]} = \sum_{i\in I([\alpha])} q_i$. This implies that the action of
  $\LLambda$ on $\ell^{\infty}(\GGamma/\LLambda)$ has finite orbits if and only
  if for any $[\alpha] \in I(\GGamma)/\LLambda$ there exist only finitely many
  $[\beta] \in I(\GGamma)/\LLambda$ such that
  $\Delta(p_{[\alpha]})(p_\LLambda \otimes p_{[\beta]})\neq 0$. This means that
  we can find $\alpha'\in[\alpha]$, $\beta'\in[\beta]$ such that
  $\alpha'\backsim\beta'$, in other words $\dc\alpha=\dc\beta$. This ends the
  proof.
\end{proof}

We will now represent invariant functions using $\Delta(p_\LLambda)$, see
Theorem~\ref{thm_strong_inv}, which is essentially a consequence of
Proposition~\ref{prp_quotient_description} and ``strong invariance'' of the Haar
weights (\cite[Proposition 5.24]{KustermansVaes}). The quantities $\kappa$ below
will play a fundamental role in the sequel. Recall that for
$v \in \Corep(\GGamma)$ we denote by $v_\LLambda$ the sum of all irreducible
subobjects of $v$ equivalent to an element of $I(\LLambda)$.

\begin{definition} \label{def_kappa} For $\alpha$, $\beta \in I(\GGamma)$ we put
  $\kappa_{\alpha,\beta} = \qdim (\alpha\otimes\beta)_\LLambda$ and
  $\kappa_\alpha = \kappa_{\bar\alpha,\alpha}$.
\end{definition}

We first make the connection between the constants $\kappa_{\alpha,\beta}$ and
the Hopf algebra structure of $\ell^\infty(\GGamma)$. Recall that we denote
$a\varphi = \varphi(\,\cdot\, a)$, $\varphi a = \varphi(a \,\cdot\,)$ if
$\varphi$ is a linear form on an algebra and $a$ an element of this algebra.

\begin{lemma} \label{lem_kappa} For every $\alpha$, $\beta \in I(\GGamma)$ we
  have
  \begin{displaymath}
    (h_Rp_\LLambda\otimes p_\beta)\Delta(p_\alpha) = \frac{\qdim(\alpha)}{\qdim(\beta)} ~
    \kappa_{\alpha,\bar\beta} p_\beta.
  \end{displaymath}
  In particular
  $(h_Rp_\LLambda\otimes p_\alpha)\Delta(p_\alpha) =
  \kappa_{\bar\alpha}p_\alpha$.
\end{lemma}

\begin{proof}
  We first show that for any $\alpha$, $\beta$, $\lambda \in I(\GGamma)$ the
  linear map $(h_Rp_\lambda\otimes p_\beta)\Delta(p_\alpha) \in B(H_\beta)$ is
  an intertwiner, hence a scalar.  Indeed, we have
  $h_R(a) = \qdim(\alpha) t_\alpha^*(1\otimes a)t_\alpha$ for any
  $a \in p_\alpha c(\GGamma) = B(H_\alpha)$, where
  $t_\alpha \in \Hom(1,\bar\alpha\otimes\alpha)$ is such that
  $\|t_\alpha\|^2 = \qdim(\alpha)$. Then, if $(v_i)_i$ is an ONB of
  $\Hom(\alpha, \lambda\otimes\beta)$ and
  $P_\alpha^{\lambda,\beta} = \sum v_iv_i^*\in B(H_\lambda \otimes H_\beta)$
  denotes the orthogonal projection onto the $\alpha$-isotypic component of
  $\lambda\otimes\beta$, we have
  \begin{align*}
    (h_Rp_\lambda\otimes p_\beta)\Delta(p_\alpha) &= 
                                                    \qdim(\lambda) ~ {\ts \sum_i}
                                                    (t_\lambda^* \otimes\id)(\id\otimes v_i p_\alpha v_i^*) 
                                                    (t_\lambda \otimes\id) \\
                                                  &= \qdim(\lambda) ~ (t_\lambda^* \otimes\id)(\id\otimes P_\alpha^{\lambda,\beta}) 
                                                    (t_\lambda \otimes\id) \in \Hom(\beta,\beta).
  \end{align*}
  Hence there is a scalar $\kappa_{\alpha,\bar\beta}^{\lambda}\geq 0$ such that
  $(h_Rp_\lambda\otimes p_\beta)\Delta(p_\alpha) =
  \kappa_{\alpha,\bar\beta}^{\lambda} p_\beta$.  This scalar can be computed by
  evaluating both sides against $h_R$. We obtain, after dividing both sides
  by $\qdim(\beta)$:
  \begin{align*}
    \qdim(\beta) \times \kappa_{\alpha,\bar\beta}^{\lambda} &= 
                                                              \qdim(\lambda) ~ t_\beta^* (\id\otimes t_\lambda^* \otimes\id)
                                                              (\id\otimes\id\otimes P_\alpha^{\lambda,\beta}) 
                                                              (\id\otimes t_\lambda \otimes\id)t_\beta \\
                                                            &= \qdim(\lambda) ~ t_{\lambda\otimes\beta}^*(\id_{\bar\beta\otimes\bar\lambda}
                                                              \otimes P_\alpha^{\lambda,\beta}) t_{\lambda\otimes\beta} 
                                                              = \qdim(\lambda) ~ {\ts\sum_i} 
                                                              t_{\lambda\otimes\beta}^*(\id_{\bar\beta\otimes\bar\lambda}
                                                              \otimes v_iv_i^*) t_{\lambda\otimes\beta} \\
                                                            &= \qdim(\lambda)  \cc_\alpha^{\lambda,\beta} t_\alpha^*t_\alpha = 
                                                              \cc_\alpha^{\lambda,\beta} \qdim(\lambda) \qdim(\alpha),
  \end{align*}
  where $\cc_\alpha^{\lambda,\beta}=\dim(\Hom(\alpha,\lambda\otimes\beta))$. Note that an analogous formula appears already in work of Izumi, see the remark before 
	Corollary 3.7 in \cite{Izumi}. 
	
	By Frobenius reciprocity we also have 
  $\cc_\alpha^{\lambda,\beta} = \cc_\lambda^{\alpha,\bar\beta}$. Summing over
  $\lambda \in I(\LLambda)$ we get
  \begin{align*}
    \qdim(\beta) \times (h_Rp_\LLambda\otimes p_\beta)\Delta(p_\alpha) &= \qdim(\alpha) \left(
      \ts{\sum_{\lambda\in I(\LLambda)}} \cc_\lambda^{\alpha,\bar\beta} \qdim (\lambda) \right)p_\beta \\
    &= \qdim(\alpha) \qdim (\alpha\otimes\bar\beta)_\LLambda p_\beta,
  \end{align*}
  which yields the claim.
\end{proof}

Theorem~\ref{thm_strong_inv} below is a very useful tool for the study of Hecke
algebras associated to discrete quantum groups. It is merely a materialization
of Proposition~\ref{prp_quotient_description}, which says that one can recover
$a \in p_{[\alpha]} c(\GGamma/\LLambda)$ from $a_\alpha$, and it is a simple
consequence of the so-called strong invariance properties of the Haar
weights. It is well-known in the case when $\LLambda = \{e\}$ --- then
$p_\LLambda$ is the support of the co-unit and $\kappa_\alpha = 1$ for every
$\alpha$. It has several corollaries important for what follows.

\begin{theorem} \label{thm_strong_inv} For any $a \in c_c(\GGamma/\LLambda)$,
  $b \in c_c(\LLambda\bs\GGamma)$ and any choices of representatives
  $\alpha \in [\alpha]$, $\beta\in[\beta]$ we have
  \begin{align}
    \label{eq_strong_inv_right}
    a &= \sum_{[\alpha]\in I(\GGamma)/\LLambda} \kappa_\alpha^{-1} 
        (S^{-1}(a_\alpha)h_R\otimes\id) \Delta(p_\LLambda)
        ~= \sum_{[\alpha]\in I(\GGamma)/\LLambda} \kappa_\alpha^{-1} 
        (h_RS(a_\alpha)\otimes\id) \Delta(p_\LLambda) \\
    \label{eq_strong_inv_left}
    b &= \sum_{[\beta]\in \LLambda\bs I(\GGamma)} \kappa_{\bar\beta}^{-1} 
        (\id\otimes h_LS^{-1}(b_\beta)) \Delta(p_\LLambda)
        ~= \sum_{[\beta]\in \LLambda\bs I(\GGamma)} \kappa_{\bar\beta}^{-1} 
        (\id\otimes S(b_\beta)h_L) \Delta(p_\LLambda)
  \end{align}
\end{theorem}

\begin{proof}
  We start with $a \in c_c(\LLambda\bs\GGamma)$, $\alpha \in I(\GGamma)$ and
  $\lambda \in I(\LLambda)$.  By strong right invariance we have
  $(a_\alpha h_R\otimes\id)\Delta(p_\lambda) = (h_R\otimes
  S)((p_\lambda\otimes\id)\Delta(a_\alpha))$ hence
  $(a_\alpha h_R\otimes p_{\bar\alpha})\Delta(p_\lambda) = {(h_R\otimes
  S)}$ $((p_\lambda\otimes p_{\alpha})\Delta(a_\alpha))$.  By left invariance of $a$
  we have
  $(p_\lambda\otimes p_{\alpha})\Delta(a_\alpha) = (1\otimes
  a_\alpha)(p_\lambda\otimes p_{\alpha})\Delta(p_\alpha)$ so that
  $(a_\alpha h_R\otimes p_{\bar\alpha})\Delta(p_\lambda) = (h_R\otimes S)((
  p_\lambda\otimes p_\alpha)\Delta(p_\alpha)) S(a_\alpha)$.  Summing over
  $\lambda\in I(\LLambda)$ and applying Lemma~\ref{lem_kappa} we obtain
  \begin{displaymath}
    (a_\alpha h_R\otimes p_{\bar\alpha})\Delta(p_\LLambda) = (h_R\otimes S)((
    p_\LLambda\otimes p_\alpha)\Delta(p_\alpha))  S(a_\alpha) = \kappa_{\bar\alpha} S(a_\alpha).
  \end{displaymath}

  By Lemma~\ref{lemma_quotient_antipode}, for $a \in c_c(\GGamma/\LLambda)$ we
  can apply this formula to
  $S^{-1}(a_\alpha) \in p_{\bar\alpha}c(\LLambda\bs\GGamma)$, which yields
  $a_\alpha = \kappa_\alpha^{-1} (S^{-1}(a_\alpha)h_R\otimes
  p_\alpha)\Delta(p_\LLambda)$.  Now we observe that the right-hand side also
  lies in $p_\alpha c_c(\GGamma/\LLambda)$. Indeed we have
  $\Delta(p_\LLambda)(1\otimes p_\LLambda) = p_\LLambda\otimes p_\LLambda$ hence
  $\Delta^2(p_\LLambda)(1\otimes 1\otimes p_\LLambda) =
  \Delta(p_\LLambda)\otimes p_\LLambda$, which implies
  $\Delta(x)(1\otimes p_\LLambda) = x\otimes p_\LLambda$ for any
  $x = (\varphi\otimes\id)\Delta(p_\LLambda)$,
  by applying $\varphi\otimes\id\otimes\id$. Consequently we can apply
  Proposition~\ref{prp_quotient_description} which yields
  $p_{[\alpha]}a = \kappa_\alpha^{-1} (S^{-1}(a_\alpha)h_R\otimes p_{[\alpha]})
  \Delta(p_\LLambda)$ and summing over $[\alpha]$ we obtain the first equality
  in~\eqref{eq_strong_inv_right}.

  Equation~\eqref{eq_strong_inv_left} follows by
  applying~\eqref{eq_strong_inv_right} to $a = S^{-1}(b)$: this yields
  $b = \sum_{[\alpha]\in I(\GGamma)/\LLambda} \kappa_\alpha^{-1}$
  ${(h_Rb_{\bar\alpha}\otimes S)}\Delta(p_\LLambda)$.  Since
  $p_\LLambda = S^{-1}(p_\LLambda)$ and
  $\Delta S^{-1} = (S^{-1}\otimes S^{-1})\sigma\Delta$ we also have
  $b = \sum_{[\alpha]\in I(\GGamma)/\LLambda} \kappa_\alpha^{-1}
  (\id\otimes(h_Rb_{\bar\alpha})S^{-1})\Delta(p_\LLambda)$.  Finally we have
  $(h_Rb_{\bar\alpha})S^{-1} = S(b_{\bar\alpha})h_L$ and we obtain the rightmost
  side of~\eqref{eq_strong_inv_left} since $[\beta] = [\bar\alpha]$ runs through
  $\LLambda\bs I(\GGamma)$ when $[\alpha]$ runs through $I(\GGamma)/\LLambda$.

  The missing identities in~\eqref{eq_strong_inv_right}
  and~\eqref{eq_strong_inv_left} follow by replacing $a$, $b$ with $a^*$, $b^*$
  and taking adjoints on both sides. Alternatively one can use the fact that as
  we are in the context of discrete quantum groups we have
  $h_L(ab) = h_L(b S^2(a))$, $h_R(ab) = h_R(b S^{-2}(a))$.
\end{proof}

\begin{corollary} \label{crl_indep} Let $a \in c(\GGamma/\LLambda)$,
  $b\in c(\LLambda\bs\GGamma)$. Then $\kappa_\alpha^{-1}h_L(a_\alpha)$ only
  depends on the class $[\alpha] \in I(\GGamma)/\LLambda$, and
  $\kappa_{\bar\beta}^{-1} h_R(b_\beta)$ only depends on the class
  $[\beta] \in \LLambda\bs I(\GGamma)$.
\end{corollary}

\begin{proof}
  One can assume that $a \in p_{[\alpha]} c(\GGamma/\LLambda)$ and
  $b = S(a) \in p_{[\beta]} c(\LLambda\bs\GGamma)$ with $\beta = \bar\alpha$.
  We have
  $a = \kappa_\alpha^{-1} (S^{-1}(a_\alpha)h_R \otimes\id)\Delta(p_\LLambda)$
  by~\eqref{eq_strong_inv_right} and
  $p_{[\beta]} = \kappa_{\bar\beta}^{-1} (\id\otimes
  h_LS^{-1}(p_\beta))\Delta(p_\LLambda)$ by ~\eqref{eq_strong_inv_left}. We
  compute then
  \begin{align*}
    \kappa_{\bar\beta}^{-1} h_R(b_\beta) &= 
                                           \kappa_{\bar\beta}^{-1}h_L(S^{-1}(p_\beta)a) =
                                           \kappa_{\bar\beta}^{-1}\kappa_\alpha^{-1} 
                                           (S^{-1}(a_\alpha)h_R \otimes h_LS^{-1}(p_\beta))\Delta(p_\LLambda) \\
                                         &= \kappa_\alpha^{-1} h_R(p_\beta S^{-1}(a_\alpha))
                                           = \kappa_\alpha^{-1} h_L(a_\alpha S(p_\beta)) = \kappa_\alpha^{-1} h_L(a_\alpha).
  \end{align*}
  Since the left-hand (resp. right-hand) side does not depend on the choice of
  $\alpha$ in $[\alpha]$ (resp. $\beta \in [\beta]$), we are done.
\end{proof}

Note in particular that $\kappa_\alpha^{-1}h_L(p_\alpha) = \qdim(\alpha)^2/\kappa_\alpha$ 
depends only on $[\alpha] \in I(\GGamma)/\LLambda$. We have in fact the more general result below. 

\begin{corollary} \label{crl_kappa_indep} The quantity
  $\kappa_{\bar\alpha,\beta} / (\qdim(\bar\alpha)\qdim(\beta))$ depends only on
  the classes $[\alpha]$, $[\beta] \in I(\GGamma)/\LLambda$.
\end{corollary}

\begin{proof}
  We can assume $[\alpha] = [\beta]$ --- otherwise
  $\kappa(\bar\alpha,\beta) = 0$. Since
  $\kappa_{\bar\alpha,\beta} = \kappa_{\bar\beta,\alpha}$ it suffices to prove
  the independence on $\beta$. We apply~\eqref{eq_strong_inv_right} with
  $a = p_{[\alpha]}$: this yields
  $p_{[\alpha]} = \kappa_\alpha^{-1} (h_Rp_{\bar\alpha}\otimes
  \id)\Delta(p_\LLambda)$.  In particular
  $h_R(p_\beta) = \kappa_\alpha^{-1}(h_Rp_{\bar\alpha}\otimes
  h_Rp_\beta)\Delta(p_\LLambda)$.  Computing as in the proof of
  Lemma~\ref{lem_kappa} and denoting
  $P_\LLambda^{\bar\alpha,\beta} = \sum_{\lambda\in\LLambda}
  P_\lambda^{\bar\alpha,\beta}$ this can be written
  \begin{align*}
    (\qdim(\beta))^2 &= \kappa_\alpha^{-1} \qdim(\alpha) \qdim(\beta) ~
                       t_{\bar\alpha\otimes\beta}^* (\id\otimes\id\otimes P_\LLambda^{\bar\alpha,\beta}) 
                       t_{\bar\alpha\otimes\beta} \\
                     &= \kappa_\alpha^{-1} \qdim(\alpha) \qdim(\beta) \qdim(\bar\alpha\otimes\beta)_\LLambda
                       = \kappa_\alpha^{-1} \qdim(\alpha) \qdim(\beta) ~ \kappa_{\bar\alpha,\beta}.
  \end{align*}
  This shows that $\kappa_{\bar\alpha,\beta} / \qdim(\beta)$ does not depend on
  $\beta \in [\alpha]$.
\end{proof}

We remark that Corollary \ref{crl_kappa_indep} also has a purely categorical proof, suggested to us by the referee. 
Fix $\alpha\in I(\GGamma)$ and consider the right $ \Corep(\LLambda) $-module category 
$\Dd \subset \Corep(\GGamma)$ generated by $\alpha$. Consider also the functor 
$\underline\alpha: \Dd \to \Corep(\LLambda)$, $v\mapsto (\bar\alpha\otimes v)_\LLambda$. 
This is the internal $\Hom$ functor for $\mathcal{D}$, in the sense that 
$\Hom_\GGamma(\alpha\otimes\lambda,v) \simeq \Hom_\LLambda(\lambda,\underline\alpha(v))$
for all $v\in\Dd$, $\lambda\in\Corep(\LLambda)$. In \cite[Lemma A.4]{NeshveyevYamashita} 
it is shown that the rescaled Frobenius reciprocity isomorphism
\begin{displaymath}
  \Hom_\GGamma(\beta,\beta'\otimes\lambda) \to \Hom_\GGamma(\beta\otimes\bar\lambda,\beta'), \quad
  T \mapsto \sqrt{\frac{\kappa_{\bar\alpha,\beta'}}{\kappa_{\bar\alpha,\beta}}} ~ \tilde T,
\end{displaymath}
with $\tilde T = (\id_{\beta'}\otimes \bar R_\lambda^*) (T\otimes\id_{\bar\lambda})$, is 
unitary with respect to the hermitian structures given by
$S^*T = \langle S, T\rangle \id_\beta$, $\tilde S\tilde T^* = \langle \tilde T, \tilde S\rangle  \id_{\beta'}$ for $\beta,\beta'\in [\alpha]$. 
Now choose an isometric embedding
$T : \beta\to\beta'\otimes\lambda$ with some $\lambda\in\Corep(\LLambda)$. We have $\|T\|^2 = 1$, 
and using $\qTr(\tilde T\tilde T^*)=\qTr(TT^*)=\qdim(\beta)$we obtain $\|\tilde T\|^2 = \qdim(\beta)/\qdim(\beta ')$. 
Then unitarity of the map above then shows that 
$\kappa_{\bar\alpha,\beta'}/\qdim(\beta') = \kappa_{\bar\alpha,\beta}/\qdim(\beta)$ as required.

\bigskip

Let us finally introduce the quantum analogues $\mu$, $\nu$ of the counting
measures on $\GGamma/\LLambda$ resp.\ $\LLambda\bs\GGamma$. We will see at
Proposition~\ref{prp_invariant_measure} that $\mu$ is the (necessarily unique, up
to scalar multiplication) $\GGamma$-invariant weight on
$c_c(\GGamma/\LLambda)$. Observe that, applying~\eqref{eq_strong_inv_right} to
$a^*$ we obtain
\begin{displaymath}
  a^* = \sum_{[\alpha] \in I(\GGamma)/\LLambda} \kappa_\alpha^{-1} 
  (h_RS(a_\alpha^*)\otimes\id)\Delta(p_\LLambda).
\end{displaymath}
This explains the last identity in the following definition. By
Corollary~\ref{crl_indep} we also see that $\mu$, $\nu$ and $(a\mid b)$ as
defined below do not depend on the choices of $\gamma\in[\gamma]$.

\begin{definition} \label{def_invariant_measure}\label{def_scalar} Define
  $\mu \in c_c(\GGamma/\LLambda)^*$ by
  $\mu(a) = \sum_{[\alpha]\in I(\GGamma)/\LLambda} \kappa_\alpha^{-1}
  h_L(a_\alpha)$ where $a_\alpha = p_\alpha a$.  We endow
  $c_c(\GGamma/\LLambda)$ with the following positive-definite sesquilinear
  form, where $a$, $b \in c_c(\GGamma/\LLambda)$:
  \begin{align}
    (a\mid b) = \mu(a^*b) &= \sum_{[\gamma]\in I(\GGamma)/\LLambda} 
                            \kappa_\gamma^{-1} h_L(a^*p_\gamma b) 
                            \label{eq_scalar_right} \\
                          &= \sum_{[\alpha],[\beta]}\kappa_\alpha^{-1}\kappa_\beta^{-1}
                            (h_R S(a_\alpha^*) \otimes b_\beta h_L)\Delta(p_\LLambda). 
                            \label{eq_scalar_both}
  \end{align}
  We similarly define $\nu\in c_c(\LLambda\bs\GGamma)^*$ by
  $\nu(b) = \sum_{[\beta] \in \LLambda\bs I(\GGamma)} \kappa_{\bar\beta}^{-1}
  h_R(b_\beta)$, $b \in c_c(\LLambda\bs\GGamma)$.  Finally for
  $a\in c_c(\GGamma/\LLambda)$ we write
  $\|a\|_{\GGamma/\LLambda}:=(a\mid a)^{1/2}$.
\end{definition}

\begin{remark}
  This definition agrees via the Fourier transform with the scalar product on
  the ``dual algebra'' $\CC[\GGamma/\LLambda]$ introduced in
  \cite{VergniouxVoigt}. Recall that this algebra is defined as the relative
  tensor product
  $\CC[\GGamma/\LLambda] = \CC[\GGamma]\otimes_{\CC[\LLambda]}\CC$ with respect
  to the canonical embedding $\CC[\LLambda] \subset \CC[\GGamma]$ and to the
  counit $\epsilon : \CC[\LLambda] \to \CC$. The duality between $c_c(\GGamma)$
  and $\CC[\GGamma]$ is described by the multiplicative unitary
  $W = \bigoplus_{\alpha\in I(\GGamma)} u_\alpha \in \Mm (c_c(\GGamma)\otimes
  C[\GGamma])$ --- recall that we have $(\Delta\otimes\id)(W) = W_{13} W_{23}$,
  $(\id\otimes\Delta)(W) = W_{13}W_{12}$,
  $(S\otimes\id)(W) = W^* = (\id\otimes S^{-1})(W)$, and also, applying
  $S^{-1}\otimes S$: $(\id\otimes S)(W) = (S^{-1}\otimes\id)(W)$.  Recall also
  that we have a canonical conditional expectation
  $E : \CC[\GGamma] \to \CC[\LLambda]$, see \cite{Vergnioux_Amalg}.  Following
  \cite{VergniouxVoigt} we consider the linear bijection
  $\Ff_\LLambda^{-1} : \CC[\GGamma/\LLambda] \to c_c(\GGamma/\LLambda)$ given by
  $\Ff_\LLambda^{-1}(x) = {(\id\otimes \epsilon E S(x))}(W)$, which corresponds to
  (the inverse of) the Fourier transform when $\LLambda = \{e\}$. Using
  Theorem~\ref{thm_strong_inv} one can check the following explicit formula for
  the inverse map
  $\Ff_\LLambda : c_c(\GGamma/\LLambda) \to \CC[\GGamma/\LLambda]$:
  \begin{displaymath}
    \Ff_\LLambda(a) =  (a\mu\otimes\id)(W)\otimes_{\CC[\LLambda]} 1.
  \end{displaymath}
  On the other hand the space $\CC[\GGamma/\LLambda]$ has a natural
  prehilbertian structure given by ${(x\mid y)} = \epsilon E(x^*y)$. One can then
  easily check that $\Ff_\LLambda$ is an isometry with respect to this scalar
  product on $\CC[\GGamma/\LLambda]$ and the one of Definition~\ref{def_scalar}
  on $c_c(\GGamma/\LLambda)$.
\end{remark}

\subsubsection{The constants $\kappa$}

In contrast to the classical case, where they trivialize, the constants $\kappa$
of Definition~\ref{def_kappa} play an important role in the theory of quantum
group Hecke algebras (or in the quantum version of the Clifford theory). In this
subsection we give more details about them, which will however not be used in
the rest of the article (except in Proposition~\ref{prp_convol_subalg}).

Let us first note that related constants already appeared in Lemma~3.5 of
\cite{KKSV_Furstenberg}, which states the existence, for every
$\alpha \in I(\GGamma)$, of a constant $\eta_\alpha>0$ such that
\begin{displaymath}
  (\id \otimes \qTr_{\bar{\alpha}}) \Delta (p_\LLambda) = \eta_\alpha p_{[\alpha]},
\end{displaymath}
where $[\alpha] \in \LLambda\bs I(\GGamma)$.  As we have
$h_Lp_{\bar{\alpha}} = \qdim(\alpha) \qTr_{\bar{\alpha}}$, comparing the
formulas shows that we have in fact
$\eta_\alpha = \kappa_{\bar\alpha}/\qdim(\alpha)$.

The next example shows that we do not have $\kappa_{\bar\alpha} = \kappa_\alpha$
in general.

\begin{example} \label{exple_free_kappa} We consider the quantum subgroup
  $\LLambda \subset \GGamma = \ZZ*\LLambda$, where $\LLambda$ is any discrete
  quantum group. Denote by $a$ the generating corepresentation of $\ZZ$ and take
  $v\in I(\LLambda)$. Consider $\alpha = v\otimes a$, which is an irreducible
  corepresentation of $\GGamma$. We have
  $\bar\alpha\otimes \alpha \simeq \bigoplus_{w\subset\bar v\otimes v}
  a^{-1}\otimes w\otimes a$, hence $\kappa_{\alpha} = 1$, and
  $\alpha\otimes\bar\alpha = v\otimes \bar v$ hence
  $\kappa_{\bar\alpha} = \qdim(v)^2$.

  If $\LLambda$ is non-classical one can choose $v$ such that
  $\kappa_{\bar\alpha} = \qdim(v)^2 > 1 = \kappa_\alpha$, and if $\LLambda$ is
  finite, we are in the setting of Hecke pairs. For $\LLambda$ finite and non
  classical we can further consider $\LLambda^\infty = \varinjlim\LLambda^n$
  which is still commensurated inside $\GGamma^\infty$. For
  $w\in\Corep(\GGamma)$ we denote $w^{(k)} \in \Corep(\GGamma^\infty)$ the
  corepresentation corresponding to $w$ and the $k$th copy of $\GGamma$ in
  $\GGamma^\infty$. Then for any $v \in I(\LLambda)$ such that $\qdim(v) > 1$
  and $\alpha = \bigotimes_{k=1}^n(v\otimes a)^{(k)}\in I(\GGamma^\infty)$ we
  have $\kappa_{\bar\alpha}/\kappa_\alpha = \qdim(v)^{2n}$ which is not bounded.

  If one allows non commensurated quantum subgroups, one can take for $\LLambda$ the
  dual of $SU(2)$ and we see again that the ratios
  $\kappa_{\bar\alpha}/\kappa_\alpha$ are not bounded when $\alpha$
  varies. Notice also that $v = \alpha \otimes \bar a$,
  $\kappa_{v} = \qdim(v)^2$ and $\kappa_{\alpha} = 1$. This shows that there is
  no control on the ratios $\kappa_\gamma/\kappa_\alpha$ either when
  $\gamma \subset \alpha\otimes\beta$ with $\beta \in I(\GGamma)$ fixed.
\end{example}

\bigskip

We shall now give another interpretation of the constants
$\kappa_{\alpha,\beta}$ in terms of the $\Corep(\GGamma)$-module category
$\Corep(\GGamma/\LLambda)$. Let us first introduce some more relevant notation. For $\pi \in \Corep(\GGamma)$ and
$\alpha \in I(\GGamma)$ we denote $M_{\alpha,\pi} = \Hom(\alpha,\pi)$, so that
$H_\pi\simeq \bigoplus_{\alpha\in I(\GGamma)} H_\alpha\otimes M_{\alpha,\pi}$
canonically. When $\pi = \beta\otimes\gamma$ we denote
$C_{\alpha}^{\beta,\gamma} = M_{\alpha,\pi}$, so that the dimensions
$\cc_{\alpha}^{\beta,\gamma} = \dim(C_{\alpha}^{\beta,\gamma})$ are the
structure constants of the based Grothendieck ring of $\Corep(\GGamma)$.

We can extend this notation to objects of $\Corep(\GGamma/\LLambda)$. For
$i\in I(\GGamma/\LLambda)$, $\pi\in\Corep(\GGamma)$ we thus denote
$M_{i,\pi} = \Hom_{\GGamma/\LLambda}(i,\pi)$ where $\pi$ is identified to an
object in $\Corep(\GGamma/\LLambda)$ by restriction. Similarly we denote
$C_i^{\beta,j} = \Hom_{\GGamma/\LLambda}(i,\beta\otimes j)$ for $i$,
$j\in I(\GGamma/\LLambda)$ and $\beta\in\Corep(\GGamma)$.

Recall that the modular group of the left Haar weight is implemented by the
(possibly unbounded) modular element $F \in c(\GGamma)$ via the formula
$\sigma_t^L(f) = F^{it}f F^{-it}$. For $\pi\in\Corep(\GGamma)$ we denote
$F_\pi = \pi(F) \in B(H_\pi)$ and we recall the definition of the quantum
dimension $\qdim(\pi) = \Tr(F_\pi) = \Tr(F_\pi^{-1})$. Decomposing
$H_\pi \simeq \bigoplus_\alpha H_\alpha\otimes M_{\alpha,\pi}$ we have
$F_\pi = \sum F_\alpha\otimes \id$.

In general $F$ does not belong to $c_c(\GGamma/\LLambda)$ but we still have an
induced modular structure on this subalgebra. More precisely, take
$\alpha\in I(\GGamma)$ and decompose
$H_\alpha \simeq \bigoplus_{i\in I(\GGamma/\LLambda)}H_i\otimes M_{i,\alpha}$ in
$\Corep(\GGamma/\LLambda)$. Since $\sigma_t^L$ stabilizes
$\ell^\infty(\GGamma/\LLambda)$ by Lemma \ref{lemma_quotient_antipode},
$\Ad(F_\alpha^{it})$ stabilizes
$\alpha(\ell^\infty(\GGamma/\LLambda)) \simeq \bigoplus_{i} B(H_i) \otimes \id
\subset B(H_\alpha)$, hence also
$\alpha(\ell^\infty(\GGamma/\LLambda))' \simeq \bigoplus_i \id\otimes
B(M_{i,\alpha})$. It follows that for $i$ such that $M_{i,\alpha}\neq 0$ there
exist unique positive elements $F_i \in B(H_i)$,
$L_{i,\alpha}\in B(M_{i,\alpha})$ such that $\Tr(F_i) = \Tr(F_i^{-1})$,
$\Tr(L_{i,\alpha}) = \Tr(L_{i,\alpha}^{-1})$ and
$F_\alpha = \diag_i(F_i\otimes L_{i,\alpha})$.

Furthermore, decomposing
$H_\beta\otimes H_j \simeq \bigoplus_i H_i\otimes C_i^{\beta,j}$, we claim the
existence of a corresponding decomposition
$F_\beta\otimes F_j = \diag_i(F_i\otimes D_i^{\beta,j})$. Indeed, consider the
canonical isomorphisms
\begin{displaymath}
  H_\beta\otimes H_\alpha \simeq {\ts\bigoplus_j} H_\beta\otimes H_j\otimes
  M_{j,\alpha} \simeq {\ts\bigoplus_{i,j}}H_i\otimes C_{i}^{\beta,j}\otimes
  M_{j,\alpha}.
\end{displaymath}
Since $\Delta(F) = F\otimes F$ we have
$F_{\beta\otimes\alpha} = F_\beta\otimes F_\alpha = \diag(F_\beta\otimes
F_j\otimes L_{j,\alpha})$. In particular $\Ad(F_{\beta\otimes\alpha}^{it})$
stabilizes
$B(H_\beta)\otimes \alpha(\ell^\infty(\GGamma/\LLambda)) \simeq \bigoplus_j
B(\bigoplus_i H_i\otimes C_i^{\beta,j}) \otimes \id$, and also
$(\beta\otimes\alpha)(\ell^\infty(\GGamma/\LLambda)) \simeq \bigoplus_i
B(H_i)\otimes \id\otimes\id$, hence it stabilizes the relative commutant
$\bigoplus_{i,j} \id\otimes B(C_i^{\beta,j})\otimes \id$ as well. This shows the
existence of a unique positive element $D_i^{\beta,j} \in B(C_i^{\beta,j})$ such
that $\Tr(D_i^{\beta,j}) = \Tr((D_i^{\beta,j})^{-1})$ and
$F_\beta\otimes F_\alpha = \diag_{i,j}(F_i\otimes D_i^{\beta,j}\otimes
L_{j,\alpha})$. We have then also
$F_\beta\otimes F_j = \diag_i(F_i\otimes D_i^{\beta,j})$.

We summarize the conclusions of this discussion in the next proposition.

\begin{proposition}\label{def_modular_mult}
  For $i$, $j\in I(\GGamma/\LLambda)$, $\alpha,\beta \in I(\GGamma)$ we denote
  $F_i \in B(H_i)$, $L_{i,\alpha}\in B(M_{i,\alpha})$,
  $D_i^{\beta,j} \in B(C_i^{\beta,j})$ the unique positive elements, normalized
  by the identity $\Tr(A) = \Tr(A^{-1})$, such that
  $F_\alpha = \diag(F_i\otimes L_{i,\alpha})$ and
  $F_\beta\otimes F_i = \diag(F_j\otimes D_j^{\beta,i})$. The element $F_i$ does
  not depend on the choice of $\alpha\in I(\GGamma)$ such that
  $M_{i,\alpha}\neq 0$. We introduce the quantum dimensions, quantum
  multiplicities and quantum structure coefficients as follows:
  \begin{displaymath}
    \qdim(i) = \Tr(F_i), \quad \qmult(i,\alpha) = \Tr(L_{i,\alpha}), \quad \qc_i^{\beta,j} = \Tr(D_i^{\beta,j}).
  \end{displaymath}
\end{proposition}

\begin{proof}
  The element $F_i$ does not depend on $\alpha$ because $\Ad(F_i^{it})$
  corresponds for every $t \in \mathbb{R}$ to the restriction of $\sigma_t^L$ to
  $\ell^\infty(\GGamma/\LLambda)q_i$, where $q_i$ denotes the relevant minimal
  central projection of $\ell^\infty(\GGamma/\LLambda)$.
\end{proof}

By considering the case of a quantum subgroup
$\LLambda \subset \GGamma = \LLambda\times\LLambda'$ with $\LLambda$,
$\LLambda'$ non unimodular one sees that the elements $F_i$, $L_{i,\alpha}$,
$D_i^{\beta,j}$ can be non trivial, and that their traces are in general
different from the classical dimensions $\dim(i) = \dim(H_i)$, the classical
multiplicity $\mult(i,\alpha) = \dim(\Hom(i,\alpha))$ and the classical
coefficients $\cc_i^{\beta,j} = \dim(\Hom(i,\beta\otimes j))$.

\bigskip

We denote by $\CC[\GGamma]$, $\CC[\GGamma/\LLambda]$ the free vector spaces
generated by $I(\GGamma)$, $I(\GGamma/\LLambda)$, endowed with the scalar
products such that $I(\GGamma)$, $I(\GGamma/\LLambda)$ are orthonormal bases. To
any representation $\pi \in \Corep(\GGamma/\LLambda)$ corresponds naturally an
element
$\pi = \sum_{i\in I(\GGamma/\LLambda)} \mult(i,\pi) i \in
\ZZ[\GGamma/\LLambda]$. When $\pi$ comes from $\Corep(\GGamma)$ one can use
quantum multiplicities and we denote
\begin{displaymath}
  \tilde\pi = \sum_{i\in I(\GGamma/\LLambda)} \qmult(i,\pi) i \in \RR[\GGamma/\LLambda].
\end{displaymath}
Note that we have $\tilde\pi=\pi$ in $\CC[\GGamma]$ if $\GGamma$ is unimodular.

\begin{proposition} \label{prp_kappa_cat} For $\alpha$, $\beta \in I(\GGamma)$
  we have $\kappa_{\bar\alpha,\beta} = (\tilde\alpha \mid \tilde\beta)$, where
  $(\cdot\mid\cdot)$ is the scalar product of $\CC[\GGamma/\LLambda]$. In
  particular $\kappa_\alpha = \|\tilde\alpha\|^2$.
\end{proposition}

\begin{proof}
  Recall the identification
  $\Hom_{\GGamma/\LLambda}(\alpha,\beta) = (\bar\alpha\otimes\beta)_\LLambda$
  via the map $T\mapsto (\id\otimes T)t_\alpha$,
  $T \in \Hom_{\GGamma/\LLambda}(\alpha,\beta)$. We have
  $(F_{\bar\alpha}\otimes F_\beta)(\id\otimes T)t_\alpha = (\id\otimes F_\beta T
  F_\alpha^{-1})t_\alpha$. Denote by
  $\Ff \in B(\Hom_{\GGamma/\LLambda}(\alpha,\beta))$ the map given by
  $\Ff(T) =F_\beta T F_\alpha^{-1}$ for
  $T\in \Hom_{\GGamma/\LLambda}(\alpha,\beta)\subset B(H_\alpha,H_\beta)$, so
  that we have $\kappa_{\bar\alpha,\beta} = \Tr(\Ff)$. Decomposing
  $H_\alpha \simeq \bigoplus H_i\otimes M_{i,\alpha}$,
  $H_\beta \simeq \bigoplus H_i\otimes M_{i,\beta}$ we have
  $\Hom_{\GGamma/\LLambda}(\alpha,\beta) = \bigoplus_i \id_i\otimes
  B(M_{i,\alpha}, M_{i,\beta})$. Using the corresponding decomposition of the
  $F$-matrices we have $\Ff(T) = \sum_i L_{i,\beta}T_iL_{i,\alpha}^{-1}$ for
  $T = \diag(\id\otimes T_i)$. Computing $\Tr(\Ff)$ in a basis of matrix units
  this yields
  $\kappa_{\bar\alpha,\beta} = \sum_i\Tr(L_{i,\beta}) \Tr(L_{i,\alpha}^{-1}) =
  \sum_i\qmult(i,\beta)\qmult(i,\alpha) = (\tilde\alpha \mid \tilde\beta)$.
\end{proof}

We have then the following result generalizing (and resulting from)
Corollary~\ref{crl_kappa_indep}:

\begin{proposition} \label{crl_kappa_indep_2} The vector
  $\tilde\alpha ~/\qdim(\alpha) \in \CC[\GGamma/\LLambda]$ only depends on
  $[\alpha]$.
\end{proposition}

\begin{proof}
  We have to show that for $i\in I(\GGamma/\LLambda)$, $\alpha\in I(\GGamma)$,
  the number $\qmult(i,\alpha)/\qdim(\alpha)$ depends only on $i$ and the class
  $[\alpha] \in I(\GGamma)/\LLambda$. Denote by
  $q_i \in c_c(\GGamma/\LLambda) \subset \ell^\infty(\GGamma)$ the minimal
  central projection of $c_c(\GGamma/\LLambda)$ corresponding to
  $i\in I(\GGamma/\LLambda)$. We have $(q_i)_\alpha = q_i p_\alpha$ and so by
  Corollary~\ref{crl_indep} the number $\kappa^{-1}h_L(q_ip_\alpha)$ only
  depends on $[\alpha]$. On the other hand we can compute it using the canonical
  identification $H_\alpha \simeq \bigoplus_i H_i \otimes M_{i,\alpha}$ as
  follows:
  \begin{align*}
    \kappa_\alpha^{-1}h_L(q_ip_\alpha) &= \kappa_\alpha^{-1}\qdim(\alpha) \Tr(F_\alpha^{-1}p_\alpha q_i)
                                         = \kappa_\alpha^{-1}\qdim(\alpha) (\Tr\otimes\Tr)(F_i^{-1}\otimes L_{i,\alpha}^{-1}) \\
                                       &= \qdim(i) \frac{\qdim(\alpha)^2}{\kappa_\alpha} \frac{\qmult(i,\alpha)}{\qdim(\alpha)}.
  \end{align*}
  The assertion follows since $\qdim(\alpha)^2/\kappa_\alpha$ only depends on
  $[\alpha]$ by Corollary \ref{crl_kappa_indep}.
\end{proof}

\subsection{The Hecke algebra}
\label{sec_hecke_discrete}

\subsubsection{The convolution product}

Now we introduce the Hecke convolution product. Note that, thanks to
Theorem~\ref{thm_strong_inv}, the different expressions for $a*b$ given in
the next definition are indeed equal. By comparing~\eqref{eq_convol_left} and
\eqref{eq_convol_right} one sees that they do not depend on the choices of
$\alpha \in [\alpha]$, $\beta\in [\beta]$.

\begin{definition} \label{def_convol} Let $a \in c_c(\GGamma/\LLambda)$,
  $b\in c_c(\LLambda\bs\GGamma)$. According to Theorem~\ref{thm_strong_inv}
  we can define $a*b \in c(\GGamma)$ as follows:
  \begin{align}
    a*b ~ &= \sum_{[\alpha]\in I(\GGamma)/\LLambda} \kappa_\alpha^{-1} 
            (S^{-1}(a_\alpha)h_R\otimes\id) \Delta(b) 
            ~= \sum_{[\alpha]\in I(\GGamma)/\LLambda} \kappa_\alpha^{-1} 
            (h_RS(a_\alpha)\otimes\id) \Delta(b) 
            \label{eq_convol_left} \\ 
          &= \sum_{[\alpha], [\beta]} \kappa_\alpha^{-1}\kappa_{\bar\beta}^{-1}
            (S^{-1}(a_\alpha)h_R\otimes\id\otimes h_LS^{-1}(b_\beta)) 
            \Delta^2(p_\LLambda) \label{eq_convol_both} \\
          &= \sum_{[\beta]\in \LLambda\bs I(\GGamma)} \kappa_{\bar\beta}^{-1} 
            (\id\otimes h_LS^{-1}(b_\beta)) \Delta(a)
            ~= \sum_{[\beta]\in \LLambda\bs I(\GGamma)} \kappa_{\bar\beta}^{-1} 
            (\id\otimes S(b_\beta)h_L) \Delta(a). 
            \label{eq_convol_right}
  \end{align}
\end{definition}

In the classical case $\GGamma = \Gamma$ we have $\kappa_\alpha = 1$ for all
$\alpha \in \Gamma$, and we recover the classical formulae for the Hecke
convolution product \cite{Shimura_introduction}. Note that one can in fact
define $a*b \in c(\GGamma)$ for $a\in c_c(\GGamma/\LLambda)$,
$b\in c(\LLambda\bs\GGamma)$ (resp.  $a\in c(\GGamma/\LLambda)$,
$b\in c_c(\LLambda\bs\GGamma)$) using~\eqref{eq_convol_left}
(resp.~\eqref{eq_convol_right}). The following results are immediate from the
Definition:

\begin{proposition}\label{prp_convol_coprod}
  For any $a \in c_c(\GGamma/\LLambda)$, $b\in c_c(\LLambda\bs\GGamma)$ we have
  $\Delta(a*b) = (a*\otimes\id)\Delta(b) = (\id\otimes *b)\Delta(a)$. In
  particular we have $a*b \in c(\GGamma/\LLambda)$ if
  $b\in c_c(\LLambda\bs\GGamma)\cap c(\GGamma/\LLambda)$ and
  $a*b \in c(\LLambda\bs\GGamma)$ if
  $a\in c_c(\GGamma/\LLambda)\cap c(\LLambda\bs\GGamma)$.
\end{proposition}

One can also express the convolution product using the functionals $\mu$, $\nu$
from Definition~\ref{def_scalar}, for instance we have
$a*b = ((\mu a)S\otimes\id)\Delta(b)$ for $a\in c_c(\GGamma/\LLambda)$,
$b\in c(\LLambda\bs\GGamma)$ and $a*b = (\id\otimes S(b)\mu)\Delta(a)$ for
$a\in c(\GGamma/\LLambda)$, $b\in c_c(\LLambda\bs\GGamma)$. From the first
expression and Proposition~\ref{prp_convol_coprod} we then immediately obtain  
the following statement. 
  
\begin{proposition}\label{prp_invariant_measure}
  For any $a\in c_c(\GGamma/\LLambda)$ we have $a*1 = \mu(a)1$, where $\mu$ is
  the functional of Definition~\ref{def_invariant_measure}. Moreover $\mu$ is
  $\GGamma$-invariant: we have $(\id\otimes\mu)\Delta(a) = \mu(a) 1$ for any
  $a\in c_c(\GGamma/\LLambda)$.
\end{proposition}

In the next proposition we give a description of the convolution product between
$c_c(\GGamma/\LLambda)$ and $c_c(\LLambda\bs\GGamma) = S(c_c(\GGamma/\LLambda))$
using only the structure of the $\GGamma$-invariant subalgebra
$\ell^\infty(\GGamma/\LLambda) \subset \ell^\infty(\GGamma)$, without explicit
reference to the quantum subgroup $\LLambda$. Note that by
Proposition~\ref{crl_kappa_indep_2} the fractions appearing in the
expression~\eqref{eq_convol_subalg} only depend on $i$ (and not on the choice of
$\beta_i$).

\begin{proposition}\label{prp_convol_subalg}
  Choose for each $i\in I(\GGamma/\LLambda)$ an element $\beta_i\in I(\GGamma)$
  such that $\qmult(i,\beta_i)$ $\neq 0$ and consider the linear forms
  $\varphi_i : \ell^\infty(\GGamma/\LLambda) \to \CC$,
  $\varphi_i(a) = \Tr(F_i^{-1})\Tr(F_i^{-1}a_i)$. Then for all
  $a\in c_c(\GGamma/\LLambda)$, $b\in c_c(\LLambda\bs\GGamma)$ we have
  \begin{equation}\label{eq_convol_subalg}
    a*b = \sum_{i\in I(\GGamma/\LLambda)}
    \frac {\qdim(\beta_i)\qmult(i,\beta_i)}{\kappa_{\beta_i}\qdim(i)}
    (\id\otimes S(b)\varphi_i)\Delta(a).
  \end{equation}
\end{proposition}

\begin{proof}
  We start from the last expression in equation~\eqref{eq_convol_right}. We
  shall compute the linear form
  $\sum_{[\beta]}\kappa_{\bar\beta}^{-1} S(b_\beta) h_L$ on the matrix blocks of
  $c_c(\GGamma/\LLambda)$ using the canonical identification
  $p_{\bar\beta}c_c(\GGamma/\LLambda) \simeq \bigoplus_i B(H_i)\otimes
  \id_{M(i,\bar\beta)} \subset B(H_{\bar\beta})$. For
  $a\in c_c(\GGamma/\LLambda)$ we have
  \begin{align*}
    \sum_{[\beta]}\kappa_{\bar\beta}^{-1} S(b_\beta) h_L(a) &=
                                                              \sum_{[\beta]}\sum_i \kappa_{\bar\beta}^{-1} \Tr(F_{\bar\beta}^{-1})(\Tr\otimes\Tr)(F_i^{-1}a_i S(b)_i\otimes L_{i,\bar\beta}^{-1}) \\
                                                            &= \sum_{[\beta]}\sum_i \kappa_{\bar\beta}^{-1} \qdim(\bar\beta) \Tr(F_i^{-1}a_i S(b)_i) \qmult(i,\bar\beta). 
  \end{align*}
  Putting $\bar\beta = \beta_i$ for each $i \in I(\GGamma/\LLambda)$ we
  obtain~\eqref{eq_convol_subalg}.
\end{proof}

\begin{proposition} \label{prp_convol} If $a$,
  $b \in c_c(\GGamma/\LLambda)^\LLambda := c_c(\GGamma/\LLambda) \cap
  c(\LLambda\bs\GGamma)$ then $a*b \in c_c(\GGamma/\LLambda)^\LLambda$ as
  well. The convolution production defined in this way on
  $c_c(\GGamma/\LLambda)^\LLambda$ is bilinear, associative, with unit element
  $p_\LLambda$. We have $\sigma^R_t(a*b) = \sigma^R_t(a)*\sigma^R_t(b)$ for any
  ${t\in \mathbb{R}}$. One obtains similarly a convolution product on
  $c_c(\LLambda\bs\GGamma)^\LLambda = c_c(\LLambda\bs\GGamma) \cap
  c(\GGamma/\LLambda)$.  Moreover, the map $a \mapsto a^\sharp := S(a^*)$ is an
  antimultiplicative, antilinear involution exchanging both convolution
  algebras.
\end{proposition}

\begin{proof}
  Since
  $(c_c(\GGamma)\otimes 1)\Delta(c_c(\GGamma/\LLambda)) = c_c(\GGamma)\otimes
  c_c(\GGamma/\LLambda)$ it follows from~\eqref{eq_convol_left} that
  $a*b \in c_c(\GGamma/\LLambda)$ if $a$,
  $b\in c_c(\GGamma/\LLambda)^{\LLambda}$. It follows
  from~\eqref{eq_convol_left} and Theorem~\ref{thm_strong_inv} that
  $p_\LLambda$ is a unit for the convolution product.
  
  To prove associativity we will use~\eqref{eq_convol_right} to compute $b*c$
  with $b \in c_c(\GGamma/\LLambda)$, $c \in c(\LLambda\bs\GGamma)$. This makes
  sense in the multiplier algebra $c(\GGamma)$. Indeed, fix
  $\delta \in I(\GGamma)$ and note that if
  $(p_\delta\otimes h_LS^{-1}(c_\gamma))\Delta(b)$ is non zero for
  $\gamma \in I(\Gamma)$, there exists $\beta$ such that $b_\beta\neq 0$ and
  $\beta\subset \delta\otimes\bar\gamma$. Then
  $\bar\gamma \subset \bar\delta\otimes\beta$, i.e.
  $[\gamma]\in \LLambda\bs I(\GGamma)$ is conjugate to the class in
  $I(\GGamma)/\LLambda$ of a subobject of $\bar\delta\otimes\beta$. Since there
  is only a finite number of classes $[\beta] \in I(\GGamma)/\LLambda$ such that
  $b_\beta\neq 0$, there is only a finite number of classes
  $[\gamma] \in \LLambda\bs I(\GGamma)$ which yield a non-zero term on the
  right-hand side of \eqref{eq_convol_right}.

  Now for $a$, $b$, $c \in c_c(\GGamma/\LLambda)^\LLambda$ we can write, using
  classes $[\alpha]\in I(\GGamma)/\LLambda$,
  $[\gamma]\in\LLambda\bs I(\GGamma)$:
  \begin{align*}
    a*(b*c) &= \sum_{[\alpha], [\gamma]} \kappa_\alpha^{-1} \kappa_{\bar\gamma}^{-1} 
              (S^{-1}(a_\alpha)h_R\otimes\id)\Delta(
              (\id\otimes h_LS^{-1}(c_\gamma))\Delta(b)) \\
            &= \sum_{[\alpha], [\gamma]} \kappa_\alpha^{-1} \kappa_{\bar\gamma}^{-1}
              (S^{-1}(a_\alpha)h_R\otimes\id\otimes 
              h_LS^{-1}(c_\gamma))\Delta^2(b) \\
            &= \sum_{[\alpha], [\gamma]} \kappa_\alpha^{-1} \kappa_{\bar\gamma}^{-1}
              (\id\otimes  h_LS^{-1}(c_\gamma))\Delta(
              (S^{-1}(a_\alpha)h_R\otimes\id)\Delta(b)) = (a*b)*c.
  \end{align*}
  Since $\Delta\sigma_t^R = (\sigma_t^R\otimes\sigma_t^R)\Delta$, the modular
  automorphisms leave $B(H_\alpha)$ invariant and $S\sigma^R_t = \sigma^R_tS$ we
  have, for $a\in c_c(\GGamma/\LLambda)$ and $b\in c(\LLambda\bs\GGamma)$:
  \begin{align*}
    \sigma_t^R(a)*\sigma_t^R(b) &= \sum_{[\alpha]\in I(\GGamma)/\LLambda} 
                                  \kappa_\alpha^{-1} 
                                  (h_R\otimes\id) [(\sigma_t^R\otimes\sigma_t^R)\Delta(b)
                                  (\sigma_t^RS^{-1}(a_\alpha)\otimes 1)] \\
                                &= \sum_{[\alpha]\in I(\GGamma)/\LLambda} 
                                  \kappa_\alpha^{-1} (h_R\otimes\sigma_t^R) [\Delta(b)
                                  (S^{-1}(a_\alpha)\otimes 1)] = \sigma_t^R(a*b).
  \end{align*}
  We compute similarly for the involution, using first~\eqref{eq_convol_left}
  and then~\eqref{eq_convol_right}:
  \begin{align*}
    a^\sharp*b^\sharp &= \sum_{[\alpha] \in I(\GGamma)/\LLambda} \kappa_\alpha^{-1} 
                        (a_{\bar\alpha}^*h_R\otimes\id)\Delta(S(b^*))
                        = \sum_{[\alpha]  \in I(\GGamma)/\LLambda} \kappa_\alpha^{-1} 
                        (S\otimes a_{\bar\alpha}^*h_RS)\Delta(b^*) \\
                      &= \sum_{[\alpha]  \in \LLambda\bs I(\GGamma)} \kappa_{\bar\alpha}^{-1} 
                        [\id \otimes (\overline{a_{\alpha}^*h_RS} )\Delta(b)]^\sharp
                        = (b*a)^\sharp,
  \end{align*}
  since $\overline{a_{\alpha}^*h_RS} = S(a_\alpha)h_L$.
\end{proof}

We can finally make the following definition. Recall that we have
$c_c(\GGamma/\LLambda) \cap c_c(\LLambda\bs\GGamma) =
c_c(\LLambda\bs\GGamma'/\LLambda)$ where $\GGamma'$ is the commensurator of
$\LLambda$ in $\GGamma$, and $\GGamma' = \GGamma$ if $(\GGamma,\LLambda)$ is a
Hecke pair.

\begin{definition} \label{df_hecke_algebra} Let $\LLambda$ be a quantum subgroup
  of a discrete quantum group $\GGamma$. The {\em Hecke algebra} associated with
  $(\GGamma,\LLambda)$ is
  $\Hh(\GGamma,\LLambda) := c_c(\GGamma/\LLambda) \cap c_c(\LLambda\bs\GGamma)$
  equipped with the convolution product $*$ and the involution
  $\,\cdot\,^\sharp$.
\end{definition}

\begin{remark}\label{rk_struct_hecke}
  Note that $\Hh(\GGamma,\LLambda)$ depends only on the pair
  $(\GGamma',\LLambda)$, where $\GGamma'$ is the commensurator of $\LLambda$ in
  $\GGamma$. The algebra $\Hh(\GGamma,\LLambda)$ is also equipped with more
  structure coming from the ambient space $c(\GGamma)$: the ``pointwise''
  product and involution, the scaling and modular groups
  $(\tau_t)_{t\in \mathbb{R}}$, $(\sigma_t^R)_{t\in \mathbb{R}}$,
  $(\sigma_t^L)_{t\in \mathbb{R}}$, see
  Lemma~\ref{lemma_quotient_antipode}. Moreover all these one-parameter groups
  act by automorphisms of $\Hh(\GGamma,\LLambda)$, which can be shown similarly
  as it was done for $(\sigma_t^R)_{t\in \mathbb{R}}$ in the last proposition.
\end{remark}

\begin{example}\label{exple_normal}
  Let us mention some ``trivial'' examples where $\Hh(\GGamma,\LLambda)$ is
  non-trivial. We will investigate more interesting examples in
  subsection~\ref{sec_HNN}. If $\LLambda$ is finite or of finite index in
  $\GGamma$, but different from $\GGamma$, then obviously $\GGamma = \GGamma'$
  and $\CC p_\LLambda \subsetneq \Hh(\GGamma,\LLambda)$. If $\LLambda$ is finite
  and $\GGamma$ infinite, $\dim(\Hh(\GGamma,\LLambda)) = +\infty$.

  Consider now the case when $\LLambda$ is normal in $\GGamma$, i.e.\
  $\ell^\infty(\GGamma/\LLambda) = \ell^\infty(\LLambda\bs\GGamma)$. Then
  $\Delta(\ell^\infty(\GGamma/\LLambda))$ $\subset \ell^\infty(\GGamma/\LLambda)
  \bar\otimes \ell^\infty(\GGamma/\LLambda)$, so that we have a discrete quantum
  group $\GGamma/\LLambda$ underlying the quantum quotient space. The dual $\HH$
  of $\GGamma/\LLambda$ identifies with a normal subgroup of the compact quantum
  group $\GG$ dual to $\GGamma$, with restriction morphism
  $\rho : C_u(\GG) \to C_u(\HH)$, such that given $\alpha\in\Corep(\GGamma)$ we
  have $\alpha\in\Corep(\LLambda)$ iff $(\id\otimes\rho)(\alpha)$ is
  trivial --- or, equivalently, contains the $1$-dimensional
  corepresentation. It is then easy to check that $\alpha\sim\beta$ iff
  $\one \subset (\id\otimes\rho)(\bar\alpha\otimes\beta)$ iff
  $\one \subset (\id\otimes\rho)(\beta\otimes\bar\alpha)$ iff
  $\alpha\backsim\beta$. In particular
  $I(\GGamma)/\LLambda = \LLambda\bs I(\GGamma)/\LLambda = \LLambda\bs
  I(\GGamma)$ and $L(\dc\alpha) = R(\dc\alpha) = 1$ for all
  $\alpha \in I(\GGamma)$, so that $(\GGamma,\LLambda)$ is a Hecke pair.

  By uniqueness and the invariance result of
  Proposition~\ref{prp_invariant_measure}, the functional $\mu$, $\nu$ of
  Definition~\ref{def_invariant_measure} are the left, resp. right Haar weights
  of $\GGamma/\LLambda$, normalized by the condition
  $\mu(p_\LLambda) = \nu(p_\LLambda) = 1$. Moreover, from the formula
  $a*b = (\id\otimes S(b)\mu)\Delta(a)$ we recognize the usual convolution
  product of the discrete quantum group algebra $\ell^\infty(\GGamma/\LLambda)$
  (or its opposite, depending on conventions), transported from the product of
  $\CC[\GGamma/\LLambda]$ via the Fourier transform.
\end{example}

\subsubsection{Hecke operators} Now we want to identify the Hecke algebra with
an algebra of equivariant endomorphisms of $c_c(\GGamma/\LLambda)$. We turn
$c_c(\GGamma/\LLambda)$ into a right $\CC[\GGamma]$-module via the formula
$a\cdot x=(x\otimes \id)\Delta(a)$, where $a\in c_c(\GGamma/\LLambda)$,
$x \in \CC[\GGamma]$ and we view $\CC[\GGamma]$ as linear subspace of
$c_c(\GGamma)^*$ via evaluation.
By definition, a linear map $F:c_c(\GGamma/\LLambda)\to c_c(\GGamma/\LLambda)$
is $\GGamma$-equivariant iff it is $\CC[\GGamma]$-linear, and we write
$\End_\GGamma(c_c(\GGamma/\LLambda))$ for the space of $\GGamma$-equivariant
linear endomorphisms of $c_c(\GGamma/\LLambda)$.

Restricting the action of $\CC[\GGamma]$ to $\CC[\LLambda]$ we obtain the
associated space of fixed points
$c_c(\GGamma/\LLambda)^\LLambda = \{a \in c_c(\GGamma/\LLambda) \mid
(p_\LLambda\otimes 1)\Delta(a) = p_\LLambda\otimes a\} = \{a \in
c_c(\GGamma/\LLambda) \mid a \cdot x = \hat{\epsilon}(x) a
\}=c_c(\GGamma/\LLambda) \cap c(\LLambda\bs\GGamma)$.

\begin{proposition} \label{prp_endo} Let $\LLambda$ be a quantum subgroup of
  $\GGamma$. The map $\ev_\LLambda : F \mapsto f := F(p_\LLambda)$
  defines an antimultiplicative isomorphism from
  $\End_\GGamma(c_c(\GGamma/\LLambda))$ to $c_c(\GGamma/\LLambda)^\LLambda$,
  with inverse bijection $\Tdisc$ given by $\Tdisc(f)(a) = a*f$,
  $f \in c_c(\GGamma/\LLambda)^\LLambda$, $a \in c_c(\GGamma/\LLambda)$.
\end{proposition}

\begin{proof}
  Let $F \in \End_\GGamma(c_c(\GGamma/\LLambda))$.  For all
  $g \in c_c(\GGamma/\LLambda)$ and $\alpha \in I(\GGamma)$ we have
  ${(p_\alpha\otimes\id)}\Delta(g) \in c_c(\GGamma) \otimes
  c_c(\GGamma/\LLambda)$ and equivariance of $F$ means that we have
  $(\id\otimes F)((p_\alpha\otimes\id)\Delta(g)) =
  {(p_\alpha\otimes\id)}\Delta(F(g))$.  Applying this to $g = p_\LLambda$ we
  obtain
  $(\id\otimes F)((p_\alpha\otimes\id)\Delta(p_\LLambda)) =
  {(p_\alpha\otimes\id)}$ $\Delta(f)$.  Since $p_\LLambda$ is also left
  invariant under $\LLambda$, for $\alpha = \lambda \in I(\LLambda)$ we obtain
  $(p_\lambda\otimes\id)\Delta(f) = (\id\otimes F)(p_\lambda\otimes p_\LLambda)
  = p_\lambda\otimes f$ and we have indeed
  $f \in c_c(\GGamma/\LLambda)^\LLambda$.  The same equivariance formula can
  also be written
  $F((a_\alpha h_R\otimes\id) \Delta(p_\LLambda)) = (a_\alpha
  h_R\otimes\id)\Delta(f)$ for any $a\in c(\GGamma)$ and $\alpha\in I(\GGamma)$.

  Moreover, take $a \in p_{[\alpha]}c_c(\GGamma/\LLambda)$.  Using
  Theorem~\ref{thm_strong_inv} and the previous equivariance formula we can
  write
  \begin{align*}
    F(a) &= \kappa_\alpha^{-1} F((S^{-1}(a_\alpha)h_R\otimes\id) 
           \Delta(p_\LLambda)) \\
         &= \kappa_\alpha^{-1} (S^{-1}(a_\alpha)h_R\otimes\id) \Delta(f) 
           = a*f = \Tdisc(f)(a).
  \end{align*}
  This shows that $\Tdisc$ is a left inverse of $\ev_\LLambda$.

  Conversely starting from $f \in c_c(\GGamma/\LLambda)^\LLambda$ we can
  consider $\Tdisc(f) : a \mapsto a*f$. We already noticed at
  Proposition~\ref{prp_convol} that $\Tdisc(f)(a) \in c_c(\GGamma/\LLambda)$, and
  after Definition~\ref{def_convol} that $\Tdisc(f)$ is equivariant with respect
  to the left $\GGamma$-action induced by $\Delta$. Since $p_\LLambda$ is the
  unit of the convolution product, $\Tdisc$ is a right inverse of
  $\ev_\LLambda$. Finally $\Tdisc$ is antimultiplicative by associativity
  of the convolution product.
\end{proof}

We now use the prehilbertian structure on $c_c(\GGamma/\LLambda)$ obtained from
the functional $\mu$ and consider the corresponding subspace of adjointable
operators in $\End_\GGamma(c_c(\GGamma/\LLambda))$. Note that the formula
\eqref{eq_scalar_right} for $(a\mid b)$ given in
Definition~\ref{def_invariant_measure} also makes sense for
$a\in c_c(\GGamma/\LLambda)$, $b\in c(\GGamma/\LLambda)$, or for
$a\in c(\GGamma/\LLambda)$, $b\in c_c(\GGamma/\LLambda)$. We use this in the
following proposition. Note also that $\|p_\LLambda\|_{\GGamma/\LLambda} = 1$.

\begin{proposition} \label{prop_adjoint} For $a$, $b \in c_c(\GGamma/\LLambda)$
  and $c \in c(\GGamma/\LLambda) \cap c(\LLambda\bs\GGamma)$ we have
  $(a \mid b*c) = (a * c^\sharp \mid b)$.
\end{proposition}

\begin{proof}
  We compute, using \eqref{eq_scalar_right} and \eqref{eq_convol_left}:
  \begin{align*}
    (a*c^\sharp \mid b) &= \sum_{[\beta]} \kappa_\beta^{-1} (b_\beta h_L)
                          ((a*c^\sharp)^*)
                          = \sum_{[\alpha],[\beta]} \kappa_\alpha^{-1} \kappa_\beta^{-1} 
                          (\overline{S^{-1}(a_\alpha)h_R} \otimes b_\beta h_L)\Delta(c^{\sharp *}) \\
                        &= \sum_{[\alpha],[\beta]} \kappa_\alpha^{-1} \kappa_\beta^{-1} 
                          (b_\beta h_L \otimes \overline{S^{-1}(a_\alpha)h_R})
                          (S^{-1}\otimes S^{-1})\Delta(c) \\
                        &= \sum_{[\alpha],[\beta]} \kappa_\alpha^{-1} \kappa_\beta^{-1} 
                          (h_R S(b_\beta) \otimes S^2(a_\alpha^*)h_L)\Delta(c) 
                          = \sum_{[\alpha]} \kappa_{\alpha}^{-1}  (h_La_\alpha^*)(b*c)
                          = (a \mid b*c).
  \end{align*}
  We also used the identities
  $\overline{S^{-1}(a_\alpha)h_R} \circ S^{-1} = S^2(a_\alpha^*)h_L = h_L
  a_\alpha^*$ which are easy to check.
\end{proof}

In the next theorem, which offers an alternative description of the Hecke
algebra, we write
$\End'(c_c(\GGamma/\LLambda)) \subset \End(c_c(\GGamma/\LLambda))$ for the
subspace of adjointable maps, i.e.\ of maps $T$ for which there exists
$S \in \End(c_c(\GGamma/\LLambda))$ satisfying $(a \mid Tb) = (Sa \mid b)$ for
all $a$, $b\in c_c(\GGamma/\LLambda)$.

\begin{theorem} \label{thm_adjoint} Let $f \in
  c_c(\GGamma/\LLambda)^\LLambda$. Then we have
  $\Tdisc(f) \in \End'(c_c(\GGamma/\LLambda))$ iff
  $f \in c_c(\GGamma/\LLambda) \cap c_c(\LLambda\bs\GGamma)$. As a result,
  $\Tdisc$ implements an antimultiplicative $*$-isomorphism between
  $\Hh(\GGamma,\LLambda)$ and $\End'_\GGamma(c_c(\GGamma/\LLambda))$. If
  $(\GGamma,\LLambda)$ is a Hecke pair, all maps in
  $\End_\GGamma(c_c(\GGamma/\LLambda))$ are adjointable.
\end{theorem}

\begin{proof}
  In view of Proposition~\ref{prp_endo} it essentially suffices to understand
  the relevant $*$-struc\-tures.  By Proposition~\ref{prop_adjoint}, if
  $\Tdisc(f)$ is adjointable then its adjoint is $S : a \mapsto
  a*f^\sharp$. Taking $a = p_\LLambda$ we obtain
  $f^\sharp \in c_c(\GGamma/\LLambda)$, hence $f \in c_c(\LLambda\bs\GGamma)$.
  The converse implication and the statement about $\Tdisc$ are then clear. If
  $(\GGamma,\LLambda)$ is a Hecke pair, we have
  $c_c(\GGamma/\LLambda)^\LLambda = c_c(\GGamma/\LLambda) \cap
  c_c(\LLambda\bs\GGamma)$.
\end{proof}

Now we investigate the question whether Hecke operators extend to bounded
operators on $\ell^2(\GGamma/\LLambda)$, the completion of
$c_c(\GGamma/\LLambda)$ with respect to the norm $\|\cdot\|_{\GGamma/\LLambda}$
arising from Definition~\ref{def_scalar}. In the setting of discrete quantum
groups we prove in Theorem~\ref{thm_bounded} that this is equivalent to a
combinatorial property of the constants $\kappa_\alpha$ introduced in
Definition~\ref{def_RT} below. Surprisingly we could not prove directly that
this property always hold for Hecke pairs. However we will show later in
Section~\ref{sec_schlichting}, using the Schlichting completion, that this is indeed the case, by showing that Hecke
operators are always bounded.

\begin{definition}\label{def_RT}
  Given the inclusion $\LLambda\subset \GGamma$ and $\beta \in I(\GGamma)$ we say that $\beta$ satisfies property (RT) if there exists a constant $C_\beta$ such
  that $\kappa_\gamma \leq C_\beta \kappa_\alpha$ for all $\alpha$, $\gamma\in I(\GGamma)$ such that $\gamma\subset\alpha\otimes\beta$. We say that the
  inclusion $\LLambda\subset \GGamma$ satisfies property (RT) if the property (RT) holds for each $\beta\in I(\GGamma')$, where $\GGamma'$ denotes the
  commensurator of $\LLambda$ in $\GGamma$.
\end{definition}

Property (RT) is of course always verified in the classical case since then
$\kappa_\alpha = 1$ for all $\alpha \in I(\GGamma)$. In general it does not hold for all corepresentations, as shown by the (non-commensurated)
Example~\ref{exple_free_kappa}.

\begin{lemma}
  For $[\alpha]\in I(\GGamma)/\LLambda$, $[\beta] \in \LLambda\bs I(\GGamma)$ we
  have
  \begin{equation}\label{eq_coeff}
    p_{[\alpha]} * p_{[\beta]} 
    = \frac{\qdim(\beta)}{\kappa_{\bar\beta}}
    {\sum_{\delta\in I(\GGamma)}}
    \frac{\qdim(\delta\otimes\bar\beta)_{[\alpha]}}
    {\qdim(\delta)}~  p_\delta,
  \end{equation}  
  where $(\delta\otimes\bar\beta)_{[\alpha]}$ denotes the span of the
  subobjects of $\delta\otimes\bar\beta$ isomorphic to elements of $[\alpha]$.
\end{lemma}

\begin{proof}
  Note first that in the proof of Lemma~\ref{lem_kappa} we showed that for each
  $\gamma, \delta, \mu \in I(\GGamma)$ we have
  \begin{displaymath}
    (h_Rp_\gamma \otimes p_\delta) \Delta(p_\mu) = \qdim(\gamma) \qdim(\mu)
    \qdim(\delta)^{-1} c_\mu^{\gamma,\delta} p_\delta,
  \end{displaymath}
  where $c_\mu^{\gamma,\delta} = \dim(\Hom(\mu, \gamma \otimes \delta))$.
  Applying the antipode to this formula yields
  \begin{displaymath}
    ( p_\delta \otimes h_Lp_\gamma) \Delta(p_\mu) = \qdim(\gamma) \qdim(\mu)
    \qdim(\delta)^{-1} c_{\bar\mu}^{\bar\gamma,\bar\delta} p_\delta.
  \end{displaymath}
  Thus for each $[\alpha]\in I(\GGamma)/\LLambda$,
  $[\beta] \in \LLambda\bs I(\GGamma)$ we obtain
  \begin{align*}
    p_{[\alpha]} * p_{[\beta]} &=
    \kappa_\alpha^{-1}\kappa_{\bar\beta}^{-1}
    (h_R p_{\bar\alpha}\otimes\id\otimes h_L p_{\bar\beta})\Delta^2(p_\LLambda) \\
    &= \kappa_\alpha^{-1}\kappa_{\bar\beta}^{-1}\ts \sum_{\delta\in I(\GGamma), \lambda\in I(\LLambda)}
      (h_R p_{\bar\alpha}\otimes p_{\delta}\otimes h_L p_{\bar\beta})\Delta^2(p_\lambda) \\
    &= \kappa_\alpha^{-1}\kappa_{\bar\beta}^{-1}
      \ts\sum_{\delta\in I(\GGamma), \lambda\in I(\LLambda), \gamma\in I(\GGamma)}
      (h_R p_{\bar{\alpha}} \otimes p_\delta)\circ \Delta \circ
      (p_\gamma \otimes h_L{p_{\bar{\beta}}})(\Delta(p_\lambda))	\\
    &= \kappa_\alpha^{-1}\kappa_{\bar\beta}^{-1}
      \ts\sum_{\delta, \lambda,\gamma}
      \qdim(\bar{\alpha})\qdim(\delta)^{-1} c_\gamma^{\bar\alpha,\delta}
    \qdim(\bar\beta) \qdim(\lambda) 
    c_{\bar\lambda}^{\beta,\bar\gamma} p_\delta \\&=
    \frac{\qdim(\alpha)}{\kappa_\alpha} \frac{\qdim(\beta)}{\kappa_{\bar\beta}}
    \sum_{\delta\in I(\GGamma)}
    \frac{\qdim(\bar\alpha\otimes\delta\otimes\bar\beta)_\LLambda}
    {\qdim(\delta)}~ p_\delta. 
  \end{align*}
  We used the identity
  $\sum_{\gamma\in I(\GGamma)}
  c_\gamma^{\bar\alpha,\delta}c_{\bar\lambda}^{\beta,\bar\gamma} =
  \sum_{\gamma\in I(\GGamma)}
  c_\gamma^{\bar\alpha,\delta}c_{\lambda}^{\gamma,\bar\beta} =
  \dim\Hom(\lambda,\bar\alpha\otimes\delta\otimes\bar\beta)$, so that adding the
  quantum dimensions $\qdim(\lambda)$ over $\lambda\in I(\LLambda)$ with these multiplicities yields
  $\qdim(\bar\alpha\otimes\delta\otimes\bar\beta)_\LLambda$.
  
  Notice that the term corresponding to $\delta$ vanishes unless
  $\delta \subset \alpha\otimes\lambda\otimes\beta$ for some
  $\lambda \in I(\LLambda)$. Moreover, decomposing into irreducibles
  $\delta\otimes\bar\beta = \bigoplus \cc_{\alpha'}^{\delta,\bar\beta} \alpha'$
  we can write, using Corollary~\ref{crl_kappa_indep}:
  \begin{align*}
    \qdim(\bar\alpha\otimes\delta\otimes\bar\beta)_\LLambda 
    &= {\ts\sum_{\alpha'}} \cc_{\alpha'}^{\delta,\bar\beta}
      \qdim(\bar\alpha\otimes\alpha')_\LLambda \\
    &= {\ts\sum_{\alpha' \in [\alpha]}} \cc_{\alpha'}^{\delta,\bar\beta} \kappa_\alpha
      \frac{\qdim(\alpha')}{\qdim(\alpha)}
      = \frac{\kappa_\alpha}{\qdim(\alpha)} \qdim(\delta\otimes\bar\beta)_{[\alpha]}
  \end{align*}
  and the formula follows.
\end{proof}

\begin{theorem} \label{thm_bounded} Let $\LLambda \subset \GGamma$ be an inclusion of discrete quantum groups. The operator $\Tdisc(b)$ is bounded with respect
  to $\|\cdot\|_{\GGamma/\LLambda}$ for every $b \in \Hh(\GGamma,\LLambda)$ iff the inclusion $\LLambda\subset \GGamma$ satisfies property (RT).
\end{theorem}

\begin{proof}
  Assume first that $\Tdisc(b)$ is bounded with respect to $\|\cdot\|_{\GGamma/\LLambda}$ for all $b \in \Hh(\GGamma,\LLambda)$. In particular $\Tdisc(p_\tau)$
  is bounded for every $\tau \in \LLambda\bs I(\GGamma')/\LLambda$.  Consider $\xi_{[\alpha]} = \kappa_\alpha^{1/2} p_{[\alpha]} / \qdim(\alpha)$, which has
  norm $1$ with respect to $\|\cdot\|_{\GGamma/\LLambda}$. By definition of the scalar product, and writing $\tau$ as a disjoint union of finitely many left
  classes $[\beta]$ we have
  \begin{align*}
    (\xi_{[\gamma]} \mid \xi_{[\alpha]} * p_\tau) &=
    \frac{\kappa_\alpha^{1/2}\kappa_\gamma^{1/2}}{\qdim(\alpha) \qdim(\gamma)} (p_{[\gamma]}|p_{[\alpha]} * p_\tau) 
    = \frac{\kappa_\alpha^{1/2}\kappa_\gamma^{1/2}}{\qdim(\alpha) \qdim(\gamma)} \kappa_\gamma^{-1} h_L(p_\gamma (p_{[\alpha]}*p_\tau)) \\
    &= \frac{\kappa_\alpha^{1/2}\kappa_\gamma^{1/2}}{\qdim(\alpha) \qdim(\gamma)} \kappa_\gamma^{-1} \sum_{[\beta] \subset\tau} h_L(p_\gamma (p_{[\alpha]}*p_{[\beta]})).
  \end{align*}
  Using~\eqref{eq_coeff} in the preceding lemma we can further compute:
  \begin{align*}
  (\xi_{[\gamma]} \mid \xi_{[\alpha]} * p_\tau) &= \sum_{[\beta]\subset\tau} \frac{\kappa_\alpha^{1/2}\kappa_\gamma^{1/2}}{\qdim(\alpha) \qdim(\gamma)} \kappa_\gamma^{-1} 
	\frac{\qdim(\beta)}{\kappa_{\bar\beta}}
      \frac{\qdim(\gamma\otimes\bar\beta)_{[\alpha]}}{\qdim(\gamma)} h_L(p_\gamma) \\
    &= \sum_{[\beta]\subset \tau} \frac{\qdim(\beta)}{\kappa_{\bar\beta}}
       \frac{\kappa_\alpha^{1/2}}{\qdim(\alpha)}
      \frac{\qdim(\gamma\otimes\bar\beta)_{[\alpha]}}{\kappa_\gamma^{1/2}}.
  \end{align*}
  Finally we decompose $(\gamma\otimes\bar\beta)$ into irreducible subobjects
  $\alpha'$ and select the ones in $[\alpha]$. Using 
  Corollary~\ref{crl_kappa_indep} this yields
  \begin{align*}
  (\xi_{[\gamma]} \mid \xi_{[\alpha]} * p_\tau)
                                                  &= \sum_{[\beta]\subset \tau} \frac{\qdim(\beta)}{\kappa_{\bar\beta}}
                                                    \sum_{\alpha'\in[\alpha]} \cc_{\alpha'}^{\gamma,\bar\beta}
                                                    \frac{\kappa_\alpha^{1/2}}{\qdim(\alpha)}
                                                    \frac{\qdim(\alpha')}{\kappa_\gamma^{1/2}} 
                                                   = \sum_{[\beta]\subset \tau} \frac{\qdim(\beta)}{\kappa_{\bar\beta}}
                                                    \sum_{\alpha'\in[\alpha]} \cc_{\alpha'}^{\gamma,\bar\beta}
                                                    \frac{\kappa_{\alpha'}^{1/2}}{\kappa_\gamma^{1/2}}.
  \end{align*}
  As a result we have, for any $\alpha'$, $\gamma \in I(\GGamma)$ and $\beta\in I(\GGamma')$ such
  that $\alpha' \subset \gamma\otimes\bar\beta$:
  \begin{displaymath}
    \frac{\kappa_{\alpha'}}{\kappa_\gamma} \leq 
    \frac {\kappa_{\bar\beta}^2 ~ \|\Tdisc(p_{\dc\beta})\|^2}{\qdim(\beta)^2} .
  \end{displaymath}
  This shows the existence of the constants $C_\beta$ and the direct
  implication.

  Conversely, assume that (RT) holds. Taking
  $b \in c_c(\LLambda\bs\GGamma'/\LLambda)$, we can assume that
  $b \in p_\tau c_c(\LLambda\bs\GGamma'/\LLambda)$ with
  $\tau \in \LLambda\bs I(\GGamma')/\LLambda$, and we then have a decomposition
  $b = \sum_{[\beta]\subset\tau} p_{[\beta]} b$ with $L(\tau)$ terms, where
  $[\beta] \in \LLambda\bs I(\GGamma)$. We first fix classes $[\alpha]$,
  $[\gamma] \in I(\GGamma)/\LLambda$. For
  $a \in p_{[\alpha]}c_c(\GGamma/\LLambda)$,
  $c \in p_{[\gamma]}c_c(\GGamma/\LLambda)$ we have
  \begin{align*}
    (c \mid a*b) &= \sum_{[\beta]\subset\tau} \kappa_{\bar\beta}^{-1} 
                   (c \mid (\id\otimes S(b_{\beta})h_L)\Delta(a)) \\
                 &= \sum_{[\beta]\subset\tau,\alpha'\in[\alpha]} \kappa_{\bar\beta}^{-1} \kappa_\gamma^{-1}
                   (h_Lc_\gamma^*\otimes S(b_\beta)h_L)\Delta(a_{\alpha'}).
  \end{align*}
  The sum is in fact finite since its terms vanish unless
  $\alpha'\subset \gamma\otimes\bar\beta$.

  Now we use the case when $\LLambda$ is the trivial subgroup, denoting
  $\|\cdot\|_\GGamma$ the hermitian norm on $c_c(\GGamma)$ and $\Tdisc_\GGamma$
  the homomorphism from $c_c(\GGamma)$ to $\End(c_c(\GGamma))$. We know that
  $\Tdisc_\GGamma(d)$ extends to a bounded operator on $\ell^2(\GGamma)$ for any
  $d\in c_c(\GGamma)$, because it is an operator from the right regular
  representation of $\GGamma$. Fix a choice of representatives $\beta$ for the
  classes $[\beta]$ and put
  $M = \max_{[\beta]\subset \tau} \|\Tdisc_\GGamma(b_\beta)\|$ (so that $M$
  might depend on the choice made). We then have
  \begin{displaymath}
    |(h_Lc_\gamma^*\otimes S(b_\beta)h_L)\Delta(a_{\alpha'})| =
    |(c_\gamma \mid \Tdisc_\GGamma(b_\beta)(a_{\alpha'}))_\GGamma| \leq
    M \|c_\gamma\|_\GGamma \|a_{\alpha'}\|_\GGamma.
  \end{displaymath}
  We have moreover
  $\|c_\gamma\|_\GGamma = \sqrt\kappa_\gamma\|c\|_{\GGamma/\LLambda}$,
  $\|a_{\alpha'}\|_\GGamma = \sqrt\kappa_{\alpha'}\|a\|_{\GGamma/\LLambda}$ by
  definition.

  As a result we can write:
  \begin{align*}
    |(c\mid a*b)| &\leq M \sum_{[\beta]\subset\tau}
                    \sum_{\alpha'\in[\alpha],\alpha'\subset\gamma\otimes\bar\beta}
                    \kappa_{\bar\beta}^{-1} 
                    \kappa_\gamma^{-1/2} \kappa_{\alpha'}^{1/2}
                    \|c\|_{\GGamma/\LLambda} \|a\|_{\GGamma/\LLambda} 
                    \leq K \|c\|_{\GGamma/\LLambda} \|a\|_{\GGamma/\LLambda},
  \end{align*}
  where
  $K = M\sum_{[\beta]\subset\tau} \kappa_{\bar\beta}^{-1} C_{\bar\beta}^{1/2}
  \dim(\beta)^2$ only depends on $b$ (and not on $[\alpha]$, $[\gamma]$, $a$,
  $c$). Here we used property~(RT) and the fact from
  Lemma~\ref{lem_subobjects} that $\gamma\otimes\bar\beta$ has at most
  $\dim(\beta)^2$ irreducible subobjects.

  Let use write $[\gamma] \# [\alpha]$ if there exist
  $c\in p_{[\gamma]}c_c(\GGamma/\LLambda)$,
  $a \in p_{[\alpha]}c_c(\GGamma/\LLambda)$ such that $(c\mid a*b)\neq 0$.  The
  arguments above also show that $[\gamma]\#[\alpha]$ implies
  $\gamma \subset \alpha\otimes\lambda\otimes\beta$ for some
  $\lambda \in I(\LLambda)$ and $\beta \in\tau$. Writing $\tau$ as a disjoint
  union of $R(\tau)$ classes $[\beta_i] \in I(\GGamma)/\LLambda$, this inclusion
  implies $\gamma \subset \alpha\otimes\beta_i\otimes\mu$ for some $i$ and
  $\mu\in I(\LLambda)$. Since $\alpha\otimes\beta_i$ contains at most
  $\dim(\beta_i)^2$ irreducibles, this shows that, once $[\alpha]$ is fixed, we
  can have $[\gamma]\#[\alpha]$ for at most $\sum \dim(\beta_i)^2$ classes
  $[\gamma]$. Similarly, since $(c\mid a*b) = (c*b^\sharp\mid a)$, if $[\gamma]$
  is fixed, we can have $[\gamma]\#[\alpha]$ for at most $\sum \dim(\beta'_j)^2$
  classes $[\alpha]$, where we have now written $\tau$ as a disjoint union of
  $L(\tau)$ classes $[\beta'_j] \in \LLambda\bs I(\GGamma)$. Let us denote by $N$
  the maximum of these two sums, which depends only on $\tau$.

  Finally we can use Cauchy-Schwarz to write, for any $a$,
  $b \in c_c(\GGamma/\LLambda)$:
  \begin{align*}
    |(c \mid a*b)| &\leq \sum_{[\gamma]\#[\alpha]} 
                     |(p_{[\gamma]}c \mid (p_{[\alpha]}a) *b)|
                     \leq K \sum_{[\gamma]\#[\alpha]} \|p_{[\gamma]}c\|_{\GGamma/\LLambda} 
                     \|p_{[\alpha]}a\|_{\GGamma/\LLambda} \\
                   &\leq K \Big(\sum_{[\gamma]\#[\alpha]} 
                     \|p_{[\gamma]}c\|_{\GGamma/\LLambda}^2\Big)^{1/2}
                     \Big(\sum_{[\gamma]\#[\alpha]} \|p_{[\alpha]}a\|_{\GGamma/\LLambda}^2\Big)^{1/2}
                     \leq KN \|c\|_{\GGamma/\LLambda} \|a\|_{\GGamma/\LLambda}.
  \end{align*}
  This shows that $\Tdisc(b)$ is bounded with $\|\Tdisc(b)\| \leq KN$.
\end{proof}

\subsubsection{Modular structure}

\begin{definition}
  The {\em canonical state} on $\Hh(\GGamma,\LLambda)$ is given by the formula
  \[\omega(f) := \epsilon(f) = (p_\LLambda \mid \Tdisc(f) p_\LLambda), \;\;\; f
    \in \Hh(\GGamma,\LLambda).\]
\end{definition}

Since the sesquilinear form $(\,\cdot\, \mid \,\cdot\,)$ is positive-definite,
$\omega$ is faithful, i.e.\ $\omega(f^\sharp*f) = 0$ implies $f=0$.  To
investigate the modular properties of $\omega$ we first construct a quantum
analogue of the classical modular function. We consider the restrictions of the
functionals $\mu$, $\nu$ introduced in Definition~\ref{def_invariant_measure} to
$c_c(\LLambda\bs\GGamma'/\LLambda)$. These forms are faithful by
Proposition~\ref{prp_quotient_description}, and so we can make the following definition.
\begin{definition} \label{def_modular_element} The {\em modular element}
  associated with $(\GGamma,\LLambda)$ is the unique element
  $\nabla \in c(\LLambda\bs\GGamma'/\LLambda)$ such that
  $\nu(a) = \mu(\nabla a)$ for all $a \in c_c(\LLambda\bs\GGamma'/\LLambda)$.
\end{definition}
More concretely, fix $\alpha\in I(\GGamma')$. Since $(a\mapsto p_\alpha a)$ is
faithful on $p_{\dc\alpha}c(\LLambda\bs\GGamma/\LLambda)$ we can consider the
positive linear forms $\mu_\alpha$, $\nu_\alpha$ defined on the
finite dimensional \Cst algebra $p_\alpha c(\LLambda\bs\GGamma/\LLambda)$ by the
following identities, for every
$a \in p_{\dc\alpha}c(\LLambda\bs\GGamma/\LLambda)$:
\begin{align*}
  \mu_\alpha(a_\alpha) &= \sum \{\kappa_\delta^{-1} h_L(a_\delta) \mid 
                         [\delta]\subset \dc\alpha, 
                         [\delta] \in I(\GGamma)/\LLambda \}, \\
  \nu_\alpha(a_\alpha) &= \sum \{\kappa_{\bar\delta}^{-1} h_R(a_\delta) \mid
                         [\delta]\subset\dc\alpha,
                         [\delta] \in \LLambda\bs I(\GGamma) \}.
\end{align*}
Then $\nabla$ is characterized by the identities
$\nu_\alpha(a_\alpha) = \mu_\alpha(\nabla_\alpha a_\alpha)$,
$a \in p_{\dc\alpha}c(\LLambda\bs\GGamma/\LLambda)$,
$\dc\alpha \in \LLambda\bs I(\GGamma')/\LLambda$. Note that, according to the
next lemma, we also have $\nu(a) = \mu(a\nabla)$. Naturally $\nabla = \nabla^*$.

\begin{lemma} \label{lem_modular_easy} We have $p_\LLambda\nabla = p_\LLambda$,
  $\sigma_t^R(\nabla) = \nabla$ for any $t \in \mathbb{R}$, $S(\nabla) = \nabla^{-1} = S^{-1}(\nabla)$.
\end{lemma}

\begin{proof}
  Since $h_R$ and $h_L$ are both $\sigma_t^R = \sigma_{-t}^L$ invariant, this is
  also the case of $\mu_\alpha$ and $\nu_\alpha$, hence we have
  $\nu_\alpha(a) = \nu_\alpha(\sigma^R_{-t}(a)) = \mu_\alpha(\nabla_\alpha
  \sigma^R_{-t}(a)) = \mu_\alpha(\sigma_t^R(\nabla_\alpha) a)$ and we conclude
  by uniqueness that $\sigma_t^R(\nabla) = \nabla$. This, together with the definition of $\mu$ and the fact that $(\sigma_t^R)_{t \in \mathbb{R}}$ is the modular automorphism group for $h_R$, also implies that
  $\mu(a\nabla) = \mu(\nabla a)$ for all
  $a \in c_c(\LLambda\bs\GGamma'/\LLambda)$. Similarly, noting that
  $\nu S = \mu = \nu S^{-1}$ we can write
  $\nu(S(\nabla)a) = \mu(S^{-1}(a)\nabla) = \mu(\nabla S^{-1}(a)) =
  \nu(S^{-1}(a)) = \mu(a)$ and we conclude that $S(\nabla) = \nabla^{-1}$. We
  proceed in the same way with $S^{-1}$.
\end{proof}

The next result shows that the operator $\nabla$ indeed plays the role of the modular operator for the canonical state on the Hecke algebra (or rather its relevant von Neumann algebraic completion).

\begin{theorem} \label{thm_modularautomorphism}
  The maps
  $\theta_t : a \mapsto \sigma_t^R(\nabla^{it}a) = \nabla^{it}\sigma_t^R(a)$, $t \in \mathbb{R}$, $a \in \Hh(\GGamma,\LLambda)$
  define a $1$-parameter group of $*$-automorphisms of $\Hh(\GGamma,\LLambda)$
  and $\omega$ is a $\theta$-KMS$_1$ state.
\end{theorem}

\begin{proof}
  Using the fact that $\epsilon$ is the counit and equalities~\eqref{eq_convol_left},
  \eqref{eq_convol_right} we obtain the following expressions:
  \begin{displaymath}
    \omega(a*b) = \sum_{[\alpha]\in I(\GGamma)/\LLambda} 
    \kappa_\alpha^{-1} h_R(S(a_\alpha) b) 
    = \sum_{[\beta]\in\LLambda\bs I(\GGamma)} 
    \kappa_{\bar\beta}^{-1} h_L(a S(b_\beta)).
  \end{displaymath}
  We can then compute, for $a$, $b\in c_c(\LLambda\bs\GGamma'/\LLambda)$:
  \begin{align*}
    \omega(a*b) &=  
                  \sum_{[\alpha]\in I(\GGamma)/\LLambda} \kappa_\alpha^{-1} 
                  h_L(a_\alpha S(b_{\bar\alpha})) 
                  = \sum_{\dc \alpha \in \LLambda\bs\GGamma/\LLambda} 
                  \mu_\alpha(a_\alpha S(b_{\bar \alpha})) \\
                &= \sum_{\dc \alpha \in \LLambda\bs\GGamma/\LLambda} 
                  \nu_\alpha(\nabla_\alpha^{-1} a_\alpha S(b_{\bar \alpha})) 
                  = \sum_{[\alpha] \in \LLambda\bs  I(\GGamma)} \kappa_{\bar\alpha}^{-1}
                  h_R(\nabla_\alpha^{-1} a_\alpha S(b_{\bar \alpha})) \\
                &= \sum_{[\alpha] \in I(\GGamma)/\LLambda} \kappa_\alpha^{-1}
                  h_R( S(b_\alpha)\sigma_i^R(\nabla^{-1} a))
                  = \omega(b * \sigma_i^R(\nabla^{-1} a)).
  \end{align*}
  We also have the following variant of the last step of the computation: since
  $h_R$ is $\sigma^R$-invariant and $\sigma^R$ commutes with $S$, it is easy to check that $h_R( S(b_\alpha)\sigma_i^R(\nabla^{-1} a)) = h_R(S(\nabla\sigma_{-i}^R(b_\alpha))a)$, and this yields $\omega(a*b) = \omega((\nabla \sigma_{-i}^R(b))* a)$.

  From these properties we first deduce:
  \begin{align*}
    \omega(a*b*c) &= \omega(b*c*\sigma_i^R(\nabla^{-1}a)) 
                    = \omega(c*\sigma_i^R(\nabla^{-1}a)*\sigma_i^R(\nabla^{-1}b)) \\
                  &= \omega (\nabla\sigma_{-i}^R(\sigma_i^R(\nabla^{-1}a)*
                    \sigma_i^R(\nabla^{-1}b))*c).
  \end{align*}
  By faithfulness of $\omega$ this yields
  $\sigma_i^R(\nabla^{-1}a)* \sigma_i^R(\nabla^{-1}b) =
  \sigma_i^R(\nabla^{-1}(a*b))$, hence by Proposition~\ref{prp_convol} we have
  $(\nabla^{-1} a)* (\nabla^{-1} b) = \nabla^{-1}(a*b)$. This implies
  $(\nabla^ka)*(\nabla^kb) = \nabla^k(a*b)$ for all $k \in \ZZ$ and, by the
  usual argument, $(\nabla^za)*(\nabla^zb) = \nabla^z(a*b)$ for all
  $z\in\CC$. It follows that the maps $\theta_t$ are multiplicative for the
  convolution product. They are also compatible with the involution since
  $\sigma_t^R S = S \sigma_t^R$ and
  $(\nabla^{it}a)^\sharp = \nabla^{it\sharp} a^\sharp = \nabla^{it} a^\sharp$
  for real $t$, using the property $S(\nabla^{-1}) = \nabla$.
\end{proof}

\begin{corollary}
  We have $\Delta(\nabla) = \nabla \otimes\nabla$.
\end{corollary}

\begin{proof}
  We use the property $(\nabla a)*(\nabla b) = \nabla (a*b)$  for $a$,
  $b \in c_c(\LLambda\bs\GGamma'/\LLambda)$, established in the
  proof of the previous proposition. Since $S(\nabla) = \nabla^{-1}$ we
  can write
  \begin{align*}
    \nabla^{-1}(\nabla a * \nabla b) &=
                                       \sum \kappa_\alpha^{-1} (h_R S(\nabla a_\alpha)\otimes \nabla^{-1})
                                       \Delta(\nabla b) \\
                                     &= \sum \kappa_\alpha^{-1} (h_R S(a_\alpha)\otimes\id) 
                                       [(\nabla^{-1}\otimes \nabla^{-1})\Delta(\nabla)\Delta(b)] \\
                                     &= (a*b) = \sum \kappa_\alpha^{-1} (h_R S(a_\alpha)\otimes\id)[\Delta(b)].
  \end{align*}
  Since
  $\Delta(c(\LLambda\bs\GGamma/\LLambda)) \subset
  \Mm(c_c(\LLambda\bs\GGamma)\otimes c_c(\GGamma/\LLambda))$ and
  $h_Rp_{\bar\alpha}$ is faithful on $p_{[\bar\alpha]}c_c(\LLambda\bs\GGamma)$
  by Proposition~\ref{prp_quotient_description}, we can conclude that
  $(\nabla^{-1}\otimes \nabla^{-1})\Delta(\nabla) = 1 \otimes 1$.
\end{proof}

We shall now give an explicit formula for the modular function $\nabla$ in terms
of the structure of the inclusion $\LLambda\subset \GGamma$. This will involve
quantum analogues $\tilde L_\alpha$, $\tilde R_\alpha$ of the counting functions
$L$, $R$ which arise from the interplay between the modular structure of the
Haar weight $h_R$ and the structure of the quantum quotient space
$\LLambda\bs\GGamma/\LLambda$.

For every $\alpha\in I(\GGamma)$, we have a unique $h_L$-preserving
(resp. $h_R$-preserving) conditional expectation from
$p_\alpha c(\GGamma) = B(H_\alpha)$ onto the subalgebra
$p_\alpha c(\LLambda\bs\GGamma/\LLambda)$. We consider the following related
maps:

\begin{definition}
  We denote by $E^L_\alpha$ (resp. $E^R_\alpha$) the unique map
  $c(\GGamma) \to p_{\dc\alpha} c(\LLambda\bs\GGamma/\LLambda)$ such that
  $h_L(p_\alpha E^L_\alpha(a) b) = h_L(p_\alpha a b)$ (resp.
  $h_R(p_\alpha E^R_\alpha(a) b) = h_R(p_\alpha a b)$) for all
  $b\in c(\LLambda\bs\GGamma/\LLambda)$, $a \in c(\GGamma)$.
\end{definition}

In the classical case, $E^L_\alpha(f)$ is the constant function on $\dc\alpha$,
equal to the value $f(\alpha)$. Let us record the following property of these
maps in connection with Woronowicz' modular element.

\begin{lemma}
  We have $E^R_\alpha(F^{-2}) = E^L_\alpha(F^2)^{-1}$.
\end{lemma}

\begin{proof}
  Recall that $h_R(a) = h_L(F^2a)$ for all $a\in c_c(\GGamma)$. In particular we
  have, for $a\in c_c(\LLambda\bs\GGamma/\LLambda)$ and $\alpha \in I(\GGamma)$:
  \begin{displaymath}
    h_R(a_\alpha) = h_L(F^2 a_\alpha) = h_L(E^L_\alpha(F^2)a_\alpha) = h_R(F^{-2}E^L_\alpha(F^2)a_\alpha) = h_R(E^R_\alpha(F^{-2})E^L_\alpha(F^2)a_\alpha),
  \end{displaymath}
  since $E^L_\alpha(F^2)a_\alpha \in p_\alpha
  c_c(\LLambda\bs\GGamma/\LLambda)$. As $h_R$ is faithful on
  $p_\alpha c_c(\LLambda\bs\GGamma/\LLambda)$, we can infer that
  $p_\alpha E^R_\alpha(F^{-2}) = (p_\alpha E^L_\alpha(F^2))^{-1}$ and we
  conclude by Proposition~\ref{prp_quotient_description}.
\end{proof}

\begin{definition}
  Fix $\alpha\in I(\GGamma')$ and choose elements $\delta_i \sim \alpha$
  (resp. $\epsilon_j\backsim\alpha$) such that $\dc\alpha$ is the disjoint union
  of the classes $[\delta_i]\in \LLambda\bs I(\GGamma)$ (resp.
  $[\epsilon_j]\in I(\GGamma)/\LLambda$). We define the following elements of
  $p_{\dc\alpha}c_c(\LLambda\bs\GGamma/\LLambda)$:
  \begin{displaymath}
    \tilde L_\alpha = 
    \sum_i \frac{\qdim(\delta_i)^2}{\kappa_{\bar\delta_i}} E_{\delta_i}^L(F^2), \quad
    \tilde R_\alpha = 
    \sum_j \frac{\qdim(\epsilon_j)^2}{\kappa_{\epsilon_j}}E_{\epsilon_j}^R(F^{-2}).
  \end{displaymath}
\end{definition}

\begin{remark}\label{rk_Ltilda}
  Note that we have $F\in c(\LLambda\bs\GGamma/\LLambda)$ iff $\LLambda$
  is unimodular (i.e.\ $F|_{\LLambda} = I$), since $\Delta(F) = F\otimes F$. In this case we have
  $E^L_\alpha(F^t) = E^R_\alpha(F^t) = p_{\dc\alpha}F^t$ for all $t\in\RR$, and hence,
  writing $F_{\dc\alpha} = p_{\dc\alpha}F$:
  \begin{displaymath}
    \tilde L_\alpha =  \left({\ts \sum_i \frac{\qdim(\delta_i)^2}{\kappa_{\bar\delta_i}}}\right) F_{\dc\alpha}^2, \quad
    \tilde R_\alpha =  \left({\ts \sum_j \frac{\qdim(\epsilon_j)^2}{\kappa_{\epsilon_j}}}\right) F_{\dc\alpha}^{-2}.
  \end{displaymath}
  Since $\qdim(\delta_i)^2/\kappa_{\bar\delta_i}$ only depends on the class
  $[\delta_i] \in \LLambda\bs I(\GGamma)$ we can drop the constraint
  $\delta_i\sim\alpha$ in the definition of $\tilde L_\alpha$ and we see in
  particular that $\tilde L_\alpha$, $\tilde R_\alpha$ only depend on
  $\dc\alpha$ in this case. If we have moreover
  $\kappa_\delta = \kappa_{\bar\delta}$ for all $\delta\in I(\GGamma)$, the terms
  $\qdim(\delta_i)^2/\kappa_{\bar\delta_i}$ only depend on
  $\dc{\delta_i} = \dc\alpha$ and hence we have
  \begin{displaymath}
    \tilde L_\alpha = L(\dc\alpha)~ \kappa_{\alpha}^{-1}\qdim(\alpha)^2 F_{\dc\alpha}^2, \quad
    \tilde R_\alpha = R(\dc\alpha)~ \kappa_{\alpha}^{-1}\qdim(\alpha)^2 F_{\dc\alpha}^{-2}.
  \end{displaymath}

  On the other hand let us consider the case when
  $p_{\dc\alpha} c_c(\LLambda\bs\GGamma/\LLambda) = \CC p_{\dc\alpha}$. Then we
  have $E^L_\alpha(F^2) = p_{\dc\alpha} = E^R_\alpha(F^{-2})$ (as for example $h_L(F_\alpha^2) = h_L(p_\alpha)$), so that
  \begin{displaymath}
    \tilde L_\alpha =  \left({\ts \sum_i \frac{\qdim(\delta_i)^2}{\kappa_{\bar\delta_i}}}\right) p_{\dc\alpha}, \quad
    \tilde R_\alpha =  \left({\ts \sum_j \frac{\qdim(\epsilon_j)^2}{\kappa_{\epsilon_j}}}\right) p_{\dc\alpha},
  \end{displaymath}
  with the same simplification as above if $\kappa_\delta = \kappa_{\bar\delta}$
  for all $\delta$.
\end{remark}

\begin{proposition} \label{prop_compute_nabla} We have
  $\nabla_\alpha = p_\alpha \tilde R_{\alpha}^{-1}E^L_\alpha(F^2)^{-1}\tilde
  L_{\alpha}$.  In particular, if $\LLambda$ is unimodular and
  $\kappa_\delta = \kappa_{\bar\delta}$ for all $\delta\in I(\GGamma)$, we
  obtain the ``semi-classical'' formula
  $\nabla_\alpha = \frac{L(\dc\alpha)}{R(\dc\alpha)} F_\alpha^2$.
\end{proposition}

\begin{proof}
  Take $\alpha \in I(\GGamma)$ and
  $a \in p_{\dc\alpha} c(\LLambda\bs\GGamma/\LLambda)$. We choose elements
  $\epsilon_j \in \dc\alpha$ such that $\dc\alpha$ is the disjoint union of the
  classes $[\epsilon_j] \in I(\GGamma)/\LLambda$ and $\epsilon_j\backsim\alpha$
  for all $i$.

  By Corollary~\ref{crl_indep} we have
  \begin{align*}
    \kappa_{\bar\epsilon_j}^{-1} h_L(a_{\epsilon_j}) = \kappa_{\bar\epsilon_j}^{-1} h_R(p_{\epsilon_j}F^{-2}a) = \kappa_{\bar\epsilon_j}^{-1}
    h_R(p_{\epsilon_j}E^R_{\epsilon_j}(F^{-2})a) = \kappa_{\bar\alpha}^{-1} h_R(p_\alpha E^R_{\epsilon_j}(F^{-2})a) .
  \end{align*}
  Recall moreover that we have
  $\kappa_{\bar\epsilon}/\kappa_{\bar\alpha} = (\qdim(\epsilon)/\qdim(\alpha))^2$ when
  $\epsilon\backsim\alpha$. Hence we can write
  \begin{align*}
    \mu_\alpha(a_\alpha) &= {\ts\sum_j} \kappa_{\epsilon_j}^{-1} h_L(a_{\epsilon_j})
                           = {\ts\sum_j} \frac{\kappa_{\bar\epsilon_j}} {\kappa_{\epsilon_j}}
                           \kappa_{\bar\epsilon_j}^{-1} h_L(a_{\epsilon_j}) 
                           =    {\ts\sum_j} \frac{\kappa_{\bar\epsilon_j}} {\kappa_{\epsilon_j}}
                           \kappa_{\bar\alpha}^{-1} h_R(E^R_{\epsilon_j}(F^{-2})a_{\alpha}) \\
                         &= \qdim(\alpha)^{-2} h_R(\tilde R_{\alpha} a_\alpha) = \qdim(\alpha)^{-2} h_L(F^2\tilde R_{\alpha} a_\alpha)
                           = \qdim(\alpha)^{-2} h_L(E^L_\alpha(F^2)\tilde R_{\alpha} a_\alpha).
  \end{align*}
  We proceed similarly on the other side with classes
  $[\delta_i]\in \LLambda\bs I(\GGamma)$ and $\delta_i\sim\alpha$ :
  \begin{align*}
    \nu_\alpha(a_\alpha) 
    &= {\ts\sum_i} \frac{\kappa_{\delta_i}} {\kappa_{\bar\delta_i}}
      \kappa_{\delta_i}^{-1} h_R(a_{\delta_i}) 
      = {\ts\sum_i} \frac{\kappa_{\delta_i}} {\kappa_{\bar\delta_i}}
      \kappa_{\delta_i}^{-1} h_L(F^2a_{\delta_i}) 
      = {\ts\sum_i} \frac{\kappa_{\delta_i}} {\kappa_{\bar\delta_i}}
      \kappa_{\delta_i}^{-1} h_L(E_{\delta_i}^L(F^2)a_{\delta_i}) \\
    &= {\ts\sum_i} \frac{\kappa_{\delta_i}} {\kappa_{\bar\delta_i}}
      \kappa_{\alpha}^{-1} h_L(E_{\delta_i}^L(F^2)a_{\alpha}) 
      = \qdim(\alpha)^{-2} h_L(\tilde L_{\alpha} a_{\alpha}).
  \end{align*}
  This yields the result by definition of $\nabla$.
\end{proof}

\subsection{Examples: HNN extensions}\label{sec_HNN}

Let $\GGamma_0$ be a discrete quantum group with two quantum subgroups
$\LLambda_\epsilon \subset \GGamma_0$ ($\epsilon=\pm 1$).  Following
\cite{Fima_HNN}, we start with an isomorphism between the two quantum subgroups,
described via a Hopf $*$-algebra isomorphism
$\theta : \CC[\LLambda_1] \to \CC[\LLambda_{-1}]$ and we form
$\GGamma = HNN(\GGamma_0,\theta)$. Recall that $\CC[\GGamma]$ is generated by
$\CC[\GGamma_0]$ and a group-like unitary $w$ such that
$w^\epsilon b w^{-\epsilon} = \theta^\epsilon(b)$ for
$b \in \CC[\LLambda_\epsilon]$. We denote by
$E_\epsilon : \CC[\GGamma_0] \to \CC[\LLambda_\epsilon]$ the canonical
conditional expectations. The algebra $\CC[\GGamma]$ is the direct sum of the
subspaces
\begin{displaymath}
  \CC[\GGamma]_n = \{x_0 w^{\epsilon_1}x_1 \cdots w^{\epsilon_n}x_n \mid x_i \in \CC[\GGamma_0], \epsilon_i = \pm 1, E_{\epsilon_i}(x_i) = 0 ~\text{whenever}~
  \epsilon_{i+1}\neq \epsilon_i\}. 
\end{displaymath}
The subspaces $\CC[\GGamma]_n$, $n\geq 1$ span the kernel of the canonical
conditional expectation $E_0 : \CC[\GGamma] \to \CC[\GGamma_0] = \CC[\GGamma]_0$
and the Haar state of $\GGamma$ is $h = h_0 \circ E_0$ --- see \cite{Fima_HNN}.
It follows in particular that $I(\GGamma)$ is partitioned into the subsets
\begin{align*}
  I(\GGamma)_n &= \{\alpha \subset \alpha_0\otimes w^{\epsilon_1}\otimes\alpha_1\otimes \cdots \otimes w^{\epsilon_n}\otimes\alpha_n \mid \\
               & \hspace{3cm} \alpha_i \in I(\GGamma_0), \epsilon_i = \pm 1, \alpha_i \notin I(\LLambda_{\epsilon_i}) ~\text{whenever}~ \epsilon_{i+1}\neq \epsilon_i\}. 
\end{align*}

\begin{proposition}\label{prop_HNN_Hecke}
  Assume that the quantum subgroups $\LLambda_\epsilon$ have finite index in
  $\GGamma_0$ and at least one of them is distinct from $\GGamma_0$. Then
  $\GGamma_0$ is commensurated in $\GGamma$, not normal, and of infinite index.
\end{proposition}

\begin{proof}
  Write
  $I(\GGamma_0)/\LLambda_\epsilon = \{[\gamma_{\epsilon,0}], \ldots,
  [\gamma_{\epsilon,p}]\}$ with $\gamma_{\epsilon,0} = 1$.  Then any
  $\alpha \in I(\GGamma)_n$ is contained in a representation
  $\gamma_{-\epsilon_1,k_1} \otimes w^{\epsilon_1}\otimes
  \gamma_{-\epsilon_2,k_2}\otimes w^{\epsilon_2}\otimes\cdots \otimes
  w^{\epsilon_n}\otimes\alpha_n$ with $k_i \neq 0$ if
  $\epsilon_i \neq \epsilon_{i+1}$. Indeed, starting from
  $\alpha \subset \alpha_0\otimes w^{\epsilon_1}\otimes\alpha_1\otimes \cdots
  \otimes w^{\epsilon_n}\otimes\alpha_n$ as previously, write
  $\alpha_0 \subset \gamma_{-\epsilon_1,k_1} \otimes \lambda$ with
  $\lambda \in I(\LLambda_{-\epsilon_1})$. Observe moreover that
  $\lambda\otimes w^{\epsilon_1} \simeq w^{\epsilon_1} \otimes
  \theta^{-\epsilon_1}(\lambda)$ and decompose
  $\theta^{-\epsilon_1}(\lambda) \otimes \alpha_1$ into irreducible subobjects
  $\alpha'_1$. Since $\alpha$ is irreducible it appears as a subobject of one of
  the corresponding corepresentations
  $\gamma_{-\epsilon_1,k_1}\otimes w^{\epsilon_1}\otimes\alpha'_1\otimes
  w^{\epsilon_2} \cdots \otimes w^{\epsilon_n}\otimes\alpha_n$. Moreover since
  $\theta^{-\epsilon_1}(\lambda) \in I(\LLambda_{\epsilon_1})$ we have
  $\alpha_1\notin I(\LLambda_{\epsilon_1})$ $\Rightarrow$
  $\alpha'_1\notin I(\LLambda_{\epsilon_1})$. Iterating the procedure we see
  that $\alpha'_1$ can also be chosen among the representatives
  $\gamma_{-\epsilon_2,k}$ etc.

  In particular it follows that $I(\GGamma)_n/\GGamma_0$ --- and similarly
  $\GGamma_0\bs I(\GGamma)_n$ --- is finite. Since $I(\GGamma)_n$ is clearly
  saturated with respect to the equivalence relations $\sim$, $\backsim$
  relative to $\GGamma_0$, this shows that $\GGamma_0$ is commensurated in
  $\GGamma$. It is never of finite index in $\GGamma$ since the subsets
  $I(\GGamma)_n$, $n\geq 1$, are non empty. Finally, assuming e.g.
  $\LLambda_1 \neq \GGamma_0$ and taking
  $\alpha \in I(\GGamma_0) \setminus I(\LLambda_1)$, then we get that
  $w\otimes\alpha\otimes w^*$ belongs to $I(\GGamma)_2$. If we had
  $c(\GGamma/\GGamma_0) = c(\GGamma_0\bs\GGamma)$, then $w\otimes\alpha$ would
  be equivalent to a corepresentation of the form $\beta\otimes w$ with
  $\beta\in I(\GGamma_0)$, but then $w\otimes\alpha\otimes w^* \simeq \beta$
  would belong to $I(\GGamma)_0$. Hence $\GGamma_0$ is not normal in $\GGamma$.
\end{proof}

Denote by $\Coker\GGamma\LLambda$ the weak closure of
$\{(\id\otimes\varphi)\Delta(a) \mid a\in c_0(\GGamma/\LLambda), \varphi\in
c_0(\GGamma/\LLambda)^*\}$ in $\ell^\infty(\GGamma)$. Recall from
\cite{KKSV_Furstenberg} that the action of $\GGamma$ on $\GGamma/\LLambda$ is
called faithful if we have $\Coker\GGamma\LLambda = \ell^\infty(\GGamma)$.

\begin{lemma}\label{lemma_fusion_faithful}
  Let $\LLambda \subset \GGamma$ be a quantum subgroup. Assume that for every
  non-trivial $\alpha\in I(\GGamma)$ we can find $\gamma_\alpha \in I(\GGamma)$
  such that no subobjects of $\alpha\otimes\gamma_\alpha$ belong to
  $[\gamma_\alpha] \in I(\GGamma)/\LLambda$. Then the action of $\GGamma$ on
  $\GGamma/\LLambda$ is faithful.
\end{lemma}

\begin{proof}
  We denote $p_0$ the minimal central projection $p_\alpha$ corresponding to the
  trivial corepresentation $\alpha = 1$. It is also the central support of the
  counit $\epsilon$ of $c_c(\GGamma)$. The condition on $\alpha$,
  $\gamma_\alpha$ can also be written
  $(p_\alpha\otimes p_{\gamma_\alpha})\Delta(p_{[\gamma_\alpha]}) = 0$. On the
  other hand we have
  $(p_0\otimes p_{\gamma_\alpha})\Delta(p_{[\gamma_\alpha]}) =
  (p_0\epsilon\otimes p_{\gamma_\alpha})\Delta(p_{[\gamma_\alpha]}) = p_0\otimes
  p_{\gamma_\alpha} p_{[\gamma_\alpha]} = p_0\otimes p_{\gamma_\alpha}$. For
  $\alpha\neq 1$ we denote
  $x_\alpha = \qdim(\gamma_\alpha)^{-2} (\id\otimes
  h_Lp_{\gamma_\alpha})\Delta(p_{[\gamma_\alpha]})$. We have
  $p_\alpha x_\alpha = 0$, $p_0 x_\alpha = p_0$ and $\|x_\alpha\|\leq 1$ since
  $\qdim(\gamma)^{-2} h_Lp_\gamma$ is a state. Introduce a total order on
  $I(\GGamma)$ and put $y_F = \prod_{\alpha \in F} x_\alpha$ for
  $F\subset I(\GGamma)$ finite, $1\notin F$. Then the net $(y_F)_F$ converges to
  $p_0$ in the weak topology. Indeed the elements of the net are uniformly
  bounded in norm, and for $F\subset I(\GGamma)$ finite, $1\notin F$, we have
  $p_0 y_F = p_0$ and $p_\beta y_F = 0$ as soon as $\beta\in F$. Since
  $x_\alpha \in \Coker\GGamma\LLambda$ for all $\alpha\in I(\GGamma)$,
  $\alpha\neq 1$, it follows that $p_0 \in \Coker\GGamma\LLambda$, which implies
  $\Coker\GGamma\LLambda = \ell^\infty(\GGamma)$ by
  \cite[Prop.~2.10]{KKSV_Furstenberg}.
\end{proof}

We still denote by $\theta : I(\LLambda_1) \to I(\LLambda_{-1})$ the map induced
by $\theta$ on irreducible representations.  We define by induction
$\Dom \theta^k\subset I(\GGamma_0)$, for $k\in\ZZ$, by putting
$\Dom\theta^0 = I(\GGamma_0)$ and
$\Dom \theta^{(n+1)\epsilon} = \{\alpha\in\Dom\theta^{n\epsilon} \mid
\theta^{n\epsilon}(\alpha)\in I(\LLambda_\epsilon)\}$ for $n\in\NN$,
$\epsilon = \pm 1$.

\begin{proposition}
  Assume that $\bigcap_{k\in\ZZ}\Dom \theta^k = \{1\}$. Then the action of
  $\GGamma$ on $\GGamma/\GGamma_0$ is faithful.
\end{proposition}

\begin{proof}
  We apply Lemma~\ref{lemma_fusion_faithful}. Take $\alpha\in I(\GGamma)$
  non-trivial. If $\alpha \notin I(\GGamma_0)$, we can just take
  $\gamma=1$. Hence we can assume $\alpha \in \GGamma_0$, $\alpha\neq 1$. By
  assumption there exists $\epsilon \in \{\pm 1\}$, $n\in\NN^*$ such that
  $\alpha\in\Dom \theta^{(n-1)\epsilon}$ but
  $\alpha\notin\Dom \theta^{n\epsilon}$. We take
  $\gamma_\alpha = w^{-n\epsilon}$. We have then
  $\alpha\otimes w^{-n\epsilon} = w^{-(n-1)\epsilon} \otimes
  \theta^{(n-1)\epsilon}(\alpha)\otimes w^{-\epsilon}$, which is irreducible and
  not equivalent to $w^{-n\epsilon}\otimes\beta$ with $\beta\in I(\GGamma_0)$ ---
  indeed $\beta$ is in $I(\GGamma)_0$ but
  $w^\epsilon\otimes \theta^{(n-1)\epsilon}(\alpha)\otimes w^{-\epsilon}$ is in
  $I(\GGamma)_2$ since
  $\theta^{(n-1)\epsilon}(\alpha)\notin I(\LLambda_\epsilon)$.
\end{proof}

Note that, in the classical case, if $K$ is a central subgroup of $\Gamma_0$
contained in $\bigcap_{k\in\ZZ}\Dom\theta^{k}$, then it acts trivially on
$\Gamma/\Gamma_0$.

Now we compute the modular function $\nabla$ on generators, using
Proposition~\ref{prop_compute_nabla}. Clearly $\nabla_\alpha = p_\alpha$ for
$\alpha\in I(\GGamma_0) \subset I(\GGamma)$. The value at $w$ is given as
follows in terms of $\GGamma_0$ and $\LLambda_\epsilon$:

\begin{proposition}
  Assume that the quantum subgroups $\LLambda_\epsilon$ have finite index in
  $\GGamma_0$ and denote
  $I(\GGamma_0)/\LLambda_{-1} = \{[\epsilon_0], \ldots, [\epsilon_p]\}$,
  $ \LLambda_1\bs I(\GGamma) = \{[\delta_0], \ldots, [\delta_q]\}$. Then we have
  $\nabla_w = p_w \tilde R_w^{-1}\tilde L_w$ with
  \begin{displaymath}
    \tilde L_w =
    \sum_{i=1}^q \frac{\qdim(\delta_i\otimes\bar\delta_i)}
    {\qdim(\delta_i\otimes\bar\delta_i)_{\LLambda_1}} p_{\dc w} \quad\text{and}\quad
    \tilde R_w =
    \sum_{j=1}^p \frac{\qdim(\bar\epsilon_j\otimes\epsilon_j)}
    {\qdim(\bar\epsilon_j\otimes\epsilon_j)_{\LLambda_{-1}}} p_{\dc w}.
  \end{displaymath}
\end{proposition}

\begin{proof}
  We have $\dc w = \bigcup_j [\epsilon_j\otimes w]$ and
  $\epsilon_j\otimes w\backsim w$, see the proof of
  Proposition~\ref{prop_HNN_Hecke}. Note that $\epsilon_j\otimes w$ is
  irreducible (whereas $\epsilon_j\otimes w\otimes\alpha$ needs not to be):
  indeed
  $\epsilon_j\otimes w\otimes w^*\otimes \bar\epsilon_j = \epsilon_j\otimes
  \bar\epsilon_j$ contains the trivial representation only once.  Moreover we
  claim that if $k\neq l$ then the classes $[\epsilon_k\otimes w]$,
  $[\epsilon_l\otimes w]$ are distinct, i.e.\
  $w^*\otimes \bar\epsilon_l\otimes\epsilon_k\otimes w$ has no subobject in
  $I(\GGamma_0)$. Indeed by definition $\bar\epsilon_l\otimes\epsilon_k$ has no
  subobject in $\LLambda_{-1}$, hence the irreducible subobjects of
  $w^*\otimes \bar\epsilon_l\otimes\epsilon_k\otimes w$ belong to
  $I(\GGamma)_2$.

  Then we can apply the definition of $\tilde L_w$, $\tilde R_w$ and
  Proposition~\ref{prop_compute_nabla}. Note that we are in the case when
  $p_{\dc w} c(\GGamma_0\bs\GGamma/\GGamma_0) = \CC p_{\dc w}$ since
  $\dim(w) = 1$, see Remark~\ref{rk_Ltilda}. For $\alpha = \epsilon\otimes w$,
  $\epsilon\in I(\GGamma_0)$ as above, we have
  $\kappa_\alpha = \qdim(w^*\otimes\bar\epsilon\otimes\epsilon\otimes
  w)_{\GGamma_0} = \qdim(\bar\epsilon\otimes\epsilon)_{\LLambda_{-1}}$ ---
  whereas
  $\kappa_{\bar \alpha} = \qdim(\epsilon\otimes\bar\epsilon) =
  \qdim(\bar\epsilon\otimes\epsilon)$. This gives the formula for $\tilde R_w$,
  and the one for $\tilde L_w$ follows by symmetry.
\end{proof}

\begin{example} \label{exple_HNN_profinite} One can construct quantum examples
  as follows. Take two finite quantum groups $\Sigma_{\pm 1}$, for instance
  duals of classical finite groups. Form the restricted product
  $\GGamma_0 = \prod'_{k\in\ZZ^*} \Sigma_{\sgn(k)}$, which is the dual of a
  profinite group if $\Sigma_{\pm 1}$ is the dual of a finite classical
  group. If one group $\Sigma_\epsilon$ is not classical, $\GGamma_0$ is a
  unimodular non-classical discrete quantum group. Consider the finite index
  subgroups
  $\LLambda_{\epsilon} = \prod'_{k\in\ZZ^*,k\neq\epsilon}\Sigma_{\sgn(k)}$. We
  have evident isomorphisms $\LLambda_\epsilon \simeq \GGamma$ obtained by
  shifting the copies of $\Sigma_\epsilon$ towards $k = 0$ in the restricted
  product. We denote by $\theta : \CC[\LLambda_1]\ \to \CC[\LLambda_{-1}]$ the
  corresponding isomorphism.

  Denoting $I_\epsilon = \Corep(\Sigma_\epsilon)$, we have natural
  identifications $I(\GGamma_0) \simeq \prod'_k I_{\sgn(k)}$,
  $\LLambda_\epsilon\bs I(\GGamma_0) = I(\GGamma_0)/\LLambda_\epsilon \simeq
  I_\epsilon$. For $\gamma\in I_{-1}$ we have
  $(\bar\gamma\otimes\gamma)_{\LLambda_{-1}} = \{1\}$ hence
  $\tilde R_w = \sum_{\gamma\in I_{-1}} \dim(\gamma)^2 = \#\Sigma_{-1}$ and
  similarly $\tilde L_w = \#\Sigma_1$, where we denote
  $\#\Sigma = \dim(c(\Sigma))$. As a result the modular function $\nabla$ of the
  Hecke pair $(\GGamma,\GGamma_0)$ is non trivial as soon as $\Sigma_1$,
  $\Sigma_{-1}$ have different dimensions/cardinals. If one of $\Sigma_{\pm 1}$
  is non classical (e.g. the dual of a non abelian finite group), the HNN
  extension $\GGamma$ is neither classical, nor co-classical (but it is
  unimodular).

  An element $\alpha = (\alpha_k)_{k\in\ZZ^*}$ of $I(\GGamma_0)$ is in
  $\Dom\theta^{n\epsilon}$, $n\in\NN^*$, iff we have
  $\alpha_\epsilon = \cdots = \alpha_{n\epsilon} = 1$. Hence
  $\bigcap_{k\in\ZZ}\Dom\theta^k = \{1\}$ and the action of $\GGamma$ on
  $\GGamma/\GGamma_0$ is faithful. Observe also that $\GGamma$ is finitely
  generated although $\GGamma_0$ is not: indeed, denoting $\Sigma^{(k)}$ the
  copy of $\Sigma_{\sgn(k)}$ in $\GGamma_0$ we have
  $\theta^{-\epsilon}(\Sigma^{(n\epsilon)}) = \Sigma^{((n+1)\epsilon)}$, so that
  $\GGamma$ is generated by $\Sigma^{(1)}$, $\Sigma^{(-1)}$ and $w$.
\end{example}

\begin{example} \label{exple_HNN_lie} One can also construct quantum examples by
  taking for $\GGamma_0$ the dual of a compact group $G$, and using quantum subgroups
  $\LLambda_\epsilon$ associated with quotients $H_\epsilon = G/K_\epsilon$. The
  index of $\LLambda_\epsilon$ in $\GGamma$ is finite iff $K_\epsilon$ is
  finite. If $G$ is connected, the subgroups $K_\epsilon$ must then be central,
  and we have $\# I(\GGamma_0)/\LLambda_\epsilon = \# K_\epsilon$.

  Assume that $G$ is a connected compact Lie group. Then the fundamental group
  of $H_\epsilon$ remembers the cardinality of the kernel $K_\epsilon$,
  and since we assume $H_1$ and $H_{-1}$ to be isomorphic we will always have
  $L(\dc w) = \# K_1 = \# K_{-1} = R(\dc w)$ in this case.  Similarly, subobjects of
  $\bar\gamma\otimes\gamma$ factor through the center $Z(G)$ for any
  $\gamma \in I(\GGamma_0) \subset \Rep(G)$, hence always belong to
  $I(\LLambda_\epsilon)$ so that we have
  $\qdim(\bar\gamma\otimes\gamma)_{\LLambda_\epsilon} =
  \qdim(\bar\gamma\otimes\gamma)$ and
  $\tilde L_w = \tilde R_w = (\# K_1) p_w$. Moreover, in most
  simple Lie groups the center is cyclic so that $\# K_\epsilon$ determines
  $K_\epsilon$ and $\LLambda_1 \simeq \LLambda_{-1}$ implies in fact
  $\LLambda_1 = \LLambda_{-1}$. This does not mean that the Hecke algebra will
  be completely trivial. One can also take for $\theta$ a non-inner automorphism
  of $H$ to make the construction more interesting, so that the resulting
  $\GGamma$ looks like a variant of the partial crossed-product construction.

  A typical case is given by
  $G = SU(n) \twoheadrightarrow H_1 = H_{-1} = PSU(n)$, to be compared with the
  ``classical case'' of the Baumslag-Solitar group
  $BS(n,n) = HNN(\ZZ, \id : n\ZZ \to n\ZZ)$.  On the other hand
  $Z(\mathrm{Spin}(4k)) = (\ZZ/2\ZZ)^2$, so that the dual of $SO(4k)$ can be
  realized in two different ways as a subgroup of the dual of
  $\mathrm{Spin}(4k)$. Of course one can also look at $SO(3)\times SU(2)$ which
  is a quotient of $SU(2)\times SU(2)$ in two different ways.  In all these
  cases we have $\nabla = 1$ because $\tilde L_w = \tilde R_w$.

  Note that since $K_\epsilon$ is contained in any maximal torus of $G$, the
  above construction is compatible with $q$-deformations -- $K_\epsilon$ remains
  a quantum subgroup of the compact quantum group $\GG_q$ corresponding to $G$.
  However, we still get $\nabla=1$, by essentially the same reasoning since the
  fusion ring of $\GG_q$ is the same as the one of $G$.
\end{example}

\section{Compact open quantum Hecke pairs}
\label{sec_compact_open_hecke}

In this section we introduce Hecke algebras in the setting of locally compact
(algebraic) quantum groups with compact open quantum subgroups. We then describe
a generalized Schlichting completion, which allows us to subsume the Hecke
algebras from section \ref{sec_discrete_hecke} in this setting, and deduce some
analytic consequences. Finally, we describe how to pass from arbitrary Hecke
pairs to their reduced versions, and exhibit some new examples of algebraic
quantum groups.

\subsection{Compact open Hecke algebras}

Let us fix an algebraic quantum group $\GG$ together with an algebraic quantum
subgroup $\HH\subset \GG$.  Recall that this is determined by a non-zero central
projection $p_\HH\in \Polc(\GG)$ such that
$\Delta(p_\HH)(1\otimes p_\HH) = p_\HH\otimes p_\HH$. Throughout this section we
normalize the Haar functionals of $\Polc(\GG)$ in such a way that
$\varphi(p_\HH) = \psi(p_\HH) = 1$.

The definition of compactly supported functions on the quantum homogeneous space
$\GG/\HH$ is easier than in the discrete case, since the relevant invariant
functions on $\GG$ are also compactly supported.  Namely, we define
$\Pold(\GG/\HH) = \Polc(\GG)^\HH = \{f\in\Polc(\GG) \mid \Delta(f) {(1\otimes
  p_\HH)} = f\otimes p_\HH \}$.  It is shown in
\cite{LandstadVanDaele_AlgSubgroups1} that this algebra is a direct sum of
matrix algebras, i.e. it corresponds to a ``discrete'' quantum space. We denote
by $c_0(\GG/\HH)$ the closure of $\Pold(\GG/\HH)$ in $C_0(\GG)$ and
$\ell^2(\GG/\HH)$ its closure in $L^2(\GG)$. It can be shown that $c_0(\GG/\HH)$
is a quantum homogeneous space in the sense of \cite{Vaes_Imprimitivity}, cf
\cite[Proposition 6.2, Theorem 6.4]{KalantarKasprzakSkalski_Open}.  One can
define $\Pold(\HH\bs\GG)$, $\Pold(\HH\bs\GG/\HH)$ exactly in the same way, as
well as the corresponding $c_0$ and $\ell^2$ spaces.

\bigskip

To define the Hecke algebra $\Hh(\GG,\HH)$ it suffices to restrict the natural
convolution product of $\Polc(\GG)$ to the space of $\HH$-biinvariant functions.
This product is transported from the dual multiplier Hopf algebra $\Dd(\GG)$ via
the Fourier transform $ \Ff: \Polc(\GG) \rightarrow \Dd(\GG)$ determined by
$(\Ff(f), h) = \varphi(hf)$. Explicitly we have, for $f$, $g\in\Polc(\GG)$:
\begin{displaymath}
  f * g = (f\varphi\otimes\id)( S^{-1}\otimes\id)\Delta(g) 
 = (\id\otimes \varphi S^{-1}(g))\Delta(f).
\end{displaymath} 
Together with the $*$-structure $f^\sharp:=\Ff^{-1}(\Ff(f)^*) $, not to be
confused with the given $*$-structure on $\Polc(\GG)$, this turns $\Polc(\GG)$
into a $*$-algebra. Explicitly we have $f^\sharp = S(f)^*\delta$, where
$\delta\in \Mm(\Polc(\GG))$ is the modular element of $\GG$.

The left regular representation of the dual algebra
$ \lambda: \Dd(\GG) \rightarrow B(L^2(\GG))$ is then given by
$\lambda(\Ff(f))(\Lambda(g)) = \Lambda(f * g)$, see
\cite[Section~4.2.2]{VoigtYuncken_Book}.  We also have the right regular
representation $\rho: \Dd(\GG) \rightarrow B(L^2(\GG))$ given by
$\rho(\Ff(f)) = \hat{J} \lambda(\Ff(f))^* \hat{J}$. Here $\hat{J}$ is the
modular conjugation operator for $\hat{\varphi}$, the dual left Haar weight also
given by the formula $\hat{\varphi}(\Ff(f)) = \epsilon(f)$.

Explicitly we have
$\rho(\Ff(f))(\Lambda(g)) = \Lambda(g * \hat{\sigma}_{-i/2}(f))$ for all
$f,g \in \Polc(\GG)$.  Here, by slight abuse of notation, we write
$ \hat{\sigma}_{-i/2}(f) $ instead of $ \Ff^{-1}(\hat{\sigma}_{-i/2}(\Ff(f)) $,
where $ (\hat{\sigma}_t)_{t \in \mathbb{R}} $ is the modular group of
$ \hat{\varphi} $. Note that the map $f \mapsto \hat{\sigma}_{-i/2}(f)$ is an
algebra isomorphism from $(\Polc(\GG), *)$ to $\Dd(\GG)$. We have
$\delta p_\HH = p_\HH$ because $\HH$ is compact and
$\hat{\sigma}_t(p_\HH) = p_\HH$ for all $t\in \mathbb{R}$ because the
restriction of $\hat{\varphi}: \Dd(\GG) \rightarrow \CC$ to $\Dd(\HH)$ is the
left Haar weight of $\hat{\HH}$.

\begin{lemma}
  We have $\Pold(\HH \bs \GG/\HH) = p_\HH * \Polc(\GG) * p_\HH$. In particular
  $\Pold(\HH \bs \GG/\HH)$ is closed under the convolution product $*$ and the
  involution $\sharp$.
\end{lemma}

\begin{proof} 
  Let $f\in \Polc(\GG)$. Using the definition of the convolution product we
  calculate
  $p_\HH *f=(p_\HH\varphi\otimes\id)(S^{-1}\otimes\id)\Delta(f)
  = (h\otimes\id)(\pi_\HH S^{-1}\otimes\id)\Delta(f) = (h\pi_\HH\otimes\id)\Delta(f)$, where $h$ is the Haar
  functional of $\Pol(\HH)$ and $\pi_\HH: \Polc(\GG) \rightarrow \Pol(\HH)$ the
  restriction map. It follows that $p_\HH * f = f$ iff $f \in
  \Pold(\HH\bs\GG)$. Similarly one checks $f\in \Pold(\GG/\HH)$ iff
  $f * p_\HH = f$.  Since $p_\HH$ is a projection in the convolution algebra
  this yields the claim.
\end{proof}

This allows us to give the following definition.

\begin{definition} 
  The Hecke algebra of $(\GG,\HH)$ is the $*$-algebra
  $\Hh(\GG,\HH) = \Pold(\HH \bs \GG/\HH)$ with the convolution product and
  $*$-structure $\sharp$ as above.
\end{definition}

We obtain a nondegenerate $*$-representation of $\Hh(\GG,\HH)$ on
$\ell^2(\GG/\HH)$ by considering the restriction of the right regular
representation.  Here we use that for $f \in \Pold(\HH \bs \GG/\HH)$ and
$g \in \Pold(\GG/\HH)$ the element $g * f$ is again contained in
$\Pold(\GG/\HH)$, so that $\rho(\Ff(f))$ indeed maps $\ell^2(\GG/\HH)$ to
itself. We also observe that $\rho\circ\Ff : \Hh(\GG, \HH) \rightarrow B(\ell^2(\GG/\HH))$
is anti-multiplicative for the convolution product, that is,
$\rho(\Ff(f * g)) = \rho(\Ff(g)) \rho(\Ff(f))$ for all $f, g \in \Hh(\GG,\HH)$.

In the same way as in the discrete case we view $\Pold(\GG/\HH)$ as a
$\Dd(\GG)$-module, see the discussion before Proposition \ref{prp_endo}, and
obtain the space $\End_\GG(\Pold(\GG/\HH))$ of $\Dd(\GG)$-module maps.

\begin{proposition} \label{prp_endotd} We have (mutually inverse)
  anti-multiplicative algebra isomorphisms
  \begin{align*}
    &T : \Hh(\GG,\HH) \rightarrow \End_\GG(\Pold(\GG/\HH)), \quad T(f)(h) = h * f \qquad \text{and} \\
    &T^{-1} : \End_\GG(\Pold(\GG/\HH)) \rightarrow \Hh(\GG, \HH) \cong \Pold(\GG/\HH)^\HH, \quad T^{-1}(F) = F(p_\HH).
  \end{align*}
\end{proposition}

\begin{proof} 
  Since $T(f)$ for $f \in \Hh(\GG, \HH)$ commutes with the left convolution
  action of $\Dd(\GG)$ it is clear that $T$ is well-defined, and it is obvious
  from the definition that $T$ is anti-multiplicative.

  Let us verify that $T$ is an isomorphism by verifying that the above formula
  for $T^{-1}$ yields indeed its inverse. Well-definedness is again easy to
  check, and using left $\HH$-invariance of
  $f \in \Hh(\GG, \HH) = \Pold(\HH \bs \GG/\HH)$ we get
  \begin{displaymath}
    (T^{-1} T)(f) = T(f)(p_\HH) = p_\HH * f = f.  
  \end{displaymath}
  Conversely, given $F \in \Hom_\GG(\Pold(\GG/\HH), \Pold(\GG/\HH))$ we compute
  \begin{align*}
    (T T^{-1})(F)(h) &= (T T^{-1})(F)(h * p_\HH) = h * (T T^{-1})(F)(p_\HH) \\
                                                &= h * p_\HH * T^{-1}(F) = h * F(p_\HH) = F(h * p_\HH) = F(h) 
  \end{align*}
  for all $h \in \Pold(\GG/\HH)$, using that $\Pold(\GG/\HH) \subset \Polc(\GG)$
  because $\HH \subset \GG$ is compact open.
\end{proof}

\subsection{Example: quantum doubles}

Let us consider the situation where $\HH$ is a compact quantum group and
$\GG = \HH \bowtie \hat{\HH}$ the quantum double of $\HH$. Recall that $\GG$ is
the locally compact quantum group given by the von Neumann algebra
$L^\infty(\GG) = L^\infty(\HH) \bar\otimes \Ll(\HH)$, equipped with the
coproduct
\begin{displaymath}
  \Delta_{\GG} = (\id \otimes \sigma \otimes \id)(\id \otimes \ad(W) \otimes \id)
  (\Delta \otimes \hat{\Delta}),
\end{displaymath}
where $\ad(W)$ is conjugation with the multiplicative unitary
$W \in L^\infty(\HH) \bar\otimes \Ll(\HH)$. This is a special case of the
generalized quantum doubles studied in \cite{BaajVaes}.

In fact, the quantum double of a compact quantum group is naturally an algebraic
quantum group in the sense of Van Daele, which allows us to give algebraic
descriptions of almost all the data involved \cite{DvD_DrinfeldDouble},
\cite{VoigtYuncken_Book}. More precisely, if we write $\Pol(\HH)$ for the
polynomial function algebra of $\HH$ as before and $\Dd(\HH)$ for the algebraic
convolution algebra, then $W\in \Mm(\Pol(\HH) \odot \Dd(\HH))$, and
$ \Polc(\GG)=\Pol(\HH) \odot\Dd(\HH)$, equipped with the comultiplication given
by the formula above, defines an algebraic quantum group.  It contains $\HH$
naturally as an algebraic compact open quantum subgroup, and the corresponding
group-like projection is $p_\HH=1 \otimes p_0$, where $p_0 \in \Dd(\HH)$ is the
central support of the counit. One checks that
\begin{align*}
  \Pold(\HH \bs\GG)=1\otimes\Dd(\HH) \qquad \Pold(\GG/\HH) &= W^*(1 \otimes \Dd(\HH)) W   
\end{align*}
inside $\Polc(\GG)$, and the space of $\HH$-biinvariant functions on $\GG$ is
\begin{displaymath}
  \Pold(\HH \bs\GG/\HH)\cong\{x\in \Dd(\HH) \mid W^*(1 \otimes x) W=1 \otimes x\}=Z(\Dd(\HH)), 
\end{displaymath}
i.e. $\Pold(\HH \bs\GG/\HH)$ identifies with the center of the algebraic
convolution algebra of $\HH$ with respect to its ordinary product.

By definition, the Hecke algebra $\Hh(\GG,\HH)$ is $\Pold(\HH\bs\GG/\HH)$
equipped with the restriction of the convolution product on $\Polc(\GG)$. In the
present situation, it is more convenient to describe the $*$-subalgebra
$\Ff(\Pold(\HH\bs\GG/\HH)) \subset \Dd(\GG)$ obtained via the Fourier transform
$\Ff: \Polc(\GG) \rightarrow \Dd(\GG), \Ff(f)(h)=\varphi(hf)$. As discussed in
\cite[Chapter 4]{VoigtYuncken_Book}, we can identify
$\Dd(\GG) = \Dd(\HH) \bowtie \Pol(\HH)$, which is the algebraic tensor product
$\Dd(\HH) \otimes \Pol(\HH)$ equipped with the twisted multiplication
\begin{displaymath}
  (x \bowtie f)(y \bowtie g) := x  y_{(2)} (y_{(1)}, f_{(1)}) \bowtie
  (\hat{S}(y_{(3)}), f_{(3)}) f_{(2)} g, 
\end{displaymath}
for $x,y \in \Dd(\HH)$, $f,g \in \Pol(\HH)$. The $*$-structure on $\Dd(\GG)$ is
defined in such a way that both $\Dd(\HH) \bowtie 1$ and $1 \bowtie \Pol(\HH)$
are $*$-subalgebras of $\Mm(\Dd(\GG))$, and the natural skew-pairing between
$\Dd(\GG)$ and $\Polc(\GG)$ is
\begin{displaymath}
  (y \bowtie g, f \otimes x) = (y,f) (g,x), \quad x,y \in\Dd(\HH), f,g \in \Pol(\HH),
\end{displaymath}
again following the conventions in \cite{VoigtYuncken_Book}.  The left and right
invariant Haar functional on $\Polc(\GG)$ is given by
$\varphi= \hat{h} \otimes h_R $, where $\hat{h}$ is the Haar state of
$\Pol(\HH)$ and $h_R$ the right Haar functional on $\Dd(\HH)$, compare
\cite[Proposition 4.19]{VoigtYuncken_Book}. As is well-known, using the Fourier transform $\Ff$ we
can identify $ Z(\Dd(\HH)) \subset (\Polc(\GG), *)$ with the
$*$-subalgebra
${(p_0 \bowtie 1)}{(1 \bowtie \Pol(\HH))}$ ${(p_0 \bowtie 1)}$ $= p_0
\bowtie \Pol(\HH)^{\ad} \subset \Dd(\GG)$,
where
\begin{align*}
  \Pol(\HH)^{\ad} &= \{f \in \Pol(\HH)\mid f_{(2)} \otimes \hat{h}(f_{(1)} S^{-1}(f_{(3)}))
                    = f \otimes 1\} \\ 
  &= \{f \in \Pol(\HH)\mid \Delta^{\mathrm{cop}}(f) = \Delta(f)\}, 
\end{align*}
with the product and $*$-structure induced from $\Pol(\HH)$.
This is precisely the algebra of characters inside $\Pol(\HH)$.

We note that, with suitable adjustments, similar computations go through for
generalized quantum doubles built out of compact and discrete quantum groups.

\subsection{The Schlichting completion}
\label{sec_schlichting}

Let $\GGamma$ be a discrete quantum group and $\LLambda \subset \GGamma$ be a
quantum subgroup. If $\LLambda \subset \GGamma$ is almost normal
(see Definition~\ref{def_commensurator}) we shall construct a pair $(\GG,\HH)$
consisting of an algebraic quantum group $\GG$ and a compact open quantum
subgroup $\HH \subset \GG$, playing the role of the Schlichting completion of
the Hecke pair $(\GGamma,\LLambda)$ \cite{Schlichting}.

More precisely, our strategy is as follows. We first define the algebra
$\Polc(\GG)$ as a subalgebra of $\ell^\infty(\GGamma)$ using the ``discrete''
Hecke convolution product, see Definition~\ref{def_schlichting}. The key point
of the construction consists then in proving that this algebra is a multiplier
Hopf $*$-algebra, and more specifically, that the coproduct takes its values in
the appropriate subspace of $\Mm(\Polc(\GG)\odot\Polc(\GG))$, see
Proposition~\ref{prop_coprod_schlichting}.

It is then easy to see that the
projection $p_\LLambda$ corresponds to a CQG algebra
$\Polc(\HH)\subset\Polc(\GG)$, and that $\Pold(\GG/\HH) = c_c(\GGamma/\LLambda)$
as subspaces of $\ell^\infty(\GGamma)$, see Propositions~\ref{prop_CQG}
and~\ref{prop_schlichting_homogeneous}. Using the Haar functional of $\Polc(\HH)$
and the $\GGamma$-invariant functional $\mu$ on $c_c(\GGamma/\LLambda)$ one can
then construct the integrals of $\Polc(\GG)$, so that $\GG$ is in fact a locally
compact quantum group by~\cite{KustermansVanDaele_CstarAlgebraic}.
We end the section by making the connection between the ``discrete'' and
``compact open'' Hecke algebras $\Hh(\GGamma,\LLambda)$ and $\Hh(\GG,\HH)$.

\begin{lemma}\label{lem_convol_support}
  We say that $\gamma \in I(\GGamma)$ is in the support of
  $x\in\ell^\infty(\GGamma)$ if $p_\gamma x\neq 0$. Then given
  $\gamma \in I(\GGamma)$, $a\in c_c(\GGamma/\LLambda)$ and
  $b\in c_c(\LLambda\bs\GGamma)$ such that $\gamma \in \Supp(a*b)$, there exist
  $\alpha\in\Supp(a)$, $\beta\in\Supp(b)$ and $\lambda\in I(\LLambda)$ such that
  $\gamma \subset \alpha\otimes\lambda\otimes\beta$.
\end{lemma}

\begin{proof}
  If $\gamma \in \Supp a*b$, at least one term of the sum \eqref{eq_convol_both}
  has $\gamma$ in its support. Hence there exist $\alpha\in\Supp(a)$,
  $\beta\in\Supp(b)$ such the corepresentation
  $\bar\alpha\otimes\gamma\otimes\bar\beta$ contains an element
  $\lambda \in I(\LLambda)$. By Frobenius reciprocity a non-zero morphism
  $\lambda \to \bar\alpha\otimes\gamma\otimes\bar\beta$ induces a non-zero
  morphism $\alpha\otimes\lambda\otimes\beta \to \gamma$ and we are done.
\end{proof}

Similarly one can define the support $\Supp(\varphi)$ of a linear functional
$\varphi\in c_c(\GGamma)^*$ as the set of elements $\gamma\in I(\GGamma)$ such
that $p_\gamma\varphi \neq 0$. If $\varphi$ has finite support, it extends
uniquely to a normal functional $\varphi \in \ell^\infty(\GGamma)_*$.

\begin{definition}\label{def_schlichting}
  Given a Hecke pair $(\GGamma, \LLambda)$ we denote by $\Polc(\GG)$
  (resp. $C_0(\GG)$) the subalgebra (resp. the \Cst subalgebra) of
  $\ell^\infty(\GGamma)$ generated by the elements $a*b$ with
  $a\in c_c(\GGamma/\LLambda)$, $b\in c_c(\LLambda\bs\GGamma)$.
\end{definition}

In the following lemmas we will always assume that $\Polc(\GG)$ and $C_0(\GG)$
arise in the above way from a Hecke pair. Note that by definition of the
convolution product, e.g.\ \eqref{eq_convol_left}, $a*b$ is a finite sum of
elements of $\ell^\infty(\GGamma)$ so that it is indeed in
$\ell^\infty(\GGamma)$. It is easy to check, using both expressions in
\eqref{eq_convol_left}, that $a^* * b^* = (a*b)^*$, so that $\Polc(\GG)$ is in
fact a $*$-subalgebra of $\ell^\infty(\GGamma)$ and $C_0(\GG)$ is its norm
closure.

In view of the definition of $a*b$ one can also say that the algebra
$\Polc(\GG)$ is generated by elements of the form
$(\id\otimes \varphi)\Delta(a)$ where $a\in c(\GGamma/\LLambda)$,
$\varphi \in c_c(\GGamma/\LLambda)^*$ have both finite support over
$I(\GGamma)/\LLambda$, i.e.\ $p_\tau a = 0$, $p_\tau\varphi = 0$ for all but a
finite number of right classes $\tau$. We denote $c_c(\GGamma/\LLambda)^\vee$
the space of these finitely supported functionals. It follows from this
description that the weak closure of $\Polc(\GG)$ in $\ell^\infty(\GGamma)$ is
the so-called cokernel of the $\GGamma$-action on $\GGamma/\LLambda$, which is
known to be a Baaj-Vaes subalgebra \cite[Definition~2.8,
Proposition~2.9]{KKSV_Furstenberg}. Below we prove a more precise,
$C^*$-algebraic version of this property, specific to the case of Hecke pairs.

For $x \in C_0(\GG)$ we can consider $\Delta(x)$ which is a priori an element of
$\ell^\infty(\GGamma) \bar\otimes \ell^\infty(\GGamma)$.

\begin{proposition}\label{prop_coprod_schlichting}
  We have $\Delta(\Polc(\GG)) (1\otimes \Polc(\GG))$,
  $\Delta(\Polc(\GG)) (\Polc(\GG)\otimes 1)$ $\subset$
  ${\Polc(\GG)\odot \Polc(\GG)}$ and $\Delta(C_0(\GG))(1\otimes C_0(\GG))$,
  $\Delta(C_0(\GG))(C_0(\GG)\otimes 1)$ $\subset C_0(\GG)\otimes C_0(\GG)$.
\end{proposition}

\begin{proof}
  The assertions about $C_0(\GG)$ follow by density from the ones for
  $\Polc(\GG)$, since $\Delta$ is continuous. Let us prove that
  $\Delta(\Polc(\GG)) (\Polc(\GG)\otimes 1) \subset {\Polc(\GG)\odot
    \Polc(\GG)}$. By multiplicativity it suffices to consider elements of the
  form $\Delta(x)(y\otimes 1)$ with $x = a*b$, $y = c*d$,
  $a,c \in c_c(\GGamma/\LLambda)$, $b,d \in c_c(\LLambda\bs\GGamma)$, and by
  linearity we can assume $a = p_{[\alpha]}a$, $b = p_{[\beta]}b$,
  $c = p_{[\gamma]}c$, $d = p_{[\delta]}d$ using left or right classes as
  appropriate.

  As a first step, consider a linear functional $\varphi\in c_c(\GGamma)^*$ with
  finite support and compute, using~\eqref{eq_convol_left} and coassociativity:
  \begin{displaymath}
    (\id\otimes\varphi)\Delta(x) = \kappa_\alpha^{-1}(h_RS(a_\alpha)\otimes\id\otimes \varphi)\Delta^2(b) = a*((\id\otimes\varphi)\Delta(b)).
  \end{displaymath}
  Since
  $\Delta(\ell^\infty(\LLambda\bs\GGamma))\subset
  \ell^\infty(\LLambda\bs\GGamma)\bar\otimes \ell^\infty(\GGamma)$ we have
  $(\id\otimes\varphi)\Delta(b) \subset \ell^\infty(\LLambda\bs\GGamma)$. Take
  moreover $\tau\in \LLambda\bs I(\GGamma)$ such that
  $(p_\tau\otimes\varphi)\Delta(b)\neq 0$. Then there exist $\nu\in\tau$,
  $\mu\in\Supp(\varphi)$ and $\lambda\in I(\LLambda)$ such that
  $\lambda\otimes\beta$ and $\nu\otimes\mu$ have a common irreducible
  subobject. By Frobenius reciprocity this implies
  $\nu\subset \lambda\otimes\beta\otimes\bar\mu$. Since $\Supp(\varphi)$ is
  finite, this shows that $\tau$ belongs to a finite subset of
  $\LLambda\bs I(\GGamma)$ (depending on $\Supp(\varphi)$ and $\beta$). Hence we
  have $(\id\otimes\varphi)\Delta(b) \in c_c(\LLambda\bs\GGamma)$ and
  $(\id\otimes\varphi)\Delta(x) \in \Polc(\GG)$.

  Second step. Observe that we have by~\eqref{eq_convol_right} and
  coassociativity:
  \begin{displaymath}
    \Delta(a*b)(y\otimes 1) = \kappa_{\bar\beta}^{-1}(\id\otimes\id\otimes h_LS^{-1}(b_\beta))(\id\otimes\Delta)[\Delta(a)(y\otimes 1)].    
  \end{displaymath}
  Recall that we have
  $\Delta(a)\in \ell^\infty(\GGamma)\bar\otimes
  \ell^\infty(\GGamma/\LLambda)$. Moreover, take $\tau \in I(\GGamma)/\LLambda$
  such that $\Delta(a)(y\otimes p_{\tau}) \neq 0$. Then there exist
  $\mu\in\Supp(y)$, $\nu\in\tau$, $\lambda\in I(\LLambda)$ such that
  $\alpha\otimes\lambda$ and $\mu\otimes\nu$ have a common irreducible
  subobject. By Frobenius reciprocity this implies
  $\nu\subset \bar\mu\otimes\alpha\otimes\lambda$. According to
  Lemma~\ref{lem_convol_support} there exists $\lambda'\in I(\LLambda)$ such
  that $\mu\subset \gamma\otimes\lambda'\otimes\delta$. Since double classes in
  $\LLambda\bs I(\GGamma)/\LLambda$ are finite unions of right classes it
  follows that $\tau$ belongs to a finite subset $P\subset I(\GGamma)/\LLambda$.

  Note that
  $\ell^\infty(\PP) := \sum_{\tau\in P}p_\tau\ell^\infty(\GGamma/\LLambda)$ is a
  finite dimensional subspace of $c_c(\GGamma/\LLambda)$. It follows that there
  is a finite family of vectors $t_i\in \ell^\infty(\PP)$ and elements
  $s_i \in \ell^\infty(\GGamma)$, such that
  $\Delta(a)(y\otimes 1) = \sum s_i\otimes t_i$. We have then, according to the
  above equation:
  $\Delta(a*b)(y\otimes 1) = \sum s_i\otimes (t_i*b) \in
  \ell^\infty(\GGamma)\odot (\ell^\infty(\PP)*b)$.

  Choose now a finite, linearly independent family of vectors $x_k = a_k*b$ in
  the finite dimensional subspace $\ell^\infty(\PP)*b \subset \Polc(\GG)$, and
  elements $z_k \in \ell^\infty(\GGamma)$ such that
  $\Delta(x)(y\otimes 1) = \sum z_k\otimes x_k$. Choose a corresponding family
  of linear forms with finite support $\varphi_k \in c_c(\GGamma)^*$ such that
  $\varphi_l(x_k) = \delta_{k,l}$ for all $k,l$; we have then
  $z_k = ((\id\otimes\varphi_k)\Delta(x))y$.
  Applying the first step to $\varphi_k$ we get
  $(\id\otimes\varphi_k)\Delta(x) \in \Polc(\GG)$, hence $z_k\in \Polc(\GG)$.
\end{proof}

Recall that the antipode of $\GGamma$ is well-defined as a map
$S : c(\GGamma) \to c(\GGamma)$ or $c_c(\GGamma) \to c_c(\GGamma)$, but not in
general from $\ell^\infty(\GGamma)$ to itself. It exchanges the subspaces
$c_c(\GGamma/\LLambda)$ and $c_c(\LLambda\bs\GGamma)$ of $c(\GGamma)$, which are
also subspaces of $\ell^\infty(\GGamma)$.

\begin{proposition} 
  We have $S(\Polc(\GG)) \subset \Polc(\GG)$. Equipped with the restriction of
  $\Delta$, $\Polc(\GG)$ (resp. $C_0(\GG)$) is a multiplier Hopf-$*$-algebra
  (resp. a bicancellative Hopf-$C^*$-algebra).
\end{proposition}

\begin{proof}
  For any $a\in c_c(\GGamma)/\LLambda$, $b\in c_c(\LLambda\bs\GGamma)$ we have
  $S(b^*)*S(a^*) = S((a*b)^*) = S(a^* * b^*)$, see the last part of the proof of
  Proposition~\ref{prp_convol} where bi-invariance of $a$, $b$ is not
  used. Replacing $a$, $b$ by their adjoints and using the fact that $S$
  exchanges $c_c(\GGamma/\LLambda)$ and $c_c(\LLambda\bs\GGamma)$ we see that
  $S$ stabilizes the canonical generating subspace of the algebra
  $\Polc(\GG)$. Since $S$ is antimultiplicative, it stabilizes $\Polc(\GG)$.

  We also obtain a character $ \epsilon: \Polc(\GG) \rightarrow \CC $ by
  restricting the counit of $ \ell^\infty(\GGamma)$. From the fact that
  $ \GGamma $ is a discrete quantum group we know that
  $ (\epsilon \otimes \id) \Delta = \id = (\id \otimes \epsilon)\Delta $ on the
  level of $\ell^\infty(\GGamma)$,
  which implies that $ \epsilon $ is a counit for $ \Polc(\GG)$. Using that
  $ S $ is an algebra anti-automorphism of $ \Polc(\GG) $ one then checks easily
  that $\Polc(\GG)$ is a multiplier Hopf $*$-algebra.
  Upon taking completions it follows that $ C_0(\GG) $ is a bicancellative Hopf
  $ C^*$-algebra.
\end{proof}

Note that the central projection $p_\LLambda\in\ell^\infty(\GGamma)$ belongs to
$\Polc(\GG)$ since $p_\LLambda * p_\LLambda = p_\LLambda$. It moreover satisfies
the property
$\Delta(p_\LLambda)(1\otimes p_\LLambda) = p_\LLambda\otimes p_\LLambda =
\Delta(p_\LLambda)(p_\LLambda\otimes 1)$, in $\ell^\infty(\GGamma)$ hence also
in $\Polc(\GG)$.

\begin{definition}
  We denote $p_\HH = p_\LLambda \in \Polc(\GG)$ and
  $\Pol(\HH) = p_\HH \Polc(\GG)$, and let $C(\HH)$ be the norm closure of
  $\Pol(\HH)$ in $\ell^\infty(\GGamma)$.
\end{definition}

We now describe the connection of the construction developed above to the
classical Schlichting completion, as described for example in
\cite{KLQ_Schlichting}.

\begin{proposition}
  Let $(\GGamma,\LLambda) = (\Gamma,\Lambda)$ be a classical discrete Hecke
  pair, with Schlichting completion $(G,H)$. The canonical map $\Gamma \to G$
  with dense image induces an embedding $C_0(G)\subset
  \ell^\infty(\Gamma)$. Under this identification we have $C_0(G) = C_0(\GG)$
  and $\Polc(G) = \Polc(\GG)$, and similarly $C(H)=C(\HH)$ and
  $ \Pol(H)=\Pol(\HH)$.
\end{proposition}

\begin{proof}
  Observe that $a*b$ for $a\in c_c(\Gamma/\Lambda)$,
  $b\in c_c(\Lambda\bs\Gamma)$ is right invariant under the action of the
  intersection of the point stabilizers in $ \Gamma $ of all points in the
  finite set $ \Supp(b) \subset \Lambda\bs\Gamma$. Since this group contains the
  intersection of finitely many point stabilizers of the action of $\Gamma$ on
  $\Gamma/\Lambda$ we see that $a*b\in \Pol_c(G)$, compare the description of
  the latter in \cite{KLQ_Schlichting}.  Hence we get
  $\Polc(\GG) \subset \Polc(G)$. Conversely, an element $g$ of $\Pol(G)$ can be
  written as finite sum of characteristic functions on $\Gamma/\Gamma_F$ for
  finite sets $F\subset \Gamma/\Lambda$.  Modulo left translation by $\Gamma$ we
  can assume that $g$ is the characteristic function of the point $\Gamma_F $ in
  $\Gamma/\Gamma_F$. We can write this as product of all $a_\gamma*f_\gamma$,
  where
  $ a_\gamma =\delta_{\gamma\Lambda} \in
  c_c(\GGamma/\LLambda),f_\gamma=\ev_{\gamma\Lambda} \in
  c_c(\GGamma/\LLambda)^*$ for $\gamma \in F$.  This yields the equality
  $\Polc(G) = \Polc(\GG)$.
	
  From the fact that both $C_0(G)$ and $C_0(\GG)$ are completions of
  $\Polc(G) = \Polc(\GG)$ inside $B(\ell^2(\Gamma))$ we get
  $C_0(G)=C_0(\GG)$. Finally, the claim about the canonical subgroups follows
  from $\Pol(H)=p_\Lambda \Polc(G)=p_\Lambda \Polc(\GG)=\Pol(\HH)$.
\end{proof}

We now return to the general setup of quantum Hecke pairs.

\begin{proposition} \label{prop_CQG}
  Equipped with the restriction $\Delta_\HH$ of $(p_\HH\otimes p_\HH)\Delta$,
  $\Pol(\HH)$ (resp. $C(\HH)$) is a CQG algebra (resp. a Woronowicz \Cst
  algebra).
\end{proposition}

\begin{proof}
  Recall that the comultiplication on $\Polc(\GG)$ is implemented on the Hilbert
  space level by conjugation with the multiplicative unitary for $\GGamma$,
  i.e.\ for $f \in \ell^\infty(\GGamma)$ we have
  $ \Delta(f)=W^*(1 \otimes f) W $ in
  $B(\ell^2(\GGamma) \otimes \ell^2(\GGamma))$. In particular, the
  comultiplication of $\Pol(\HH)$ extends continuously to a unital
  $*$-homomorphism $\Delta_\HH : C(\HH) \rightarrow C(\HH) \otimes C(\HH)$.
  Since we already know that $\Pol(\HH)$ is a Hopf $*$-algebra the cancellation
  conditions for $ C(\HH) $ are satisfied. Hence $C(\HH)$ is a Woronowicz \Cst
  algebra.
	
  This implies that there is a Haar functional on $\Pol(\HH)$, obtained by
  restricting from the Haar state of $C(\HH)$. We conclude that $\Pol(\HH)$ is a
  CQG algebra, compare \cite{KlimykSchmuedgen_Book}.
\end{proof}

So the corresponding compact quantum group $\HH$ is an ``algebraic''
compact open quantum subgroup of $\GG$, with restriction map induced by the
projection $p_\HH = p_\LLambda$. We denote by $h$ its Haar functional. We can
identify the corresponding homogeneous space
$\Pold(\GG/\HH)= \{a\in \Polc(\GG) \mid (1 \otimes p_\HH)(\Delta(a)) = a \otimes
p_\HH\}$ as follows.

\begin{proposition}\label{prop_schlichting_homogeneous}
  We have $\Pold(\GG/\HH) = c_c(\GGamma/\LLambda)$ as subspaces of
  $\ell^\infty(\GGamma)$. The coproduct restricts to an algebraic action
  $\Pold(\GG/\HH) \to \Mm(\Polc(\GG)\odot \Pold(\GG/\HH))$, in particular we
  have $\Delta(\Pold(\GG/\HH))$
  $(\Polc(\GG)\otimes 1) \subset \Polc(\GG)\odot \Pold(\GG/\HH)$.

  If $\mu$ is a $\GGamma$-invariant functional on $c_c(\GGamma/\LLambda)$, i.e.\
  ${(\id\otimes\mu)}$ ${((p_\alpha\otimes 1)\Delta(x))} = \mu(x)p_\alpha$ for
  all $x\in c_c(\GGamma/\LLambda)$ and $\alpha \in I(\GGamma)$, then $\mu$ is at
  the same time a $\GG$-invariant functional on $\Pold(\GG/\HH)$, i.e.\
  ${(\id\otimes\mu)} {((y\otimes 1)\Delta(x))} = \mu(x)y$ for all
  $x\in \Pold(\GG/\HH)$, $y\in \Polc(\GG)$.
\end{proposition}

\begin{proof}
  Take $a\in c_c(\GGamma/\LLambda)$. Then we have $a = a * p_\LLambda$ hence
  $a\in \Polc(\GG)$. By definition of $c(\GGamma/\LLambda)$ we have
  $(1\otimes p_\LLambda)\Delta(a) = a\otimes p_\LLambda$ hence
  $a\in \Pold(\GG/\HH)$. For the converse inclusion, take $x\in
  \Pold(\GG/\HH)$. In particular we have
  $(1\otimes p_\LLambda)\Delta(x) = x\otimes p_\LLambda$ so
  $x\in c(\GGamma/\LLambda)$. It remains to prove that $x$ has finite support in
  this algebra, i.e.\ $p_\tau x = 0$ for all but a finite number of classes
  $\tau\in I(\GGamma)/\LLambda$. It is clearly sufficient to prove this for
  $x = a*b$ with $a\in c_c(\GGamma/\LLambda)$, $b\in c_c(\LLambda\bs\GGamma)$,
  since taking products of such elements reduces the support. But this results
  from Lemma~\ref{lem_convol_support} and the Hecke condition: if
  $\gamma\in\Supp(a*b)$ then $\gamma \subset \alpha \otimes \mu$ with
  $\alpha\in\Supp(a)$ and $\mu\in \dc\beta$, $\beta\in\Supp(b)$; writing
  $\dc\beta$ as a finite union of right classes $[\gamma]$ and decomposing
  $\alpha\otimes\gamma$ into a finite number of irreducibles $\delta$ we see
  that $\Supp(a*b)$ in included in the union of the finite number of right
  classes $[\delta]$.

  For $x\in \Pold(\GG/\HH)$, $y\in \Polc(\GG)$ we have
  $\Delta(x)(y\otimes 1) \subset \Polc(\GG)\odot \Polc(\GG)$ by
  Proposition~\ref{prop_coprod_schlichting}, since $x\in \Polc(\GG)$. It remains
  to show that $z = (y\varphi\otimes\id)\Delta(x) \in \Pold(\GG/\HH)$ for any
  $\varphi \in c_c(\GGamma)^*$ with finite support. But we can write
  \begin{displaymath}
    (1\otimes p_\LLambda)\Delta(z) = (\varphi\otimes\id\otimes \id)((y\otimes
    1)\Delta\otimes\id)[(1\otimes p_\LLambda)\Delta(x)] 
  \end{displaymath}
  where all terms belong to the corresponding algebraic tensor products, and
  since ${(1\otimes p_\LLambda)}$ $\Delta(x) = x\otimes p_\LLambda$ we recognize
  $(1\otimes p_\LLambda)\Delta(z) = z\otimes p_\LLambda$. The last assertion is
  trivial because we can check the equality
  ${(\id\otimes\mu)} {((y\otimes 1)\Delta(x))} = \mu(x)y$ by multiplying by an
  arbitrary central projection $p_\alpha$.
\end{proof}

Note that $(\id\otimes h p_\HH)\Delta(x)$ is well-defined in $\Polc(\GG)$ for
any $x\in \Polc(\GG)$ since
$(\id\otimes p_\HH)\Delta(x) \in \Polc(\GG)\odot p_\HH\Polc(\GG) =
\Polc(\GG)\odot \Pol(\HH)$.

\begin{proposition}\label{prop_conditional_expect}
  For any $x\in \Polc(\GG)$ we have
  $(\id\otimes h p_\HH)\Delta(x) \in \Pold(\GG/\HH)$. If $\mu$ is a
  $\GG$-invariant functional on $\Pold(\GG/\HH)$ then
  $\varphi : x\mapsto \mu[(\id\otimes hp_\HH)\Delta(x)]$ defines a left
  invariant functional on $\Polc(\GG)$.
\end{proposition}

\begin{proof}
  Define a map $T : \Polc(\GG) \to \Polc(\GG)$ by setting
  $T(x):= (\id\otimes hp_\HH)\Delta(x)$, $x\in \Polc(\GG)$. Let us check that
  $T(x) \in \Pold(\GG/\HH)$. We have
  $(\Delta\otimes\id)\Delta(x) = (\id\otimes\Delta)\Delta(x)$ in
  $\ell^\infty(\GGamma)\bar\otimes \ell^\infty(\GGamma)\bar\otimes
  \ell^\infty(\GGamma)$.  Multiplying on the left by
  $1\otimes p_\HH\otimes p_\HH$ we obtain, since
  ${(p_\HH\otimes p_\HH)}$ ${\Delta(1-p_\HH)} = 0$:
  \begin{displaymath}
    (1\otimes p_\HH\otimes 1)(\Delta\otimes\id)((1\otimes p_\HH)\Delta(x)) =
    (\id\otimes\Delta_\HH)((1\otimes p_\HH)\Delta(x)).
  \end{displaymath}
  Note that both sides of the identity now lie in
  $\Polc(\GG)\odot \Pol(\HH)\odot\Pol(\HH)$. Applying $\id\otimes\id\otimes h$
  we obtain
  $(1\otimes p_\HH)\Delta(T(x)) = (\id\otimes p_\HH h)((1\otimes
  p_\HH)\Delta(x)) = T(x)\otimes p_\HH$, by invariance of $h$.

  On the other hand, starting again from the coassociativity relation and
  multiplying by $y\otimes 1\otimes p_\HH$ we get the following identity in
  $\Polc(\GG)\odot \Polc(\GG)\odot \Polc(\HH)$:
  \begin{displaymath}
    ((y\otimes 1)\Delta\otimes\id)((1\otimes p_\HH)\Delta(x)) =
    (\id\otimes(1\otimes p_\HH)\Delta)((y\otimes 1)\Delta(x)).
  \end{displaymath}
  Applying $\id\otimes\id\otimes h$ we obtain
  $(y\otimes 1)\Delta(T(x)) = (\id\otimes T)((y\otimes 1)\Delta(x))$. In other
  words, $T : \Polc(\GG)\to \Pold(\GG/\HH)$ is equivariant with respect to the
  left $\GG$-actions.
  
  So we can indeed define $\varphi = \mu\circ T$, and it is left invariant if
  $\mu$ is invariant on $\Pold(\GG/\HH)$:
  \begin{displaymath}
    y\mu T(x) = (\id\otimes \mu)((y\otimes 1)\Delta(T(x))) =
    (\id\otimes \mu T)((y\otimes 1)\Delta(x)),
  \end{displaymath}
  using the previous equivariance identity for $T$.
\end{proof}

\begin{theorem}
  Suppose that $(\GGamma, \LLambda)$ is a quantum Hecke pair. Then $\GG$
  introduced in Definition \ref{def_schlichting} is an algebraic quantum
  group. We call the pair $(\GG,\HH)$ the Schlichting completion of
  $(\GGamma, \LLambda)$.
\end{theorem}

\begin{proof}
  We apply Proposition~\ref{prop_conditional_expect} to the functional $\mu$ on
  $c_c(\GGamma/\LLambda)$ from Definition~\ref{def_scalar}, which is invariant
  by Proposition~\ref{prp_invariant_measure}, and defines at the same time an
  invariant functional on $\Pold(\GG/\HH)$ by
  Proposition~\ref{prop_schlichting_homogeneous}. We have
  $\varphi(p_\HH) = \mu(p_\LLambda) = 1$, hence $\varphi$ does not
  vanish. Finally it is positive because $\mu\otimes h$ is positive; indeed
  $\mu$ and $h$ have both \Cst algebraic realizations --- recall that $\mu$ is a
  sum of positive forms on each matrix factor of
  $\Pold(\GG/\HH) = c_c(\GGamma/\LLambda)$.
\end{proof}

\bigskip

Now we compare the Hecke algebras of a Hecke pair and of its Schlichting
completion. This will provide an analytic proof that the ``discrete'' Hecke
operators on $\ell^2(\GGamma/\LLambda)$ are bounded.

\begin{proposition} \label{prp_heckeident} Let $(\GG,\HH)$ be the Schlichting
  completion of a quantum Hecke pair $(\GGamma, \LLambda)$. Then we have canonical
  identifications
  \begin{displaymath}
    \Hh(\GG,\HH) \simeq \End_\GG(\Pold(\GG/\HH)) = 
    \End_\GGamma(c_c(\GGamma/\LLambda)) \simeq \Hh(\GGamma, \LLambda), 
  \end{displaymath}
  compatible with the multiplications. This identification is compatible with
  the $*$-structures if we identify a bi-invariant function
  $ f \in \Pold(\HH \bs \GG/\HH)$, viewed as element of $\Hh(\GG,\HH)$, with
  $\hat{\sigma}_{-i/2}(f) \in c_c(\LLambda \bs\GGamma/\LLambda)$, viewed as
  element of $\Hh(\GGamma, \LLambda)$.
\end{proposition}

\begin{proof}
  The first and last identifications are given by Propositions~\ref{prp_endotd}
  and~\ref{prop_adjoint} respectively. The identity in the middle is given by
  Proposition~\ref{prop_schlichting_homogeneous}.
\end{proof}

\begin{proposition}
  We have a canonical unitary isomorphism
  $\ell^2(\GGamma/\LLambda) \simeq \ell^2(\GG/\HH)$ induced by the equality
  $c_c(\GGamma/\LLambda) = \Pold(\GG/\HH)$ in $\ell^\infty(\GGamma)$. The Hecke
  algebra $\Hh(\GGamma,\LLambda)$ acts by bounded operators on
  $\ell^2(\GGamma/\LLambda)$.
\end{proposition}

\begin{proof} 
  The identification $c_c(\GGamma/\LLambda) = \Pold(\GG/\HH)$ in
  Proposition~\ref{prop_schlichting_homogeneous} is isometric since by
  construction the restriction of $\varphi$ to
  $\Pold(\GG/\HH) = c_c(\GGamma/\LLambda)$ coincides with $\mu$. Hence it
  induces a unitary isomorphism $\ell^2(\GGamma/\LLambda)\simeq
  \ell^2(\GG/\HH)$. Now the action of $f\in\Hh(\GGamma,\LLambda)$ on
  $c_c(\GGamma/\LLambda) = \Pold(\GG/\HH)$ agrees with the convolution on the right by
   $\hat{\sigma}_{i/2}(f) \in \Pold(\HH \bs \GG/\HH)$ by Proposition
  \ref{prp_heckeident}.  Hence it suffices to observe that the latter is
  obtained by restriction of the right regular representation of $\Dd(\GG)$ on
  $L^2(\GG)$, which acts by bounded operators.
\end{proof}

This, together with the Theorem \ref{thm_bounded}, yields an
analytical proof of Property (RT) from Definition~\ref{def_RT} for Hecke
pairs. It should be possible to give a categorical proof as well.

\begin{corollary}
  Property (RT) is satisfied by any Hecke pair $(\GGamma,\LLambda)$.
\end{corollary}

Finally we record the connection between the modular group of a discrete Hecke
pair $(\GGamma,\LLambda)$, see Theorem~\ref{thm_modularautomorphism}, and the
modular group of its Schlichting completion $\GG$.

\begin{proposition}
  The modular group of the canonical state $\omega$ on $\Hh(\GGamma,\LLambda)$
  agrees with the restriction of the modular group of $\hat{\varphi}$ to
  $\Ff(\Pold(\HH \bs \GG /\HH)) \subset \Dd(\GG)$ via the Fourier transform.
\end{proposition}

\begin{proof}
  Recall that the element in $\Hh(\GGamma,\LLambda)$ corresponding to
  $f\in \Hh(\GG,\HH)$ is $\tilde f := \hat\sigma_{-i/2}(f)$.
  We have then
  \begin{align*}
    \omega(\tilde f) &= (p_\LLambda \mid \tilde f) = (p_\HH \mid \hat\sigma_{-i/2}(f))
                       = \varphi(\pi_\HH(\hat{\sigma}_{-i/2}(f)) \\
                     &= \epsilon(\pi_\HH (\hat{\sigma}_{-i/2}(f))
                       = \epsilon(\hat{\sigma}_{-i/2}(f)) = \hat{\varphi}(\hat{\sigma}_{-i/2}(\Ff(f)))
                       = \hat{\varphi}(\Ff(f)).
  \end{align*} 
  Note that we use in the fourth equality the fact that
  $\pi_\HH(\hat{\sigma}_{-i/2}(f))$ is $\HH$-invariant, hence constant.
\end{proof}

\subsection{Reduction procedure}
\label{sec_reduction}

Starting from a discrete Hecke pair $(\GGamma,\LLambda)$, it can well happen
that the Schlichting completion $\GG$ is in fact discrete or even trivial. This
is connected to the faithfulness of the action of $\GGamma$ on
$\GGamma/\LLambda$ and to the reduction procedure that we describe now.

Suppose that $\GGamma$ is a discrete quantum group with a quantum subgroup
$\LLambda$ and the corresponding projection
$p_{\LLambda} \in \ell^\infty(\GGamma)$.  By $c_c(\GGamma/\LLambda)^\vee$ we
denote finitely supported functionals on $\GGamma/\LLambda$ as in the previous
subsection.

Recall from \cite[Definition~2.8]{KKSV_Furstenberg} that the cokernel of an action
of a discrete quantum group $\GGamma$ on a $C^*$-algebra $A = C_0(\XX)$, given by a
$*$-homomorphism $\alpha : C_0(\XX) \to M(c_0(\GGamma)\otimes C_0(\XX))$, is
defined as the following weak closure in $\ell^\infty(\GGamma)$:
\begin{equation}\label{eq_cokernel}
\CokerX\GGamma\XX = \{(\id\otimes\mu)\alpha(a) ; a\in A, \mu\in A^*\}''.
\end{equation}
Here we are concerned with the case $C_0(\XX) = c_0(\GGamma/\LLambda)$, with $\alpha$
being the appropriate restriction of $\Delta$. 

\begin{definition}
  We say that the pair $(\GGamma,\LLambda)$ is reduced if the canonical action
  of $\GGamma$ on $\GGamma/\LLambda$ is faithful, i.e.\ the
  cokernel $\Coker\GGamma\LLambda$ coincides with $\ell^\infty(\GGamma)$.
\end{definition}

Note that in the definition~\eqref{eq_cokernel} of the cokernel, when
$\XX = \GGamma/\LLambda$, we can also work with
$a\in\ell^\infty(\GGamma/\LLambda)$ and $\mu\in\ell^\infty(\GGamma/\LLambda)_*$,
or with $a\in c_c(\GGamma/\LLambda)$ and $\mu\in c_c(\GGamma/\LLambda)^\vee$.
In particular we see, following the discussion after
Definition~\ref{def_schlichting}, that a Hecke pair $(\GGamma,\LLambda)$ is
reduced if and only if $\GGamma$ embeds into its Schlichting completion $(\GG, \HH)$, i.e.\
the canonical map $\iota : \Polc(\GG) \to \ell^\infty(\GGamma)$ has strictly
dense image.

It is shown in \cite[Proposition 2.9]{KKSV_Furstenberg} that the cokernel
$\Coker\GGamma\LLambda$ is a Baaj-Vaes subalgebra of $\ell^\infty(\GGamma)$, so
that there exists a discrete quantum group $\tilde{\GGamma}$ such that
$\ell^{\infty}(\tilde{\GGamma})=\Coker\GGamma\LLambda$ (with the
comultiplication given simply by the restriction). Putting $a=p_\LLambda$ and
taking for $\mu$ the restriction of the counit $\epsilon$, we see that
$p_\LLambda$ belongs also to $\Coker\GGamma\LLambda$ and thus defines a quantum
subgroup $\tilde{\LLambda}$ of $\tilde \GGamma$ such that
$\ell^{\infty}(\tilde{\LLambda}) = p_\LLambda\ell^{\infty}(\tilde{\GGamma})$.

When $\LLambda$ is normal in $\GGamma$, it follows for example from the proof of
\cite[Theorem 2.11]{VaesVainerman} (or from results of
\cite{KalantarKasprzakSkalski_Open}) that
$\Coker\GGamma\LLambda = \ell^{\infty}(\GGamma/\LLambda)$. In this case it is
easy to see that $\tilde{\LLambda} = \{e\}$: as
$\ell^{\infty}(\GGamma/\LLambda)$ is the space of left slices of
$\Delta(p_\LLambda)$ by \cite[Theorem 3.3]{KalantarKasprzakSkalski_Open}, we have that
\begin{align*}
 p_\LLambda \ell^{\infty}(\GGamma/\LLambda) &= p_\LLambda \{(\omega \otimes \id
)\Delta(p_\LLambda) \mid \omega \in \ell^1(\GGamma)\}''\\ &= \{(\omega \otimes
\id )((1 \otimes p_\LLambda)\Delta(p_\LLambda)) \mid \omega \in
\ell^1(\GGamma)\}''= \CC p_\LLambda.
\end{align*}

\begin{proposition}\label{reductiondiscrete}
  Let $(\GGamma, \LLambda)$ and $(\tilde\GGamma, \tilde\LLambda)$ be as above.
  Then $(\tilde\GGamma, \tilde\LLambda)$ (called further the \emph{reduction} of
  $(\GGamma, \LLambda)$) is reduced,
  $\ell^{\infty}(\GGamma/\LLambda) =
  \ell^{\infty}(\tilde\GGamma/\tilde\LLambda)$,
  $\ell^{\infty}(\LLambda \backslash \GGamma) = \ell^{\infty}(\tilde\LLambda
  \bs\tilde\GGamma)$.  Moreover if
  $\theta : \ell^{\infty}(\GGamma/\LLambda)_+\to \RR_+$ is an nsf weight, then
  it is $\GGamma$-invariant iff it is $\tilde\GGamma$-invariant.
\end{proposition}

\begin{proof}
  Denote by $\tilde\alpha$ the action of $\tilde{\GGamma}$ on
  $\ell^\infty(\tilde\GGamma/\tilde\LLambda)$. Note that this is again given by
  the (suitable restriction of) the coproduct of $\ell^\infty(\GGamma)$. Thus to
  see that $(\tilde\GGamma, \tilde\LLambda)$ is reduced it suffices to show that
  $\Coker\GGamma\LLambda = \Coker{\tilde\GGamma}{\tilde\LLambda}$; this in turn will
  follow once we establish the latter part of the proposition.

  Note that
  \begin{displaymath}
    \ell^{\infty}(\tilde\GGamma/\tilde\LLambda)= \{a\in \Coker\GGamma\LLambda \mid
    (1\otimes p_\LLambda)\Delta(a) = p_\LLambda \otimes a\}.
  \end{displaymath}
  Thus it suffices to show that we have
  $\ell^{\infty}(\GGamma/\LLambda)\subset \Coker\GGamma\LLambda$. That however
  follows immediately as we can simply put $\mu=\epsilon$ in the definition of
  $\Coker\GGamma\LLambda$.

  The equality of the left coset spaces would follow in a similar manner once we
  observe that
  $\ell^{\infty}(\LLambda \backslash \GGamma) \subset
  \Coker\GGamma\LLambda$. This is true as $\Coker\GGamma\LLambda$ is
  $R$-invariant -- and
  $\ell^{\infty}(\LLambda \backslash \GGamma) = R
  \big(\ell^{\infty}(\GGamma/\LLambda)\big)$.
  The last statement is then obvious (the invariance condition is literally the
  same).
\end{proof}

Note that the inclusion
$\ell^{\infty}(\GGamma/\LLambda)\subset \Coker\GGamma\LLambda$ can be informally
understood as the classically obvious fact that the kernel of the action of
$\GGamma$ on $\GGamma /\LLambda$ is contained in $\LLambda$.

\begin{proposition}
  Let $(\GGamma, \LLambda)$ be as above and let
  $(\tilde\GGamma, \tilde\LLambda)$ be its reduction. Then $(\GGamma, \LLambda)$
  satisfies the Hecke condition if and only if $(\tilde\GGamma, \tilde\LLambda)$
  does, and if this is the case, the corresponding Hecke algebras are
  isomorphic.
\end{proposition}

\begin{proof}
  The first statement follows from Proposition~\ref{prop_hecke_orbits}, as the
  proposition above implies that the actions of $\LLambda$ on
  $\ell^{\infty}(\GGamma/\LLambda)$ and of $\tilde\LLambda$ on
  $\ell^{\infty}(\tilde\GGamma/\tilde\LLambda)$ are given by the same von
  Neumann algebraic morphism and the notion of finite orbits does not formally
  involve any quantum group structure. The second statement follows now from the
  identification of Hecke algebras as certain commutants with respect to these
  actions.
\end{proof}

We can now characterize the Hecke pairs that give rise to non-discrete
Schlichting completions as follows.

\begin{lemma}\label{lem_schlichting_discrete}
  Let $(\GGamma, \LLambda)$ be a Hecke pair, $(\tilde\GGamma,\tilde\LLambda)$
  its reduction and $(\GG, \HH)$ its Schlichting completion. Then $\GG$ is
  discrete if and only if $\tilde\LLambda$ is finite.
\end{lemma}

\begin{proof}
  Recall that we have by construction the strictly dense (equivalently,
  so-dense) inclusion $C_0(\GG) \subset \ell^\infty(\tilde\GGamma)$. Multiplying
  by $p_\LLambda = p_{\tilde\LLambda} = p_\HH$ we see that $C(\HH)$ is strictly
  dense (equivalently, so-dense) in $\ell^\infty(\tilde\LLambda)$, so that
  $\tilde\LLambda$ is finite if and only if $\HH$ is finite.

  Now if $\HH$ is finite, \cite[Proposition 4.5]{KalantarKasprzakSkalski_Open}
  implies that $\GG$ is discrete. On the other hand if $\GG$ is discrete, $\HH$,
  being an open (hence also closed by \cite[Theorem
  3.6]{KalantarKasprzakSkalski_Open}) quantum subgroup of $\GG$, must be discrete by
  \cite[Theorem 6.2]{DKSS_ClosedSubgroups}. Finally a discrete and compact
  quantum group must be finite.
\end{proof}

In particular if $(\GGamma,\LLambda)$ is a reduced Hecke pair, the associated
Schlichting completion is discrete if and only if $\LLambda$ is finite. This
shows, with the help of Lemma \ref{lemma_fusion_faithful}, that the examples of
quantum Hecke pairs discussed in the previous section lead to non-discrete
Schlichting completions.

\begin{corollary}
  The Schlichting completions associated with the HNN Hecke pairs of
  Example~\ref{exple_HNN_profinite} are non-discrete locally compact quantum
  groups, with non trivial modular group as soon as
  $\#\Sigma_1 \neq \#\Sigma_{-1}$.
\end{corollary}

Note that the scaling constant of these quantum groups equals $1$, since they
arise from algebraic quantum groups.

\bigskip

We address now the reduction procedure for compact open Hecke pairs.  Suppose
that $\GG$ is an algebraic quantum group and that $\HH$ is an algebraic compact
open quantum subgroup of $\GG$, given by a projection $p_\HH \in \Polc(\GG)$. We
shall check that the procedure described above again yields a reduction of the
pair $(\GG, \HH)$.

Consider the $*$-algebra generated as follows:
\begin{displaymath}
  \Cokeralg\GG\HH:= *{\text{-}}\mathrm{alg}\{(\id \otimes \mu)(\Delta(a)) \mid
  a \in \Pold(\GG/\HH), \mu \in \Pold(\GG/\HH)^\vee\},
\end{displaymath}
where $\Pold(\GG/\HH)^\vee$ denotes the finitely supported functionals on
$\Pold(\GG/\HH))$; note that as each of these can be written in the form $\nu b$
with $\nu \in \Pold(\GG/\HH))^\vee$ and $b \in \Pold(\GG/\HH)$, the formula
above makes sense. Note also that we can identify $\Pold(\GG/\HH))^\vee$ with
elements of the form $\psi b|_{\Pold(\GG/\HH)}$ for $b \in \Pold(\GG/\HH))$,
where $\psi$ is the right Haar weight of $\GG$. 

\begin{proposition} \label{prop_reduction_lc} The algebra $\Cokeralg\GG\HH$ defines
  an algebraic quantum group (with the structure inherited from $\Polc(\GG)$). Moreover
  $p_\HH \in \Cokeralg\GG\HH$ and $\Pold(\GG/\HH) \subset {\Cokeralg\GG\HH}$.
\end{proposition}

\begin{proof}
  We will use the map $T : \Polc(\GG)\to \Pold(\GG/\HH)$ given by the formula
  \begin{displaymath}
    T(a) = (\id \otimes h_{\HH}) (\Delta(a)(1 \otimes p_\HH)), \qquad a \in \Polc(\GG).
  \end{displaymath}
  Note that this map extends to a $C^*$-algebraic conditional expectation
  (preserving the right invariant Haar weight) from $C_0(\GG)$ onto
  $c_0(\GG/\HH)$, with the property $T(\Polc(\GG))= \Pold(\GG/\HH)$. This,
  together with \cite[Remarks, p.~342]{VanDaele_Algebraic} shows the following
  facts:
  \begin{align*}
    \Pold(\GG/\HH)^\vee &= \{\omega|_{\Pold(\GG/\HH)} \mid
                          \omega \in \widehat{\Polc(\GG)}\}, \quad\text{and}  \\
    \Pold(\GG/\HH)^\vee &= \{b\psi|_{\Pold(\GG/\HH)} \mid b \in \Pold(\GG/\HH)\} =
                          \{b\varphi|_{\Pold(\GG/\HH)}\mid b \in \Pold(\GG/\HH)\} \\
                        &= \{\varphi b|_{\Pold(\GG/\HH)}\mid b \in \Pold(\GG/\HH)\},
  \end{align*}
  where $\varphi$ is the left-invariant weight.

  Using the properties of $T$ and the fact that $\Polc(\GG)$ is a multiplier
  Hopf algebra one can show the following fact:
  $\Polc(\GG) \odot \Pold(\GG/\HH) = \Delta(\Pold(\GG/\HH))(\Polc (\GG) \otimes
  1)$. This implies (via the arguments of \cite[Proposition
  4.2]{VanDaele_Algebraic}) that convolving a functional
  $\omega \in \widehat{\Polc(\GG)}$ and a functional
  $\mu \in \Pold(\GG/\HH)^\vee$ yields
  $\omega \star \mu \in \Pold(\GG/\HH)^\vee$.

  We need to show that for every $a, b \in \Cokeralg\GG\HH$ we have (for
  example)
  $\Delta(a) (b \otimes 1) \in \Cokeralg\GG\HH \otimes \Cokeralg\GG\HH$. To this
  end it suffices to prove that for all $\omega \in \widehat{\Polc(\GG)},$
  $a \in \Pold(\GG/\HH)$ and $\mu \in \Pold(\GG/\HH)^\vee$ the elements
  $(\omega \otimes \id)\big(\Delta((\id \otimes \mu)(\Delta(a)))(b\otimes
  1)\big)$ and
  $(\id \otimes \omega)\big(\Delta((\id \otimes \mu)(\Delta(a))(b \otimes
  1)\big)$ belong to $\Cokeralg\GG\HH$ ;
  and the latter amount to noting that $\Pold(\GG/\HH)$ is right-invariant (for
  the first expression) and exploiting the convolution statement of the previous
  paragraph (for the second expression).

  Furthermore, $\Cokeralg\GG\HH$ is $S$-invariant. This is an easy consequence
  of the strong invariance of the left Haar weight, which says that for all
  $a, b \in \Polc(\GG)$ we have
  \begin{displaymath}
    S \big((\id \otimes \phi)(\Delta(a^*)(1 \otimes b)\big)  =
    (\id \otimes \phi)((1\otimes a^*)(\Delta(b)),
  \end{displaymath}
  combined with the statements in the beginning of the proof. This suffices to
  complete the proof that $\Cokeralg\GG\HH$ is a multiplier Hopf $*$-algebra,
  and in fact an algebraic quantum group, as we can just use the invariant
  weights of $\Polc(\GG)$.

  Then it suffices to show that as we have $\epsilon(p_\HH)=1$ we also have
  $\epsilon = \epsilon p_\HH$, so that $\epsilon|_{\Pold(\GG/\HH)}$ is finitely
  supported. This implies that $\Pold(\GG/\HH) \subset \Cokeralg\GG\HH$.
\end{proof}

\begin{definition}
  Let $(\GG, \HH)$ be as above, denote the algebraic quantum group
  corresponding to $\Cokeralg\GG\HH$ by $\tilde\GG$, and its compact quantum
  subgroup given by $p_\HH \in \Cokeralg\GG\HH$ by $\tilde\HH$. We call the pair
  $(\tilde\GG, \tilde\HH)$ the reduction of $(\GG, \HH)$ and say that a pair
  $(\GG, \HH)$ as above is reduced if $\Cokeralg\GG\HH = \Polc(\GG)$.
\end{definition}

\begin{proposition}
  Let $(\GG, \HH)$ and $(\tilde\GG, \tilde\HH)$ be as above.  Then
  $(\tilde\GG, \tilde\HH)$ is reduced,
  $\Pold(\GG/\HH) = \Pold(\tilde\GG/\tilde\HH)$ and
  $ \Pold(\HH \bs \GG) = \Pold(\tilde\HH\bs\tilde\GG)$.  Moreover a functional
  $\theta: \Pold(\GG/\HH) \to \CC$ is $\GG$-invariant iff it is
  $\tilde{\GG}$-invariant.
\end{proposition}

\begin{proof}
  Follows exactly the same lines as in Proposition \ref{reductiondiscrete}.
\end{proof}

Observe that the Schlichting completion is constructed specifically so that the
resulting pair $(\GG, \HH)$ is reduced. Further we will call $(\GG, \HH)$ a
\emph{Schlichting pair} whenever $\GG$ is an algebraic quantum group, $\HH$ is an
algebraic compact open quantum subgroup of $\GG$ and the pair $(\GG, \HH)$ is
reduced.

Suppose that we have two locally compact quantum groups $\GG_1$, $\GG_2$ with
respective open quantum subgroups $\HH_1$, $\HH_2$ corresponding to projections
$P_1\in C_b(\GG_1)$, $P_2\in C_b(\GG_2)$.  We say that a morphism from $\GG_1$
to $\GG_2$, described via a Hopf-$C^*$-algebra morphism
$\pi : C_0(\GG_2) \to C_b(\GG_1)$, maps $\HH_1$ to $\HH_2$ if
$\pi(P_2)\geq P_1$. One may check, using \cite[Corollary
3.8]{KalantarKasprzakSkalski_Open} that indeed one obtains then (by restriction
and multiplying by $P_1$) a quantum group morphism from $\HH_1$ to $\HH_2$.

The following abstract characterization of the Schlichting completion for
classical groups appears in \cite[Proposition~4.1]{Tzanev}. 
The injectivity of the map $\iota'$ corresponds in the classical case to the
density of the image of $\Gamma$ in $G'$, and the identity
$\iota'(p_{H'}) = p_\Lambda$, to the fact that $\Lambda$ is the preimage of $H'$.

\begin{proposition}\label{prop_schlichting_univ}
  Let $(\GGamma, \LLambda)$ be a Hecke pair and $(\GG, \HH)$ its Schlichting
  completion, with the canonical embedding
  $\iota : \Polc(\GG)\to \ell^\infty(\GGamma)$ defining the morphism from
  $\GGamma$ to $\GG$. Then for any other Schlichting pair $(\GG', \HH')$ and any
  morphism from $\GGamma$ to $\GG'$ mapping $\LLambda$ to $\HH'$ and given by
  an {\em injective} map $\iota' : \Polc(\GG')\to \ell^\infty(\GGamma)$, there
  exists a unique morphism from $\GG$ to $\GG'$, described by a map
  $\sigma : \Polc(\GG') \to \Polc(\GG)$, such that
  $\iota \circ \sigma = \iota'$. If in addition we assume that
  $\iota'(p_{\HH'}) = p_\LLambda$ then the morphism from $\GG$ to $\GG'$ is an
  isomorphism.
\end{proposition}

\begin{proof}
  A moment of thought shows that it suffices to show that
  $\iota'(\Polc(\GG')) \subset \iota (\Polc(\GG))$. Ignoring the injective
  embedding maps, and using the fact that both $(\GG, \HH)$ and $(\GG', \HH')$
  are Schlichting pairs, it suffices to note that
  $\Pold(\GG'/\HH') \subset \Pold(\GG/\HH)$. But this follows as we have
  $P\leq P'$ (again viewing both as projections in $\ell^{\infty}(\GGamma))$.
  The second part follows similarly.
\end{proof}

\subsection*{Acknowledgments} 
A.S. was partially supported by the National Science Center (NCN) grant no.~2020/39/I/ST1/01566.  R.V. was partially supported by the Agence Nationale de la
Recherche (ANR) grant ANR-19-CE40-0002 and the CEFIPRA project 6101-1.  C.V. was supported by EPSRC grant EP/T03064X/1.
This projected was started and concluded during visits of C.V.\ and R.V.\ at IMPAN Warsaw. The second and third author would like to thank the first author for the kind hospitality.

\bibliographystyle{alpha}
\bibliography{hecke}

\begin{thebibliography}{DCKSS18}

\bibitem[AD14]{Anantharaman_Approx}
Claire Anantharaman-Delaroche.
\newblock Approximation properties for coset spaces and their operator
  algebras.
\newblock In {\em The varied landscape of operator theory}, volume~17 of {\em
  Theta Ser. Adv. Math.}, pages 23--45. Theta, Bucharest, 2014.

\bibitem[BC95]{BostConnes}
Jean-Beno\^it Bost and Alain Connes.
\newblock Hecke algebras, type {III} factors and phase transitions with
  spontaneous symmetry breaking in number theory.
\newblock {\em Selecta Math. (N.S.)}, 1(3):411--457, 1995.

\bibitem[BV05]{BaajVaes}
Saad Baaj and Stefaan Vaes.
\newblock Double crossed products of locally compact quantum groups.
\newblock {\em J. Inst. Math. Jussieu}, 4(1):135--173, 2005.

\bibitem[DCKSS18]{DCKSS_DiscreteActions}
Kenny De~Commer, Pawe\l{} Kasprzak, Adam Skalski, and Piotr~M. So\l{}tan.
\newblock Quantum actions on discrete quantum spaces and a generalization of
  {C}lifford's theory of representations.
\newblock {\em Israel J. Math.}, 226(1):475--503, 2018.

\bibitem[DCY13]{DeCommerYamashita}
Kenny De~Commer and Makoto Yamashita.
\newblock Tannaka-{K}re\u{\i}n duality for compact quantum homogeneous spaces.
  {I}. {G}eneral theory.
\newblock {\em Theory Appl. Categ.}, 28:No. 31, 1099--1138, 2013.

\bibitem[DKSS12]{DKSS_ClosedSubgroups}
Matthew Daws, Pawe\l{} Kasprzak, Adam Skalski, and Piotr~M. So\l{}tan.
\newblock Closed quantum subgroups of locally compact quantum groups.
\newblock {\em Adv. Math.}, 231(6):3473--3501, 2012.

\bibitem[DVD04]{DvD_DrinfeldDouble}
Lydia Delvaux and Alfons Van~Daele.
\newblock The {D}rinfeld double of multiplier {H}opf algebras.
\newblock {\em J. Algebra}, 272(1):273--291, 2004.

\bibitem[Fim13]{Fima_HNN}
Pierre Fima.
\newblock {$K$}-amenability of {HNN} extensions of amenable discrete quantum
  groups.
\newblock {\em J. Funct. Anal.}, 265(4):507--519, 2013.

\bibitem[Izu02]{Izumi}
Masaki Izumi.
\newblock Non-commutative {P}oisson boundaries and compact quantum group
  actions.
\newblock {\em Adv. Math.}, 169(1):1--57, 2002.

\bibitem[KKS16]{KalantarKasprzakSkalski_Open}
Mehrdad Kalantar, Pawe\l{} Kasprzak, and Adam Skalski.
\newblock Open quantum subgroups of locally compact quantum groups.
\newblock {\em Adv. Math.}, 303:322--359, 2016.

\bibitem[KKSV]{KKSV_Furstenberg}
Mehrdad Kalantar, Pawe\l{} Kasprzak, Adam Skalski, and Roland Vergnioux.
\newblock Noncommutative {F}urstenberg boundary.
\newblock Analysis \& PDE, to appear, available at arXiv:2002.09657.

\bibitem[KLQ08]{KLQ_Schlichting}
Steven Kaliszewski, Magnus~B. Landstad, and John Quigg.
\newblock Hecke {$C^*$}-algebras, {S}chlichting completions and {M}orita
  equivalence.
\newblock {\em Proc. Edinb. Math. Soc. (2)}, 51(3):657--695, 2008.

\bibitem[KS97]{KlimykSchmuedgen_Book}
Anatoli Klimyk and Konrad Schm\"{u}dgen.
\newblock {\em Quantum groups and their representations}.
\newblock Texts and Monographs in Physics. Springer-Verlag, Berlin, 1997.

\bibitem[KS20]{KasprzakSoltan_Projection}
Pawe\l{} Kasprzak and Piotr~M. So\l{}tan.
\newblock Quantum groups with projection and extensions of locally compact
  quantum groups.
\newblock {\em J. Noncommut. Geom.}, 14(1):105--123, 2020.

\bibitem[Kus03]{Kustermans_UnivAlgebraic}
Johan Kustermans.
\newblock The analytic structure of algebraic quantum groups.
\newblock {\em J. Algebra}, 259(2):415--450, 2003.

\bibitem[KV00]{KustermansVaes}
Johan Kustermans and Stefaan Vaes.
\newblock Locally compact quantum groups.
\newblock {\em Ann. Sci. \'{E}cole Norm. Sup. (4)}, 33(6):837--934, 2000.

\bibitem[KVD97]{KustermansVanDaele_CstarAlgebraic}
Johan Kustermans and Alfons Van~Daele.
\newblock {$C^*$}-algebraic quantum groups arising from algebraic quantum
  groups.
\newblock {\em Internat. J. Math.}, 8(8):1067--1139, 1997.

\bibitem[LVD]{LandstadVanDaele_AlgSubgroups1}
Magnus~B. Landstad and Alfons Van~Daele.
\newblock Compact and discrete subgroups of algebraic quantum groups {I}.
\newblock Preprint, arXiv:math/0702458.

\bibitem[LVD08]{LandstadVanDaele_CompactOpen}
Magnus~B. Landstad and Alfons Van~Daele.
\newblock Groups with compact open subgroups and multiplier {H}opf
  {$*$}-algebras.
\newblock {\em Expo. Math.}, 26(3):197--217, 2008.

\bibitem[MRW12]{MeyerRoyWoronowicz_Homomorphisms}
Ralf Meyer, Sutanu Roy, and Stanis\l aw~L. Woronowicz.
\newblock Homomorphisms of quantum groups.
\newblock {\em M\"{u}nster J. Math.}, 5:1--24, 2012.

\bibitem[NT13]{NeshveyevTuset_Book}
Sergey Neshveyev and Lars Tuset.
\newblock {\em Compact quantum groups and their representation categories},
  volume~20 of {\em Cours Sp\'{e}cialis\'{e}s [Specialized Courses]}.
\newblock Soci\'{e}t\'{e} Math\'{e}matique de France, Paris, 2013.

\bibitem[NY18]{NeshveyevYamashita}
Sergey Neshveyev and Makoto Yamashita.
\newblock Categorically {M}orita equivalent compact quantum groups.
\newblock {\em Doc. Math.}, 23:2165--2216, 2018.

\bibitem[Sch80]{Schlichting}
G{\"u}nter Schlichting.
\newblock Operationen mit periodischen {S}tabilisatoren.
\newblock {\em Arch. Math. (Basel)}, 34(2):97--99, 1980.

\bibitem[Shi71]{Shimura_introduction}
Goro Shimura.
\newblock {\em Introduction to the arithmetic theory of automorphic functions}.
\newblock Publications of the Mathematical Society of Japan, No. 11. Iwanami
  Shoten, Publishers, Tokyo; Princeton University Press, Princeton, N.J., 1971.
\newblock Kan\^{o} Memorial Lectures, No. 1.

\bibitem[Tza03]{Tzanev}
Kroum Tzanev.
\newblock Hecke {$C^*$}-algebras and amenability.
\newblock {\em J. Operator Theory}, 50(1):169--178, 2003.

\bibitem[Vae05]{Vaes_Imprimitivity}
Stefaan Vaes.
\newblock A new approach to induction and imprimitivity results.
\newblock {\em J. Funct. Anal.}, 229(2):317--374, 2005.

\bibitem[VD98]{VanDaele_Algebraic}
Alfons Van~Daele.
\newblock An algebraic framework for group duality.
\newblock {\em Adv. Math.}, 140(2):323--366, 1998.

\bibitem[Ver04]{Vergnioux_Amalg}
Roland Vergnioux.
\newblock {$K$}-amenability for amalgamated free products of amenable discrete
  quantum groups.
\newblock {\em J. Funct. Anal.}, 212(1):206--221, 2004.

\bibitem[VV03]{VaesVainerman}
Stefaan Vaes and Leonid Vainerman.
\newblock On low-dimensional locally compact quantum groups.
\newblock In {\em Locally compact quantum groups and groupoids ({S}trasbourg,
  2002)}, volume~2 of {\em IRMA Lect. Math. Theor. Phys.}, pages 127--187. de
  Gruyter, Berlin, 2003.

\bibitem[VV13]{VergniouxVoigt}
Roland Vergnioux and Christian Voigt.
\newblock The {$K$}-theory of free quantum groups.
\newblock {\em Math. Ann.}, 357(1):355--400, 2013.

\bibitem[VY20]{VoigtYuncken_Book}
Christian Voigt and Robert Yuncken.
\newblock {\em Complex semisimple quantum groups and representation theory},
  volume 2264 of {\em Lecture Notes in Mathematics}.
\newblock Springer, Cham, 2020.

\bibitem[Wor98]{Woronowicz_CQG}
Stanis\l aw~L. Woronowicz.
\newblock Compact quantum groups.
\newblock In {\em Sym\'{e}tries quantiques ({L}es {H}ouches, 1995)}, pages
  845--884. North-Holland, Amsterdam, 1998.

\end{thebibliography}

\end{document}